\documentclass[11pt, twoside]{amsart}
\usepackage{shadethm,framed}
\usepackage[utf8]{inputenc}
\usepackage[english]{babel}
\usepackage[papersize={21cm,29.7cm},left=3.3cm,right=3.3cm,top=3.2cm,bottom=3.2cm]{geometry}
\usepackage{enumerate}
\usepackage{listings}
\usepackage{amsthm, amsfonts, amssymb, amsmath}
\usepackage{graphicx}
\usepackage{xcolor}
\usepackage{color}
\usepackage{framed}
\usepackage{wallpaper}
\usepackage{tikz,tikz-cd}
\usepackage{pgf,tikz,pgfplots}
\usepackage{tcolorbox}
\usepackage{scalefnt}
\usepackage{mdframed}
\usepackage{booktabs}
\usepackage{caption}
\usepackage{makecell}
\usepackage{float}
\usepackage{calrsfs}

\usepgfplotslibrary{patchplots}
\usetikzlibrary{patterns, positioning, arrows,decorations.markings}

\definecolor{blou}{rgb}{0.66, 0.44, 0.98}
\usepackage[hidelinks,urlcolor=blue]{hyperref}
\hypersetup{
    colorlinks,
    linkcolor={blue!50!black},
    citecolor=blou,
    urlcolor={blue!80!black}
}

\usepackage{pgfplots}
\pgfplotsset{width=7cm, compat=1.10}
\usepgfplotslibrary{fillbetween}
\pgfmathdeclarefunction{poly}{0}{\pgfmathparse{sin(x)+x}}

\tikzcdset{arrow style=tikz,diagrams={>={Straight Barb[scale=0.8]},all arrow/.append style = -{Latex[scale=0.8]}}}

\tcbset{
    frame code={}
    center title,
    left=0pt,
    right=0pt,
    top=0pt,
    bottom=0pt,
    colback=red!70,
    colframe=white,
    width=\dimexpr\textwidth\relax,
    enlarge left by=0mm,
    boxsep=5pt,
    arc=0pt,outer arc=0pt,
    }

\newtheoremstyle{pourdef}
  {10pt}
  {10pt}
  {}
  {}
  {\bf}
  {.~}
  { }
  {}
\newtheoremstyle{pourth}{10pt}{10pt}{\em}{}{\sc}{.~}{ }{}
\newtheoremstyle{pourpp}{10pt}{10pt}{\em}{}{\bf \em}{.~}{ }{}
\newtheoremstyle{pourrk}{10pt}{10pt}{}{}{\em}{.~}{ }{}
\newtheoremstyle{pourlm}{10pt}{10pt}{\em}{}{\bf \em}{.~}{ }{}
\newtheoremstyle{pourco}{10pt}{10pt}{\em}{}{\bf \em}{.~}{ }{}

\definecolor{black}{cmyk}{1,1,1,1}
\definecolor{colordef}{rgb}{0.2,0.5,0.07} 
\definecolor{colorprop}{rgb}{0.2,0.1,0.5}

\definecolor{color1}{rgb}{0.6,0.4,0.8}
\definecolor{color1}{rgb}{0.9,0.6,0.4}
\definecolor{color1}{rgb}{0.36, 0.54, 0.66}

\definecolor{color2}{rgb}{0.2,0.1,0.5}
\definecolor{color2}{rgb}{0.91, 0.84, 0.42}
\definecolor{color2}{rgb}{0.87, 0.36, 0.51}
\definecolor{color2}{rgb}{0.4, 0.69, 0.2}
\definecolor{color2}{rgb}{0.84, 0.23, 0.24}

\definecolor{color3}{rgb}{1.0, 0.13, 0.32}
\definecolor{color3}{rgb}{0.54, 0.17, 0.89}

\definecolor{babypink}{rgb}{0.96, 0.76, 0.76}

\definecolor{babyblue}{rgb}{0.76, 0.76, 0.95}

\makeatletter
\renewenvironment{proof}[1][\proofname]{\par
  \normalfont
  \trivlist
  \item[\hskip\labelsep\itshape
    #1.]\ignorespaces
}{%
  \endtrivlist
}
\makeatother

\newtheorem{lem}{Lemma}[section]
\newtheorem{cor}[lem]{Corollary}
\newtheorem{thm}[lem]{Theorem}
\newtheorem{fact}[lem]{Fact}
\newtheorem{prop}[lem]{Proposition}
\theoremstyle{definition}
\newtheorem{definition}[lem]{Definition}
\newtheorem{ex}[lem]{Example}

\newtheorem{rmk}[lem]{Remark}
\newtheorem{question}[lem]{Question}
\numberwithin{equation}{section}

\newcommand{\SB}{\mathbf{Sb}}
\newcommand{\Ein}{\operatorname{Ein}}

\newcommand{\Grassd}{\operatorname{Gr}_p (\mathbb{R}^{2p})}

\newcommand{\GRK}
{\operatorname{Gr}_r(\mathbb{K}^{2r})}
\newcommand{\Lag}{\operatorname{Lag}}
\def\O{\Omega}
\newcommand{\Diams}{\mathbf{D}}
\newcommand{\I}{\mathbf I}
\newcommand{\J}{\mathbf J}
\newcommand{\Fl}{\mathcal{F}}
\newcommand{\Phot}{\Lambda}
\newcommand{\Affstd}{\mathbb{A}}
\newcommand{\MRr}{\mathcal{R}}
\newcommand{\Phoc}{\mathcal{P}_N}
\def\Rr{\operatorname{Extr}_{\MRr}}
\newcommand{\Contlam}{\mathcal{X}}
\newcommand{\hyp}{\operatorname{Z}}

\newcommand{\uu}{\mathfrak{u}}
\newcommand{\g}{\mathfrak{g}}

\newcommand{\aaa}{\mathfrak{a}}
\newcommand{\Mat}{\operatorname{Mat}}
\newcommand{\tr}{\mathsf{t}}
\newcommand{\FS}{\operatorname{\Delta}}
\newcommand{\Kf}{\mathbb{K}}
\newcommand{\Cf}{\mathbb{C}}
\newcommand{\Rf}{\mathbb{R}}
\newcommand{\Hf}{\mathbb{H}}

\newcommand{\isl}{\mathsf{i}}
\newcommand{\FC}{\mathsf{F}}
\newcommand{\dev}{\mathsf{dev}}
\def\b{\mathbf{b}}
\newcommand{\ro}{\mathsf{k}}
\newcommand{\kob}{K}
\newcommand{\pol}{\mathsf{Q}}
\newcommand{\plong}{\tau}
\newcommand{\Ad}{\operatorname{Ad}}
\newcommand{\ad}{\operatorname{ad}}
\newcommand{\leng}{\mathsf{len}_{\O}}
\newcommand{\plongsl}{\operatorname{j}}
\newcommand{\pstd}{\operatorname{pr}_{\mathsf{std}}}
\newcommand{\E}{\operatorname{E}}
\newcommand{\F}{\operatorname{F}}
\newcommand{\He}{\operatorname{H}}
\newcommand{\nG}{\mathsf{n}(G)}
\newcommand{\e}{\varepsilon}
\newcommand{\Iso}{\mathrm{Isom}}

\newcommand{\Aut}{\mathsf{Aut}}
\newcommand{\hol}{\mathsf{hol}}
\newcommand{\HTT}{\operatorname{HTT}}

\newcommand{\SL}{\operatorname{SL}}
\newcommand{\PO}{\operatorname{PO}}
\newcommand{\GL}{\operatorname{GL}}
\newcommand{\PGL}{\operatorname{PGL}}
\newcommand{\SU}{\operatorname{SU}}
\newcommand{\Sp}{\operatorname{Sp}}
\newcommand{\spp}{\mathfrak{sp}}
\newcommand{\SO}{\operatorname{SO}}
\newcommand{\soo}{\mathfrak{so}}
\renewcommand{\qed}{\hfill\square}

\newcommand{\opp}{-}

\makeatletter
\def\@setcopyright{}
\def\serieslogo@{}
\def\@adminfootnotes{}
\makeatother

\makeatletter
\newcommand{\tpitchfork}{%
  \vbox{
    \baselineskip\z@skip
    \lineskip-.52ex
    \lineskiplimit\maxdimen
    \m@th
    \ialign{##\crcr\hidewidth\smash{$-$}\hidewidth\crcr$\pitchfork$\crcr}
  }%
}
\makeatother

\title[Proper almost-homogeneous domains in positive flag manifolds]{Rigidity of proper almost-homogeneous domains in positive flag manifolds}
\author{Blandine Galiay}
\date{}

\begin{document}

\begin{abstract}
   We show that, inside the Shilov boundary of any given Hermitian symmetric space of tube type, there is, up to isomorphism, only one proper domain such that every point on its boundary belongs to the closure of an orbit under its automorphism group. This gives a classification of all closed proper manifolds locally modelled on such Shilov boundaries, and provides a positive answer, in the case of flag manifolds admitting a $\Theta$-positive structure, to a rigidity question of W. van Limbeek and A. Zimmer.
\end{abstract}

\maketitle

\section{Introduction}

The general context of this paper is that of geometric structures. A manifold $M$ admits a \emph{$(G,X)$-structure}, where $G$ is a real Lie group and $X$ is a $G$-homogeneous space, if there exists an atlas of charts on $M$ with values in $X$ whose changes of coordinates are given by elements of $G$. The manifold $M$ endowed with this structure is called a \emph{$(G,X)$-manifold}. Geometric structures have been the object of deep work and fundamental theories since Klein's Erlangen program, including Poincaré--Koebe's uniformization theorem, Teichmüller theory, Thurston's geometrization program and Perelman's theorem; see \cite{goldman2010locally}.

In this paper, we consider more specifically the case where the manifold $M$ is closed, the group $G$ is semisimple and the space $X$ is a \emph{flag manifold} of $G$, i.e.\ a smooth projective variety on which $G$ acts transitively. There exist many examples of such $(G,X)$-manifolds $M$, including quotients of domains of discontinuity of Anosov representations \cite{frances2005lorentzian, guichard2012anosov, kapovich2017dynamics}. However, it is in general difficult to classify all closed $(G,X)$-manifolds. To obtain classification results, one strategy is to assume, for instance, that the manifold $M$ is a quotient $\O / \Gamma$ of a ``not too large" (i.e.\ \emph{proper}) domain $\O$ of $X$ by a discrete subgroup $\Gamma$ of $G$ acting properly discontinuously and freely on $\O$. We say that the manifold $M$ is \emph{proper}.

In the case where $G = \PGL(n, \Rf)$ is the projective linear group and $X = \mathbb{P}(\Rf^{n})$ is the real projective space, this leads to the theory of \emph{divisible convex sets},  whose main results and ideas are outlined in Section~\ref{sect_divisible_convex} below. This rich theory motivates general questions concerning ``small" divisible open sets in arbitrary flag manifolds.

With this in mind, in this paper we investigate closed proper $(G,X)$-manifolds where~$G$ is a Lie group of Hermitian type of tube type and $X$ is the Shilov boundary of the symmetric space of $G$ (in particular $X$ is a flag manifold of $G$); see Table \ref{table_shilov_bnds}. This reduces to the study of proper divisible domains of $X$. Our main result (Theorem~\ref{thm_main}) implies that these domains are actually subject to a strong rigidity, contrasting with the flexibility that can be observed in the real projective case.

\subsection{Divisible convex sets in real projective geometry} \label{sect_divisible_convex}
Convex projective geometry generalizes real hyperbolic geometry. The objects of study are the open sets~$\O \subset \mathbb{P}(\mathbb{R}^{n})$ that are \emph{properly convex}, meaning they are bounded and convex in an affine chart. The group of elements of $\PGL(n, \mathbb{R})$ preserving $\O$ is called the \emph{automorphism group} of $\O$ and is denoted by $\Aut(\O)$.  The set $\O$ is \emph{divisible} if there exists a discrete subgroup of $\Aut(\O)$ that acts cocompactly on $\O$, and \emph{quasi-homogeneous} if~$\Aut(\O)$ acts cocompactly on~$\O$. It is said to be \emph{almost-homogeneous} if its f\emph{ull orbital limit set}, that is, the union of the accumulation points of all orbits of $\Aut(\O)$, is equal to its boundary. Note that divisibility implies quasi-homogeneity, which implies almost-homogeneity.

Properly convex open subsets of the projective space that are divisible are called \emph{divisible convex sets} and admit a rich theory, which was initiated in the 1960's with work of Benzecri \cite{benzecri1960varietes}. This theory has since then been developed by numerous authors (see e.g.\ \cite{vinberg1965structure, goldman1990convex, cooper2015convex}), in particular by Benoist in the early 2000's  \cite{benoist2000automorphismes, benoist2003convexes, benoist2005convexes, benoist2006convexes}. The strictly convex case is well understood, while the  nonstrictly convex case is still in expansion \cite{islam2019rank, choi2020convex, zimmer2020higher, blayac2021boundary}. See \cite{benoist2008survey} for a survey on this theory and more references.

By \cite{vey1970automorphismes, benoist2001convexes}, the theory of divisible convex sets reduces to the case where $\O$ is \emph{irreducible}. In this case, either $\O$ is symmetric, i.e., there exists a symmetry in $\Aut(\O)$ at every point of $\O$, or $\Aut(\O)$ is a discrete Zariski-dense subgroup of $\PGL(n, \mathbb{R})$; see \cite{vinberg1965structure, koecher1999minnesota, benoist2003convexes}.

There is a list of all irreducible symmetric domains in any dimension \cite{koecher1999minnesota}. All domains of this list identify with Riemannian symmetric spaces, and their isometry group and automorphism group (as properly convex open subsets in projective space) coincide. The simplest example is the real hyperbolic space $\mathbb{H}^{n-1}$ embedded in the projective space $\mathbb{P}(\mathbb{R}^n)$ via the Klein model. In this case, we are in the framework of real hyperbolic geometry. 

There exists cocompact lattices $\Gamma$ of $\PO(n-1,1)$ that admit nontrivial deformations into $\PGL(n, \mathbb{R})$ \cite{johnson1987deformation}. By Koszul's openness Theorem~\cite{koszul1968deformations}, the image of any small deformation of~$\Gamma$ in~$\PGL(n, \mathbb{R})$ still divides a divisible convex set, which can be nonsymmetric. There also exist several constructions of nonsymmetric irreducible divisible convex sets; the first examples were constructed by Kac--Vinberg in dimension 2, using Coxeter groups \cite{vinberg1967quasi}. 
Kapovich later constructed examples that are not quasi-isometric to any symmetric space, and whose automorphism group is discrete and Gromov-hyperbolic, in any dimension $n-1\geq 4$ \cite{kapovich2007convex}. Examples with  a discrete and non-Gromov-hyperbolic automorphism group were constructed in projective dimension 3,4,5,6 by Benoist \cite{benoist2006convexes} and 3 by Ballas--Danciger--Lee \cite{ballas2018convex}, and recently, in any projective dimension $n-1\geq 3$ by Blayac--Viaggi \cite{blayac2023divisible}.

Even though it has recently been proved that divisible convex sets in the projective space are subject to a certain rigidity \cite{zimmer2020higher}, the diversity of nonsymmetric examples highlights the importance of general results concerning them, beyond those on cocompact actions on Riemannian symmetric spaces.

Other examples of properly convex sets arise by imposing weaker assumptions on the action of the automorphism group: certain properly convex sets in $\mathbb{P}(\mathbb{R}^n)$ admit actions by discrete groups $\Gamma \leq \PGL(n, \mathbb{R})$ with finite covolume, without being divisible or even quasi-homogeneous. Such examples are almost-homogeneous \cite{ballas2020properly}.

\subsection{Flag manifolds and rigidity}\label{sect_flag_mfd} The real projective space is an example of a \emph{flag manifold}, i.e.\ a compact homogeneous space $G/P$ where $G$ is a noncompact real semisimple linear Lie group and $P$ a parabolic subgroup of $G$. The question of whether the theory of divisible convex sets generalizes to other flag manifolds than the real projective space was asked by W. van Limbeek and A. Zimmer. There are natural notions of convexity, properness, divisibility and quasi-homogeneity for domains (i.e.\ connected open sets) in flag manifolds, generalizing those in the projective space (see Section~\ref{sect_dual_proper_domain} and \cite{zimmer2018proper}). By ``natural", we mean that the following property, already true in the projective space (by \cite{kobayashi1984projectively}), remains true in general flag manifolds: proper quasi-homogeneous domains are convex.

One can first consider a particular class of flag manifolds, those that identify with a symmetric space of compact type, with isometry group a maximal compact subgroup of $G$. These flag manifolds are called \emph{Nagano spaces}, \emph{extrinsinc symmetric spaces} or \emph{symmetric $R$-spaces} depending on the authors, and their list is known \cite{nagano1965transformation}. There is a wealth of literature about these manifolds, starting with
\cite{kobayashi1964filtered, takeuchi1965cell, takeuchi1968minimal, makarevivc1973open}. Nagano has classified symmetric spaces $\mathbb{X}$ of noncompact type that embed into their compact dual, which turn out to be Nagano spaces. Moreover, the image of~$\mathbb{X}$ in its dual is a proper domain. 

Nagano's list provides examples of divisible convex sets in flag manifolds (beyond the projective space), but all these examples are symmetric. The issue is to determine if there exist nonsymmetric examples, as in the projective case. This is a general question, which applies to any flag manifold:

\begin{question}[{\cite{van2019rigidity}}] \label{question_Lim_Zim}
    Given a noncompact real semisimple Lie group $G$, and a parabolic subgroup $P$ of $G$, are all proper divisible domain of $G/P$ symmetric?
\end{question}

One could also ask the same question, replacing ``divisible" with ``almost-homogeneous". A positive answer in the almost-homogeneous case would imply a positive answer to Question \ref{question_Lim_Zim}. 

First note that the study of proper almost-homogeneous domains reduces to the case where $G$ is simple, by the following fact.

\begin{lem}{\cite[Thm 1.7]{zimmer2018proper}}\label{thm_zim_semisimple}
Let $G$ be a semisimple Lie group $G$ with trivial center and no compact factors, and write $G = G_1 \times \cdots \times G_k$ where $G_i$ is a noncompact simple Lie group for all $1 \leq i \leq k$. For any parabolic subgroup $P$ of $G$, there exist parabolic subgroups $P_i \leq G_i$ such that $P = P_1 \times \cdots \times P_k$. Now let $\O \subset G/P$ be  a proper almost-homogeneous domain. Then there exist proper almost-homogeneous domains $\O_i \subset G_i /P_i$ such that $\O = \O_1 \times \cdots \times \O_k$.
\end{lem}
Although Lemma~\ref{thm_zim_semisimple} is proved in \cite[Thm 1.7]{zimmer2018proper} for quasi-homogeneous domains, their proof holds in the quasi-homogeneous case. We give it in Section~\ref{sect_group_aut_omega}.

Question \ref{question_Lim_Zim} admits a negative answer in the case where $G = \SO(n,1)$ for $n \geq 3$, and~$P$ is a minimal parabolic subgroup of $G$. In this case, the flag manifold $G/P$ is the conformal sphere and it admits proper divisible domains that are nonsymmetric (here ``proper" simply means that the complement of the domain has nonempty interior). For instance, the limit set of a representation obtained by a deformation of the natural inclusion of a cocompact lattice in $\SO(2,1)$ into $\SO(3,1)$ (called \emph{quasi-Fuchsian}) is a quasi-circle, and thus separates the two-sphere $G/P$ into two proper, divisible domains. The case where $G$ is the projective linear group $\PGL(n+1, \mathbb{R})$ and $G/P$ is the projective space $\mathbb{P}(\mathbb{R}^{n+1})$ also has a negative answer (see Section~\ref{sect_divisible_convex}). However, in other flag manifolds $G/P$, one observes more rigidity, and Zimmer makes the conjecture that any divisible convex domain of~$G/P$ is homogeneous \cite[Conj.\ 2.6]{zimmer2018proper}. 

Question \ref{question_Lim_Zim} has a complete answer for flag manifolds $G/P$ where $P$ is a nonmaximal proper parabolic subgroup:

\begin{fact}[{\cite[Thm 1.5]{zimmer2018proper}}]\label{th_zim_G/P}
Let $G$ be a noncompact real simple Lie group and~$P \leq G$ a nonmaximal proper parabolic subgroup. Then there are no proper almost-homogeneous domains in $G/P$.
\end{fact}

As for Lemma~\ref{thm_zim_semisimple},  Fact~\ref{th_zim_G/P} is proved in~\cite[Thm 1.5]{zimmer2018proper} for proper quasi-homogeneous domains, but their proof still holds for almost-homogeneous domains.

Flag manifolds defined by maximal proper parabolic subgroups have also been studied: Question \ref{question_Lim_Zim} has a positive answer for self-opposite Grassmannians $\Grassd$ \cite{van2019rigidity} and for the complex projective space  $\mathbb{P}(\mathbb{C}^n)$ for $n \geq 3 $ \cite{zimmer2013rigidity}. In both these cases, the authors show --- under some additional assumption --- that there is only one divisible proper domain in $G/P$ up to the action of $G$, and that this domain is symmetric. 

In simultaneous joint work with A. Chalumeau \cite{chalumeau2024rigidity}, we give a positive answer to Question~\ref{question_Lim_Zim} when $G/P$ is the boundary of the pseudo-hyperbolic space $\mathbb{H}^{p,q}$ for $p \geq 2$ and $q \geq 1$ --- also called \emph{the Einstein universe of signature $(p-1, q)$}. The special case where $q=1$ admits a different proof than the general case, using \emph{causality} in the Einstein universe of signature $(n-1,1)$. In Lemma~\ref{lem_existence_extremal} of the present paper, we generalize this causality approach to \emph{Shilov boundaries} of Hermitian symmetric spaces of tube type (see the next section), in order to prove our main Theorem~\ref{thm_main} below. These flag manifolds appear naturally in several contexts, such as Euclidean Jordan algebras and complex analysis (see e.g.\ \cite{faraut1994analysis}), or $\Theta$-positivity and higher Teichmüller Theory \cite{guichard2018positivity}. Here the proper parabolic subgroup $P$ is maximal, so Fact~\ref{th_zim_G/P} does not apply. For any two transverse points $p,q \in G/P$, the open set consisting of all points of $G/P$ which are transverse to both $p$ and $q$ has several connected components, two of which are proper symmetric domains called \emph{diamonds} (see Definition~\ref{def_diamant}). The main result of the present paper (Theorem~\ref{thm_main} below) is a positive answer to Question~\ref{question_Lim_Zim}: we prove that in our context of Shilov boundaries there is, up to the action of $G$, only one proper almost-homogeneous domain (hence only one proper divisible domain) in $G/P$, namely a diamond; in particular, it is symmetric.

Removing the properness assumption in Question \ref{question_Lim_Zim} adds flexibility. Indeed, as mentioned above, there exist many examples of nonproper, nonsymmetric divisible domains of flag manifolds $G/P$, constructed e.g.\ as domains of discontinuity for Anosov representations \cite{frances2005lorentzian, guichard2012anosov, kapovich2017dynamics}.

\subsection{Statement of the main theorem. } Let $G$ be a simple Lie group of Hermitian type of tube type and of real rank $r \geq 1$ (the complete list of the corresponding Lie algebras is given in Table \ref{table_shilov_bnds}). Let $\FS$ be set of simple restricted roots of $G$ and $\alpha_r \in \FS$ be the unique long root. Then the flag manifold $\Fl_{\{\alpha_r\}} = G/P_{\{\alpha_r\}}$ defined by $\{\alpha_r\}$ (see Section~\ref{sect_recalls_lie_groups}) is the \emph{Shilov boundary} of the symmetric space $\mathbb{X}_G$ of $G$, denoted by~$\SB(G)$. Recall that the \emph{automorphism group} $\Aut(\O)$ of a domain $\O \subset \SB(G)$ is the group of all elements of $G$ that leave $\O$ invariant; we say that $\O$ is \emph{almost-homogeneous} if every element of $\partial \O$ lies in the closure of the orbit of an element of $\O$ under $\Aut(\O)$.

\begin{thm}\label{thm_main} Let $G$ be a simple Lie group of Hermitian type of tube type. Then every proper almost-homogeneous domain of $\SB(G)$ is a diamond.
\end{thm}

Since every divisible proper domain of $\SB(G)$ is almost homogeneous, Theorem~\ref{thm_main} gives a positive answer to Question~\ref{question_Lim_Zim}.
It also implies that, reciprocally, every almost-homogeneous domain of $\SB(G)$ is divisible. As already mentionned in Section~\ref{sect_divisible_convex}, this property does not hold in the projective space, as there exist properly convex domains of the projective space that are almost-homogeneous and not divisible, and not even quasi-homogeneous, see \cite{ballas2018convex}.

As in \cite[Sect.\ 6.3]{chalumeau2024rigidity}, the proof of Theorem~\ref{thm_main} uses the causal structure of $\SB(G)$, i.e.\ the existence (up to taking an index-two subgroup of $G$) of a $G$-equivariant smooth family $(c_x)_{x \in \SB(G)}$ of properly convex open cones in the tangent bundle $T(\SB(G))$. Any proper domain~$\O \subset \SB(G)$ inherits a causal structure from the one of $\SB(G)$. A generalization of Liouville's classical theorem implies that $\Aut(\O)$ is commensurable to the \emph{conformal group} of $\O$, that is, the group of diffeomorphisms $f:\O \rightarrow \O$ such that~$d_x f (c_x) = c_{f(x)}$ for all $x \in \O$ \cite{kaneyuki2011automorphism}. Theorem~\ref{thm_main} states that having a cocompact conformal group characterizes diamonds among proper domains of $\SB(G)$.

More generally, if $G$ is a semisimple Lie group, Hermitian of tube type, with trivial center and no compact factors, then one can write $G = G_1 \times \cdots \times G_k$, where~$G_i$ is a noncompact simple Lie group, Hermitian of tube type, for all $1 \leq i \leq k$. Then the \emph{Shilov boundary} $\SB(G)$ of $G$ is the flag manifold $\SB(G_1) \times \cdots \times \SB(G_k)$. By Lemma~\ref{thm_zim_semisimple} and Theorem~\ref{thm_main}, one directly gets:
\begin{cor}\label{cor_main_cor} Let $G$ be a  semisimple Lie group of Hermitian type of tube type with trivial center and no compact factors, and write $G = G_1 \times \cdots \times G_k$, where $G_i$ is a noncompact simple Lie group of Hermitian type of tube type for all $1 \leq i \leq k$. Then for any proper almost-homogeneous domain $\O \subset \SB(G)$, there are diamonds~$D_i \subset \SB(G_i)$ for $1 \leq i \leq k$, such that 
\begin{equation*}
    \O = D_1 \times \cdots \times D_k \subset \SB(G_1) \times \cdots \times \SB(G_k).
\end{equation*}
\end{cor}

In particular, Question \ref{question_Lim_Zim} has a positive answer for Shilov boundaries of Hermitian symmetric spaces of tube type.

Note that the flag manifolds studied in this paper are in the list of Nagano spaces, and a diamond of $\SB(G)$ is a model for the noncompact dual of the compact symmetric space $\SB(G)$.

\subsection{$\Theta$-positive structures} \label{sect_pos_struct} Total positivity has been known and studied since the beginning of the 20th century for $\SL(N, \mathbb{R})$. It was generalized to real split semisimple Lie groups by Lusztig \cite{lusztig1994total}. On the other hand, isometry groups of irreducible Hermitian symmetric spaces of tube type were known to admit a causal structure. 

Guichard--Wienhard generalized these two notions of total positivity and causality with their notion of \emph{$\Theta$-positive structure}, where $\Theta$ is a subset of the simple roots of a semisimple Lie group $G$. They listed all pairs $(G, \Theta)$ such that $G$ admits a $\Theta$-positive structure \cite{guichard2018positivity, guichard2022generalizing}. The pairs $(G, \{ \alpha_r \})$, where $G$ is a simple Lie group of Hermitian type of tube type and $\alpha_r$ is the simple root of $G$ defining the Shilov boundary $\SB(G)$, constitute the only family of their list where the set $\Theta$ is a singleton (i.e.\ where the proper parabolic subgroup defined by $\Theta$ is maximal). Hence Theorem~\ref{thm_main} and Fact~\ref{th_zim_G/P} complete the classification of proper almost-homogeneous domains in flag manifolds admitting a $\Theta$-positive structure:

\begin{cor} 
Let $G$ be a real noncompact simple Lie group and $\Theta$ a subset of the simple restricted roots of $G$ such that $G$ admits a $\Theta$-positive structure. Then the following dichotomy holds:
    
    \begin{enumerate}
        \item If $|\Theta | = 1$, then $G$ is Hermitian of tube type and $\Fl_{\Theta} = \SB(G)$ admits exactly one proper almost-homogeneous domain (up to the action of $G$), which is a diamond.
        \item If $|\Theta| \geq 2$, then there does not exist any proper almost-homogeneous domain in $\Fl_{\Theta}$.
    \end{enumerate}
\end{cor}
Here we denote by $\Fl_{\Theta}$ the flag manifold defined by $(G, \Theta)$ (see Section~\ref{sect_flag_mfds}). Again, by Lemma~\ref{thm_zim_semisimple}, Question \ref{question_Lim_Zim} has a positive answer for flag manifolds $G/P_{\Theta}$ with a $\Theta$-positive structure, where $G$ is a noncompact semisimple (not necessarily simple) Lie group and $\Theta$ a subset of the simple roots of $G$.

\subsection{An application to $(G, G/P)$-structures. } In the real projective space $\mathbb{P}(\Rf^n)$, a nonsymmetric properly convex open set $\O$ divided by a discrete subgroup $\Gamma$ of~$\PGL(n, \mathbb{R})$ provides a closed  $(\PGL(n, \mathbb{R}), \mathbb{P}(\Rf^n))$-manifold $\O/ \Gamma$ that is not isomorphic to a compact quotient of the symmetric space of a noncompact Lie group. 

The case of Shilov boundaries of irreducible Hermitian symmetric spaces of tube type is different. Any $(G, \SB(G))$-manifold $M$ develops into an open subset of $\SB(G)$; see Section~\ref{section_cpt_manifolds_proper_dvt}. The manifold $M$ is \emph{proper} if the image of its developing map is proper in $\SB(G)$. We get the following corollary of Theorem~\ref{thm_main}, whose proof is given in Section~\ref{section_cpt_manifolds_proper_dvt}:

\begin{cor}
    \label{main_cor}
    Let $G$ be a simple Lie group of Hermitian type of tube type, and let $M$ be a closed connected $(G, \SB(G))$-manifold. Assume moreover that $M$ is proper. Then the manifold $M$ identifies as a $(G, \SB(G))$-manifold to a quotient $D/\Gamma$, where $D$ is a diamond of $\SB(G)$ and $\Gamma$ is a cocompact lattice of $\Aut(D)$. Hence the manifold $M$ is a finite cover of
    $$\mathbb{S}^1 \times (X_{L_{s}}/\Gamma'),$$
where $X_{L_{s}}$ is the symmetric space of the semisimple part $L_{s}$ of a Levi subgroup $L$ of~$P_{\{\alpha_r \}}$, and $\Gamma'$ is a cocompact lattice of $L_{s}$.
\end{cor}

In particular, with the notation of Corollary~\ref{main_cor}, the developing map of $M$ is a diffeomorphism between the universal cover $\widetilde{M}$ of $M$ and $D$; we say that the manifold $M$ is \emph{Kleinian}.

\subsection{Comparison of two invariant metrics} \label{sect_comparison_kobayashi_caratheodory} In this paragraph we discuss an important step of the proof of Theorem~\ref{thm_main} (see Proposition~\ref{lem_key_lemma1} below).

In the classical theory of divisible convex sets, the Hilbert metric is a powerful tool to understand the automorphism group of a properly convex domain $\O$, due to its~$\Aut(\O)$-invariance. For a proper almost-homogeneous domain $\O$, the Hilbert metric allows one to relate the geometry of $\partial \O$ to the dynamics of $\Aut(\O)$:

\begin{fact}[{\cite[Lem.\ 4]{vey1970automorphismes}}]\label{lem_vey}
    Let $\O$ be a properly convex domain of $\mathbb{P}(\mathbb{R}^n)$. Assume that there exists a subgroup $H \leq \Aut(\O) $ such that $H$ acts cocompactly on $\O$. Then for any extremal point $p \in \partial \O$, there exists a sequence $(h_n) \in H^{\mathbb{N}}$ such that for any compact subset $\mathcal{K} \subset \O$, the Hausdorff limit of $(h_n \cdot \mathcal{K})$ is $\{ p \}$.  
\end{fact}

For every $z \in \mathbb{P}(\mathbb{R}^n)$ we choose a lift $\widetilde{z}$ of $z $ in $\Rf^n$. Given two distinct points~$x,y$ of~$\mathbb{P}(\mathbb{R}^n)$, we denote by $\ell_{x,y}$ the unique projective line through $x$ and $y$. Finally, given four points $a,b,c,d \in \mathbb{P}(\mathbb{R}^n)$, we denote by $(a:x:y:b)$ their \emph{cross ratio}. The Hilbert metric on a properly convex domain $\O \subset \mathbb{P}(\mathbb{R}^n)$ can be defined in two ways: for any~$x,y \in \O$, the distance $H_{\O}(x,y)$ is equal to both of the two quantities

\begin{enumerate}
    \item $\inf \left \{ \log (a:x:y:b) \mid a,b \in\O \cap \ell_{x,y}, \ a,x,y,b \text{ aligned in this order} \right\} $;
    \item $\sup_{\eta, \xi \in \O^*} \log \left| \frac{\xi(\widetilde{x}) \eta(\widetilde{y})}{\xi(\widetilde{y})\eta (\widetilde{x})} \right|$, where $\O^* = \mathbb{P}\left(\{f \in \mathbb{R}^n \mid f(\widetilde{z}) \ne 0 \quad  \forall z \in \overline{\O}\}\right)$.
\end{enumerate}

One idea in this paper is to adapt some arguments and results from convex projective geometry, in particular Fact~\ref{lem_vey}, to the context of a flag manifold $\Fl_{\Theta}$ different from the real projective space. This requires constructing an $\Aut(\O)$-invariant metric on a proper domain $\O \subset \Fl_{\Theta}$ that gives enough information about the structure of $\partial \O$. Definition (2) above of the Hilbert metric already admits generalizations to arbitrary flag manifolds $\Fl_{\Theta}$ \cite{zimmer2018proper}, called \emph{Caratheodory metrics} and defined from irreducible $\Theta$-proximal representations of $G$; see Section~\ref{sect_def_cara}. These metrics turn out to be proper (i.e.\ closed balls are compact) as soon as $\O$ is \emph{dually convex} (see Section~\ref{sect_dual_proper_domain} and the proof of \cite[Thm 9.1]{zimmer2018proper}). However, in our case, these metrics do not provide enough information about $\partial \O$ to prove Theorem~\ref{thm_main}. Instead, we use the \emph{Kobayashi metric}, which generalizes Definition (1) above and is denoted by $\kob_{\O}$ in this paper. Contrary to those of Caratheodory metrics, the definition of the Kobayashi metric requires more structure on $\Fl_{\Theta}$: namely, the existence of \emph{photons} satisfying some conditions of invariance (Lemmas \ref{lem_parametrization} and \ref{lem_photons_affine}) and of abundance (Lemma~\ref{lem_paths_diamonds}). These properties are satisfied for Shilov boundaries of irreducible Hermitian symmetric spaces of tube type, but not in arbitrary flag manifolds. \\
\indent Although we will not directly use Caratheodory metrics in the proof of Theorem~\ref{thm_main}, we will indirectly use them to deduce properties of the Kobayashi metric: with the notation of Section~\ref{sect_irred_semi}, we get the following proposition, whose proof is contained in Proposition~\ref{lem_key_lemma} and Corollary~\ref{cor_kobayashi_geodesic}:

\begin{prop}\label{lem_key_lemma1} Let $G$ be a simple Lie group of Hermitian type of tube type and let~$(V, \rho)$ be a finite-dimensional real irreducible linear representation of $G$ with highest weight $\chi = N \omega_{\alpha_r}$, where $N \in \mathbb{N}_{>0}$ and $\omega_{\alpha_r}$ is the fundamental weight associated with the last simple root $\alpha_r$. Let $\O \subset \SB(G)$ be a proper dually convex domain, and let $C_{\O}$ be the Caratheodory metric on $\O$ induced by $(V, \rho)$. Then one has
$$\kob_{\O} \geq \frac{1}{N}C_{\O}.$$
\noindent In particular, the metric $\kob_{\O}$ is proper.
\end{prop}

\begin{rmk}\label{rmk_comp_kob_cara}
    A Kobayashi metric can also be defined on proper domains of other flag manifolds, such as the Grassmannians \cite{van2019rigidity} and $\partial \mathbb{H}^{p,q}$ \cite{chalumeau2024rigidity}. In this case, Proposition~\ref{lem_key_lemma1} still holds with a similar proof.
\end{rmk}

\subsection{Outline of the proof of Theorem~\ref{thm_main} and organization of the paper} We now describe the various sections of the paper in the order in which they appear, emphasizing on the key steps of the proof of Theorem~\ref{thm_main}. We denote by $G$ a simple Lie group of Hermitian type of tube type and of real rank $r \geq 1$, and by $\SB(G)$ the Shilov boundary of the symmetric space $\mathbb{X}_G$ of $G$ (see Section~\ref{shilov_bnds}).

\subsubsection{Reminders on Lie groups and Shilov boundaries} In Section~\ref{sect_prelim}, we recall some facts on real Lie groups, flag manifolds, and $\Theta$-proximal representations. In Section~\ref{shilov_bnds}, we recall and prove basic properties of Shilov boundaries of Hermitian symmetric spaces of tube type, and fix some notation for the paper.

\subsubsection{Diamonds} In Section~\ref{sect_proper_autom}, following \cite{zimmer2018proper}, we define the notions of properness and of automorphism groups for domains of $\SB(G)$, in analogy with convex projective geometry. We introduce the previously mentioned \emph{diamonds} (see Section~\ref{ssect_def_diamonds}). If~$p$ lies in the standard affine chart $\Affstd$ of $\SB(G)$ (with the notation of Section~\ref{sect_flag_mfd}), then causality allows us to define the \emph{future} and the \emph{past} of $p$ (see Section~\ref{sect_causalité}). For any point $q$ in the future of $p$, the intersection of the past of $q$ with the future of $p$ is a diamond, denoted by $\Diams(p,q)$; see Figure \ref{fig:enter-label}.

\subsubsection{Photons. }\label{sect_rank_1} In Section~\ref{sect_photons_def}, we define a \emph{photon-generating} action as an action of~$\SL_2(\mathbb{R})$ on $\SB(G)$ which is conjugate to the one induced by the $\mathfrak{sl}_2$-triple associated with the last simple root of $\g$. A \emph{photon} is then the unique closed orbit in $\SB(G)$ of a photon-generating action. It is a topological circle whose properties are similar to those of projective lines in $\mathbb{P}(\Rf^n)$ (Lemmas~\ref{lem_parametrization} and~\ref{lem_photons_affine}) and allow us to define and analyze the Kobayashi metric (see Section~\ref{sect_dist_kob}), and to understand the boundary of a proper domain in $\SB(G)$ (see Section~\ref{sect_dyn_bndr}). 

In Section~\ref{sect_inter_poly} we introduce the \emph{intersection polynomials}, to describe the intersection of a photon with a \emph{maximal proper Schubert subvariety} of $\SB(G)$, that is, a subset of~$\SB(G)$ of the form 
\begin{equation}\label{eq_hyp_Zz}
    \hyp_z = \{ x \in \SB(G) \mid x \text{ is not transverse to } z \},
\end{equation}
where $z \in \SB(G)$ (see Section~\ref{sect_flag_mfds}). The complex roots of intersection polynomials are analyzed in Section~\ref{sect_comp_cara_kob}, see Section~\ref{sect_intro_kob} below.

\subsubsection{The Kobayashi metric}\label{sect_intro_kob} In Section~\ref{sect_dist_kob}, we define the Kobayashi pseudo-metric on a domain $\O \subset \SB(G)$, as mentioned above in Section~\ref{sect_comparison_kobayashi_caratheodory}, and investigate its basic properties (see Proposition~\ref{prop_properties_kobayashi_metric}). This pseudo-metric is a metric generating the standard topology whenever $\O$ is a proper domain of $\SB(G)$ (Proposition~\ref{prop_generate_topo_std}).

A substantial part of Section~\ref{sect_dist_kob} (more precisely, Sections~\ref{sect_comp_cara_kob} and~\ref{sect_prop_K}) is devoted to proving Proposition~\ref{lem_key_lemma1}. The complex roots of an intersection polynomial describe the intersection of complexifications of a photon and of a generic maximal proper Schubert subvariety of $\SB(G)$; since this intersection is a singleton (Lemma~\ref{lem_inter_singleton}), the intersection polynomials are split (Corollary~\ref{cor_inter_split}). This observation allows us to compare two cross ratios, namely the one appearing in the definition of the Kobayashi metric and the other in that of the Caratheodory metric; see Lemma~\ref{lem_comp_crossratios}, from which follows Proposition~\ref{lem_key_lemma1}.

\subsubsection{A diamond containing $\O$. }\label{sect_diamond_containing} Let $\O \subset \SB(G)$ be a proper domain. In Section~\ref{sect_dyn_bndr}, we define the \emph{$\MRr$-extremal points} (see Definition~\ref{def_pts_extremes}) of $\O$, which are analoguous to extremal points of a properly convex subset of the projective space. The definition of $\MRr$-extremal points and that of the Kobayashi metric both involve photons. Following the strategy of \cite{van2019rigidity}, this allows us to prove an analogue of Fact~\ref{lem_vey}; see Lemma~\ref{lem_analogue_vey}, whose main consequence is a strong geometric property for $\MRr$-extremal points of an almost-homogeneous domain (see Theorem~\ref{lem_geom_prop_visuel_2}). 

By almost-homogeneity, the domain $\O$ is dually convex (see \cite[Thm 9.1]{zimmer2018proper} and Proposition~\ref{prop_zimmer_dual_convex}), that is, for any $p \in \partial \O$ there is a supporting hypersurface to $\O$ at $p$ of the form $\hyp_z$ (see~\eqref{eq_hyp_Zz}); see Section~\ref{sect_dual_proper_domain}. Lemma~\ref{lem_geom_prop_visuel_2} expresses a stronger geometric property for an extremal point $p \in \partial \O$: this supporting hypersurface $\hyp_z$ can be taken to be $\hyp_p$ itself. This lemma, applied to two \emph{strongly $\mathcal{R}$-extremal} points $p_0, q_0 \in \partial \O$ (see Definition~\ref{def_pts_extremes} and Lemma~\ref{lem_existence_extremal}), implies that $\O$ is contained in the diamond $\Diams(p_0, q_0)$ (Section~\ref{sect_end_mainth}). 

\subsubsection{Proving the equality. } In Section~\ref{sect_end_mainth}, with the notation of Section~\ref{sect_diamond_containing} above, we prove that $\O = \Diams(p_0, q_0)$. The key point is the inclusion $\Aut(\O) \leq \Aut(\Diams(p_0, q_0))$, which implies, by almost-homogeneity, that $\O$ is closed in (and hence equal to) the diamond $\Diams(p_0, q_0)$.
This inclusion holds because any automorphism of $\O$ preserves the pair $\{p_0, q_0\}$, which is proved in Proposition~\ref{lem_action_strongly_extremal_pts}, and essentially characterizes $\Diams(p_0, q_0)$ (Fact~\ref{fact_autom_diam}).

\subsubsection{An application to geometric structures} In Section~\ref{section_cpt_manifolds_proper_dvt}, we prove Corollary~\ref{main_cor}, which is a consequence of Theorem~\ref{thm_main}.

\subsection*{Aknowledgements} I thank my PhD advisor Fanny Kassel for exposing this interesting topic to me and for her guidance, both in the proof and in the writing of this paper. I also thank Sami Douba for his valuable advice on the writing of this paper, and Jeff Danciger, Wouter van Limbeek and Andy Zimmer for interesting discussions on topics related to this work. I am grateful to Yves Benoist and François Labourie, the former for suggesting that I examine Question~\ref{question_Lim_Zim} for the $3$-dimensional Einstein universe, which, when generalized, led to~\cite{chalumeau2024rigidity} and the current paper, and the latter for sharing his insights on photons in positive flag manifolds. Finally, I warmly thank Adam Chalumeau for inspiring discussions in the context of our joint project~\cite{chalumeau2024rigidity}.

\section{Notation and basic reminders in Lie theory}\label{sect_prelim}

\subsection{Real and complex projective spaces, and cross ratio} \label{sect_cross_ratio} In this section we set some notation for the elements of real and complex projective spaces, and we recall the definition of the cross ratio on $\mathbb{P}(\mathbb{R}^2)$.

Given a finite-dimensional real vector space $V$, we will denote by $[v]$ the projection in $\mathbb{P}(V)$ of a vector $v \in V \smallsetminus \{ 0 \}$. In the case where $V = \mathbb{R}^2$, we denote by $[t_1: t_2]$ the projection in $\mathbb{P}(\mathbb{R}^2)$ of a vector $(t_1, t_2) \in \mathbb{R}^2 \smallsetminus \{0\}$. 

We denote by $(\cdot: \cdot: \cdot : \cdot)$ the classical cross ratio on $\mathbb{P}(\mathbb{R}^2)$. Recall that it is~$\SL_2(\mathbb{R})$-invariant and satisfies $([1:0] : [1:1] : [1:t] : [0:1]) = t$.

If $I \subset \mathbb{P}(\mathbb{R}^2)$ is a proper open interval with (possibly equal) endpoints $t_1$ and $t_2$, then the \emph{Hilbert pseudo-metric} on $I$ is denoted $H_I$ and defined as follows: for any pair~$s_1, s_2 \in I$ such that $t_1, s_1, s_2, t_2$ are aligned in this order (taking any order if~$s_1 = s_2$ or~$t_1 = t_2$), one has $H_I(s_1,s_2) := \log(t_1: s_1: s_2 : t_2)$. If $I = \mathbb{P}(\mathbb{R}^2)$, then $H_I$ is by convention the constant map equal to $0$ on $I^2$.

In Section \ref{sect_comp_cara_kob}, we will also be led to consider the complex projective spaces. Since we will consider both real vector and their complexification, to avoid the confusion we denote by $\mathbb{P}_c(W)$ the complex projective space of a finite-dimensional complex vector space $W$, and by $[v]_c$ the projection in $\mathbb{P}_c(W)$ of a vector $v \in W \smallsetminus \{ 0 \}$. We use the notation $[z_1: z_2]_c$ for the projection in $\mathbb{P}(\mathbb{C}^2)$ of a vector $(z_1, z_2) \in \mathbb{C}^2 \smallsetminus \{0\}$.

\subsection{Reminders on Lie groups Lie algebras}\label{sect_recalls_lie_groups}

In this section we recall some well-known facts about semisimple Lie groups and fix notations that will hold for the rest of the paper. We fix a real semisimple Lie gorup $G$ with Lie algebra $\g$.

\subsubsection{$\mathfrak{sl}_2$-triples} \label{sect_sl2-triples} \mbox{ } A triple $\tr = (e,h,f)$ of elements of $\g$ satisfying the equalities \mbox{  } $[h,e] = 2e$, $[h,f] = -2f$ and $[e,f] = h$ is called an \emph{$\mathfrak{sl}_2$-triple}. There is a Lie algebras embedding $\plongsl_\tr: \mathfrak{sl}_2(\mathbb{R}) \hookrightarrow \g$ such that $\plongsl_\tr(\E) = e$, $\plongsl_\tr (\He) = h$ and $\plongsl_\tr(\F) = f$, where
\begin{equation*}
    \E = \begin{pmatrix}
        0 & 1 \\ 0 & 0
    \end{pmatrix}; \quad \He = \begin{pmatrix}
        1 & 0 \\ 0 & -1
    \end{pmatrix}; \quad \F = \begin{pmatrix}
        0 & 0 \\ -1 & 0
    \end{pmatrix}.
\end{equation*}

\subsubsection{Cartan decomposition}  Let $B$ be the Killing form on $\g$. Let $K \leq G$ be a maximal compact subgroup and $\mathfrak{p}$ be the $B$-orthogonal of the Lie algebra $\mathfrak{k}$ of $K$ in $\g$. Then one has $\g = \mathfrak{k} \oplus \mathfrak{p}$. The \emph{Cartan involution }of $\g$ (with respect to $K$) is then the Lie algebra automorphism $ \sigma: \g \rightarrow \g$ defined by $ \sigma_{|\mathfrak{k}} = \operatorname{id}_{\mathfrak{k}}$ and $ \sigma_{|\mathfrak{p}} = -\operatorname{id}_{\mathfrak{p}}$. 

\subsubsection{Restricted root system.} \label{sect_real_lie_alg} Let $\aaa \subset \mathfrak{p}$ be a maximal abelian subspace, and $\g_0$ the centralizer of $\aaa$ in $\g$. We denote by $\aaa^*$ the space of all linear forms on $\aaa$. For $\alpha \in \aaa^*$, we define
\begin{equation*}
    \g_{\alpha} := \{ X \in \g \mid [H, X] = \alpha(H)X \quad \forall H \in \aaa\}.
\end{equation*}
One has $[\g_{\alpha}, \g_{\beta}] \subset \g_{\alpha + \beta}$ for any $\alpha, \beta \in \aaa^*$. If $\alpha \in \aaa^* \smallsetminus \{0\}$ satisfies $\g_{\alpha} \ne \{0 \}$, then we say that $\alpha$ is \emph{a root} of $(\g, \aaa)$. We denote by $\Sigma = \Sigma(\g, \aaa)$ the set of all the roots of $(\g, \aaa)$. One has $\g = \g_0 \oplus \bigoplus_{\alpha \in \Sigma} \g_{\alpha}$. We fix a \emph{fundamental system} $\FS = \{ \alpha_1, \cdots , \alpha_N\} \subset \Sigma$, i.e.\ a family of roots such that any root of $\g$ can be uniquely written as $\alpha = \sum_{i=1}^N n_i \alpha_i$, where the $n_i$ all have same sign for $1 \leq i \leq N$. The elements of $\FS$ are called \emph{simple roots}. The choice of a fundamental system determines a set of \emph{positive roots} $\Sigma^+$, i.e.\ those roots $\alpha$ where the $n_i$ are all nonnegative.

For any $\alpha \in \Sigma$ and $X \in \g_{\alpha}\smallsetminus \{0\}$, there exists a unique scalar multiple $X'$ of $X$ such that $(X',  \sigma(X'), X'],  \sigma(-X'))$ is an $\mathfrak{sl}_2$-triple. The element $[ \sigma(X'), X']$ does not depend on the choice of $X \in \g_{\alpha}$, and is denoted by $h_{\alpha}$. The family $(h_{\alpha})_{\alpha \in \FS}$, forms a basis of $\aaa$, whose dual basis in $\aaa^*$ is denoted by $(\omega_{\alpha})_{\alpha \in \FS}$.

\subsubsection{The restricted Weyl group} The \emph{restricted Weyl group} $W$ of $G$ is the quotient $N_K(\aaa)/Z_K(\aaa)$ of the normalizer of $\aaa$ in $K$ (for the adjoint action) by the centralizer of $\aaa$ in $K$. It is a finite group generated by the $B$-orthogonal reflexions in $\aaa$ with respect to the kernels of the simple roots. By duality with respect to $B$ (which induces a scalar product on $\aaa$), the action of $W$ on $\aaa$ induces an action on $\aaa^*$ preserving $\Sigma$. There exists a unique $w_0 \in W$, called the \emph{longest element}, such that $w_0 \cdot \Sigma^+ = -\Sigma^+$. The element $i: \aaa^* \rightarrow \aaa^*$ defined as $i = - w_0$ is called the \emph{opposition involution}, and satisfies~$i(\FS) = \FS$.

\subsubsection{Parabolic subgroups. } \label{sect_parab_subgroups} Let $\Theta \subset \FS$ be a subset of the simple roots. \emph{The standard parabolic subgroup} $P_{\Theta}$ (resp.\ the \emph{standard opposite parabolic subgroup} $P_{\Theta}^{\opp}$) is defined as the normalizer in $G$ of the Lie algebra
\begin{equation}\label{eq_lie_u}
    \uu_{\Theta}^+ := \bigoplus_{\alpha \in \Sigma_{\Theta}^+} \g_{\alpha} \quad \Bigg{(}\text{resp.\ }  \uu_{\Theta}^- := \bigoplus_{\alpha \in \Sigma_{\Theta}^+} \g_{-\alpha} \Bigg{)},
\end{equation}
where $\Sigma_{\Theta}^+ := \Sigma^+ \smallsetminus \operatorname{Span}(\FS \smallsetminus \Theta)$. By “standard”, we mean with respect to the above choices. One has
\begin{equation}\label{eq_lie_P}
    \operatorname{Lie}(P_{\Theta}) = \g_0 \oplus \bigoplus_{\alpha \in \Sigma^+} \g_\alpha \oplus \bigoplus_{\alpha \in \Sigma \smallsetminus\Sigma_{\Theta}^+} \g_{- \alpha}.
\end{equation}

More generally, a \emph{parabolic subgroup of type $\Theta$} of $G$ is a conjugate of $P_{\Theta}$ in $G$. 

The \emph{Levi subgroup} associated with $\Theta$ is the reductive Lie group defined as the intersection $L_{\Theta} := P_{\Theta} \cap P_{\Theta}^{\opp}$. The unipotent radical of $P_{\Theta}$ (resp.\ $P_{\Theta}^{\opp}$) is $U_{\Theta}^+ := \exp (\uu_{\Theta}^+)$ (resp.\ $U_{\Theta}^- := \exp (\uu_{\Theta}^-)$). One then has $P_{\Theta} = U_{\Theta}^+ \rtimes L_{\Theta}$ (resp.\ $P_{\Theta}^{\opp} = U_{\Theta}^- \rtimes L_{\Theta}$).

\subsection{Flag manifolds}\label{sect_flag_mfds} Let $G$ be a real semisimple Lie group. The \emph{flag manifold} associated with $\Theta$ is the quotient space $\Fl_{\Theta} := G/P_{\Theta}$. The flag manifold \emph{opposite} to $\Fl_{\Theta}$ is $\Fl_{\Theta}^{\opp} := G/P_{\Theta}^{\opp}$.

\subsubsection{Transversality} The action of $G$ on $\Fl_{\Theta} \times \Fl_{\Theta}^{\opp}$ by left translations has exactly one open orbit $\mathcal{O}$, which is the orbit of $(P_{\Theta}, P^{\opp}_{\Theta})$ and is dense. Two elements $x \in \Fl_{\Theta}$ and~$y \in \Fl_{\Theta}^{\opp}$ are said to be \emph{transverse} if $(x,y) \in \mathcal{O}$.

Given a point $y \in \Fl_{\Theta}^{\opp}$ (resp.\ $x \in \Fl_{\Theta}$), we let $\hyp_{y}$ (resp.\ $\hyp_{x}$) be the set of all elements of $\Fl_{\Theta}$ (resp.\ $\Fl_{\Theta}^\opp$) that are not transverse to $y$ (resp.\ to $x$). It defines an algebraic hypersurface of $\Fl_{\Theta}$ (resp.\ of $\Fl_{\Theta}^\opp$). The space $\Fl_{\Theta} \smallsetminus \hyp_y$ is called an \emph{affine chart} (or more classically a \emph{big Schubert cell}) and is an open dense subset of $\Fl_{\Theta}$. The chart
\begin{equation}\label{eq_standard_chart}
    \Affstd := \Fl_{\Theta} \smallsetminus \hyp_{P_{\Theta}^{\opp}}
\end{equation}
is called the \emph{standard affine chart}. The bijection
\begin{equation}\label{eq_A_egal_exp}
   \varphi_{\mathsf{std}}: \begin{cases}
        \uu_{\Theta}^- &\overset{\sim}{\longrightarrow} \Affstd \\
        X &\longmapsto \exp(X) P_{\Theta}
    \end{cases}
\end{equation}
induces an affine structure on $\Affstd$. Since $G$ acts transitively on $G/P^{\opp}_{\Theta}$, any affine chart~$ \Fl_{\Theta} \smallsetminus \hyp_y$ with $y \in \Fl_{\Theta}^\opp$ admits an affine structure, which moreover depends only on $y$ (and not on the choice of $g \in G$ such that $y = g P_{\Theta}^{\opp}$).

\indent If $i(\Theta) = \Theta$, then we say that $\Fl_{\Theta}$ is \emph{self-opposite}. For any representative $k_0 \in N_K(\aaa)$ of the longest element $w_0$, one has $P_{\Theta} = k_0 P_{\Theta}^{\opp} k_0^{-1}$, and one can identify $\Fl_{\Theta}$ and $\Fl_{\Theta}^{\opp}$ via the $G$-equivariant diffeomorphism $ \Fl_{\Theta} \rightarrow \Fl_{\Theta}^{\opp}; \ hP_{\Theta} \mapsto  h k_0 P_{\Theta}^{\opp}$. We will always make this identification in this paper.

\subsubsection{The automorphism group. } \label{sect_autom_group} The group of all the Lie algebra automorphisms of $\g$ is called the \emph{automorphism group} of $\g$ and denoted by $\Aut(\g)$. It is a Lie group with Lie algebra $\g$. When $G$ is simple, the map $\Ad: G \rightarrow \Aut(\g)$ has finite kernel.

In general, the group $\Aut(\g)$ does not act on $\Fl_{\Theta}$. However, it admits a finite-index subgroup that does: indeed, any $g \in \Aut(\g)$ induces an automorphism $\psi_g$ of the fundamental system  $\Delta$. This defines a group homomorphism $\Aut(\g) \rightarrow \Aut(\Delta)$, whose kernel is denoted by $\Aut_1(\g)$ and contains $\Ad(G)$. This kernel acts on $\Fl_{\Theta}$ for any~$\Theta \subset \Delta$. This action extends the one of $G$, in the sense that for any $g \in G$ and for any $x \in \Fl_{\Theta}$, one has $g \cdot x = \Ad(g) \cdot x$. In particular $\ker(\Ad)$ acts trivially on~$\Fl_{\Theta}$.

\subsection{Irreducible representations of semisimple Lie groups. } \label{sect_irred_semi} Let $(V, \rho)$ be a finite-dimensional real linear representation of a real semisimple Lie group $G$. We will denote by $\rho_*: \g \rightarrow \operatorname{End}(V)$ the differential of $\rho$ at $\operatorname{id}$.

\subsubsection{Restricted weights}
 
 For any $\lambda \in \aaa^*$, we define
\begin{equation*}
    V^{\lambda} := \{ v \in V \mid \rho(h) \cdot v = \lambda (h) v \quad \forall h \in \aaa \}.
\end{equation*}
If $V^{\lambda} \ne \{ 0 \}$ we say that $\lambda$ is a \emph{restricted weight} of $(V, \rho)$. Given $\alpha, \lambda \in \aaa^*$, one has~$\rho_*(X)\cdot V^{\lambda} \subset V^{\lambda + \alpha} $ for all $X \in \g_{\alpha}$. For each $\alpha \in \FS$, the element $\omega_{\alpha} \in \aaa^*$ introduced in Section~\ref{sect_real_lie_alg} is called the \emph{fundamental weight} associated with $\alpha$. The cone generated by the simple roots determines a partial ordering on $\aaa^*$ given by
\begin{equation*}
    \lambda \leq \lambda' \Longleftrightarrow \lambda' - \lambda \in \sum_{\alpha \in \FS} \mathbb{R}_+ \alpha.
\end{equation*}

If $(V, \rho)$ is a finite-dimensional real irreducible linear representation of $G$, then the set of restricted weights of $(V, \rho)$ admits a unique maximal element for that ordering (see \cite[Cor.\ 3.2.3]{goodman2009symmetry}). This element is called the \emph{highest weight} of $\rho$, and denoted by $\chi_{\rho}$ or $\chi$.

\subsubsection{Proximality and $\Theta$-proximal representations. } An automorphism $g \in \GL(V)$ is said to be \emph{proximal} in $\mathbb{P}(V)$ if it has a unique eigenvalue of maximal modulus and if the corresponding eigenspace is one-dimensional.

Let $\Theta \subset \FS$ be a nonempty subset of the simple roots. We say that $(V, \rho)$ is \emph{$\Theta$-proximal} if $\rho(G)$ contains a proximal automorphism of $V$ and $\{\alpha \in \FS \mid \langle \chi, \alpha \rangle >0\} = \Theta$ (see \cite{gueritaud2017anosov}). In this case,
we denote by $V^{\chi}$ the weight space associated with $\chi$ (it is automatically one-dimensional), and $V^{<\chi}$ the sum of all other weight spaces. 
\begin{prop}[{\cite[Prop.\ 3.3]{gueritaud2017anosov}}]\label{prop_ggkw} One has:
    
\begin{enumerate}
    \item The stabilizer of $V^{\chi}$ in $G$ (resp.\ $V^{< \chi}$) is $P_{\Theta}$ (resp.\ $P_{\Theta}^{\opp}$).
    \item The maps $g \mapsto \rho(g) \cdot V^{\chi}$ and $g \mapsto \rho(g) \cdot V^{<\chi}$ induce two $\rho$-equivariant maps: 
    \begin{equation*}
        \iota: \Fl_{\Theta} \longrightarrow \mathbb{P}(V) \text{ and } \iota^*: \Fl_{\Theta}^{\opp} \longrightarrow \mathbb{P}(V^*).
    \end{equation*}
    Two elements $(x,\xi) \in \Fl_{\Theta} \times \Fl_{\Theta}^{\opp}$ are transverse if, and only if, their images $\iota(x)$ and $\iota^*(\xi)$ are.
\end{enumerate}
\end{prop}

Note that for $\nu \in V \smallsetminus \{0 \}$ and $f \in V^* \smallsetminus \{ 0 \}$, the transversality of $[\nu] \in \mathbb{P}(V)$ and~$[f] \in \mathbb{P}(V^*)$ is equivalent to $f(\nu) \ne 0$.

We will identify the space $\mathbb{P}(V^*)$ with the set $\mathbb{P}'(V)$ of projective hyperplanes of~$\mathbb{P}(V)$, via the bijection $\mathbb{P}(V^*) \overset{\sim}{\rightarrow} \mathbb{P}'(V)$ defined by $
        [f] \mapsto \mathbb{P}(\ker(f)).$ With this identification and the notations of Proposition~\ref{prop_ggkw}, two elements~$(x,\xi) \in \Fl_{\Theta} \times \Fl_{\Theta}^{\opp}$ are transverse if, and only if, one has $\iota(x) \notin \iota^*(\xi)$.

\section{Shilov boundaries and causality}\label{shilov_bnds} If $G$ is a simple Lie group of Hermitian type of tube type, that is, if the symmetric space $\mathbb{X}_G$ of $G$ is irreducible and Hermitian of tube type, then we will say that $G$ is a \emph{$\HTT$ Lie group}, and $\g$ a \emph{$\HTT$ Lie algebra}. In this section, we fix an $\HTT$ Lie group~$G$ with Lie algebra $\g$ and prove useful preliminary results.

\subsection{Strongly orthogonal roots and root system} \label{sect_root_syst_shilov} Two roots $\alpha, \beta \in \Sigma$ are called \emph{strongly orthogonal} if neither $\alpha + \beta$ nor $\alpha- \beta$ is a root. Since $G$ is of tube type, there exists a (maximal) set $\{ 2\e_1, \cdots, 2\e_r \} \subset \Sigma$ of strongly orthogonal roots, such that the set $\FS= \{ \alpha_1, \cdots, \alpha_r \}$ is a fundamental system of $\Sigma$, where $\alpha_i = \e_i - \e_{i+1}$ for $i < r$ and~$\alpha_r = 2 \e_r$. The system $\Sigma$ is then of type $C_r$ (see e.g.\ \cite{faraut1994analysis}):
\begin{equation*}
    \Sigma  = \{\pm \e_i \pm \e_j \mid 1 \leq i \leq j \leq r\}; \quad  \Sigma^+ = \{\e_i \pm \e_j \mid 1 \leq i < j \leq r\} \cup \{2 \e_i \mid 1 \leq i \leq r\}.
\end{equation*}

This is actually a characterization of $\HTT$ Lie groups:

\begin{fact}
    A simple Lie group $G$ is $\HTT$ if, and only if, its root system is of type $C_r$.
\end{fact}

Recall the notation of Section~\ref{sect_root_syst_shilov}. Then:

\begin{lem}\label{lem_auxiliary_lemma} Let $G$ be an $\HTT$ Lie group. Let $\beta \in \Sigma_{\{\alpha_r\}}^+ \smallsetminus \{ \alpha_r\}$, and let $Y_{\beta} \in \g_{\beta}$. Then there exists $Z \in \operatorname{Lie}\left(P_{\{ \alpha_r\}}\right)$ such that $[v^-, Z] = 0$ and $\Ad (\exp\left(Y_{\beta}\right))v^- =  v^- +Z$.

\end{lem}
\begin{proof}[Proof] One has 
\begin{equation*}
    \Ad \left(\exp\left(Y_{\beta}\right) \right)v^-  = \exp\left(\ad(Y_{\beta})\right) v^- =  \sum_{k = 0}^{\infty} \frac{\ad(Y_{\beta})^k v^-}{k!}.
\end{equation*}

We compute the terms $\ad(Y_{\beta})^k v^-$ for $k \in \mathbb{N}$. Recall that $\Sigma$ is of type $C_r$. In the notation of Section~\ref{sect_recalls_lie_groups}, there are thus several possible values for $\beta$:

If $\beta = 2\e_i$ for some $1 \leq i \leq r-1$ or $ \e_i +\e_j$ for some $1 \leq i < j \leq r-1$, then~$\beta-\alpha_r \notin \Sigma$, so $[Y_{\beta}, v^-] \in \g_{\beta - \alpha_r} = \{ 0 \}$. Hence for any $k \geq 1$, one has $\ad(Y_{\beta})^k v^- = 0$, and so~$\Ad(e^{Y_{\beta}})v^- = v^-$. In this case, we have $Z = 0$.

If $\beta = \e_i + \e_r$ for some $1 \leq i < r$, then  $3 \beta - \alpha_r = 3\e_i +\e_r \notin \Sigma$. Hence for any $k \geq 3$ one has $\ad(Y_{\beta})^k v^- = 0$. From this we deduce that $e^{\ad(Y_{\beta})}v^- = v^- +Z$, with 
$$Z := [Y_{\beta}, v^-] + \frac{1}{2}\left[Y_{\beta}, [Y_{\beta}, v^-]\right] \in \g_{\beta - \alpha_r} \oplus \g_{2 \beta- \alpha_r} = \g_{\e_i - \e_r} \oplus \g_{2\e_i} \subset \operatorname{Lie}\left(P_{\{ \alpha_r\}}\right)$$
by~\eqref{eq_lie_P}. Moreover, one has $[v^-, Z] \in \g_{\e_i - 3\e_r} \oplus \g_{2(\e_i-e_r)} = \{0\}$. $\qed$
\end{proof}

If $\Theta = \{\alpha_r\}$, then the flag manifold $\Fl_{\Theta}$ is called the \emph{Shilov boundary} of $\mathbb{X}_G$ and we will denote it by $\SB(G)$. 

Since $\FS$ is of type $C_r$, its automorphism group is trivial (see e.g.\ \cite{knapp1996lie}). Hence the opposition involution is trivial and the flag manifold $\SB(G)$ is self-opposite.
Moreover, the groups $\Aut_1(\g)$ and $\Aut(\g)$ coincide, hence the group $\Aut(\g)$ acts on $\Fl_{\{\alpha_r\}}$.

\subsubsection{Notation. }\label{sect_notation} When $G$ is an $\HTT$ Lie group and $\Theta = \{ \alpha_r\}$, we will always use the following simplified notation: $\uu^{\pm} = \uu_{\{\alpha_r\}}^{\pm}$, $\mathfrak{l} = \mathfrak{l}_{\{\alpha_r\}}$, $U^{\pm} = U_{\{\alpha_r\}}^{\pm}$, $P = P_{\Theta}$, $L=  L_{\Theta}$. Note that the Lie algebras $\uu^{\pm}$ are abelian.

Since $P$ is a maximal proper parabolic subgroup of $G$, the center of $L$ is at most one-dimensional, and one can write $L = \mathbb{R} \times L_s$, where $L_s$ is the semisimple part of $L$. The possible values of the Lie algebra $\mathfrak{l}_s$ of $L_{s}$ are listed in Table \ref{table_shilov_bnds}. 

\subsection{Dilations and translations.}\label{section_structure_shilov} Let $G$ be an $\HTT$ Lie group. There exists $H_0$ in the center of $\mathfrak{l}$ such that $\uu^{\pm}$ is the root space of $\ad(H_0)$ for the eigenvalue $\pm 1$ (see e.g. \cite{kaneyuki1998sylvester}). For all $t \in \Rf_{>0}$ we define $   \ell_0(t) = \exp \left(-\log\left(\sqrt{t}\right)H_0\right) \in L$. The map $\ell_0$ continuously extends to $\operatorname{id}$ at $0$. The element $\Ad(\ell_0(t))$ acts on $\uu^{\pm}$ by
    \begin{equation}\label{eq_def_lzero}
    \Ad(\ell_0(t))X =
    \begin{cases}
         t X \quad &\forall X \in \uu^-; \\
    \frac{1}{t} X \quad &\forall X \in \uu^+.
    \end{cases}
    \end{equation}

 Hence any positive dilation of $\Affstd$ (see~\eqref{eq_standard_chart}) at $P = \varphi_{\mathsf{std}}(0)$ can be realized as the restriction to $\Affstd$ of a map of the form $x \mapsto \ell_0(t) \cdot x$ of $\SB(G)$ for some $t \in \Rf_{>0}$. 

Moreover, since $\uu^-$ is abelian, any translation in $\Affstd$ is realized as left multiplication by an element of $U^- \leq G$. Each time we will talk about \emph{a translation} in $\Affstd$, it will mean that we apply a multiplication by an element of $U^-$. 

According to the two previous paragraphs, for any dilation $d$ at a point $x_0 \in \Affstd$, there exists $g \in G$ such that $d$ coincides with the restriction of the map~$x \mapsto g \cdot x$ of $\SB(G)$ to $\Affstd$. Each time we will talk about \emph{dilating at $x_0 $ in} $\Affstd$, it will mean that we apply such a map.

\begin{rmk} It is not true for a general simple Lie group $G$ and subset of the simple roots $\Theta$ that there exists $H_0 \in \mathfrak{l}_{\Theta}$ such that $\ad(H_0)X^{\pm} = \pm X^{\pm}$ for all $X^{\pm} \in \uu_{\Theta}^{\pm}$. This property is equivalent to $\uu^{\pm}_{\Theta}$ being abelian. When this is the case, the algebra $\g$ admits a decomposition $\g = \g_{-1} \oplus \g_{0} \oplus \g_{1}$, with $\g_{-1} = \uu_{\Theta}^-$, $\g_0 = \mathfrak{l}_{\Theta}$ and $\g_{1} = \uu_{\Theta}^+$, and~$[\g_k,\g_{k'}] \subset \g_{k+k'}$ for $k,k' \in \{-1, 0, 1\}$, with $\g_m := \{ 0\}$ if $m \notin \{-1, 0,1 \}$. The element $H_0$ is then called the \emph{characteristic element} of the graded Lie algebra $\g$, that is, each space $\g_k$ with $k \in \{-1, 0, 1\}$ is the eigenspace of $\ad(H_0)$ for the eigenvalue $k$.
\end{rmk}

\subsection{An invariant cone and causality} \label{sect_causalité} The identity component $L^0$ of $L$ acts irreducibly on $\uu^-$. By \cite[Prop.\ 4.7]{benoist2000automorphismes} applied to this action, there exists an open~$L^0$-invariant properly convex cone $c^0$ in $\uu^-$. This cone is defined as the interior of the convex hull in $\uu^-$ of the orbit $\Ad(L^0) \cdot v^-$. Note that the existence of such a cone endows $\SB(G)$ with an invariant causal structure (see \cite{kaneyuki2006causal}). 

For any $x \in \Affstd$ (recall~\eqref{eq_standard_chart})), there exists a unique $X \in \uu^-$ such that $x = \exp(X)P$. We define 
\begin{align*}
    &\mathbf{I}^+(x) := \varphi_{\mathsf{std}}(X + c^0) = \exp(c^0) \cdot x \text{, the \emph{future} of }x; \\
        &\mathbf{I}^-(x) := \varphi_{\mathsf{std}}(X - c^0) = \exp(-c^0) \cdot x\text{, the \emph{past} of }x; \\
        &\mathbf{J}^+(x) := \varphi_{\mathsf{std}}(X + \overline{c^0}) = \exp(\overline{c^0}) \cdot x \text{, the \emph{large future} of }x; \\
        &\mathbf{J}^-(x) := \varphi_{\mathsf{std}}(X - \overline{c^0})= \exp(-\overline{c^0}) \cdot x \text{, the \emph{large past} of }x; \\
        &\mathbf{C}^+(x) := \varphi_{\mathsf{std}}(X + \partial c^0) = \exp(\partial c^0) \cdot x \text{, the \emph{future lightcone} of }x; \\
        &\mathbf{C}^+(x) := \varphi_{\mathsf{std}}(X - \partial c^0) = \exp(-\partial c^0) \cdot x  \text{, the \emph{past lightcone} of }x; \\
       & \mathbf{C}(x) := \mathbf{C}^+(x) \cup \mathbf{C}^- (x) \text{, the \emph{lightcone} of }x. 
\end{align*}

The following fact is well-known:

\begin{fact}\label{fact_cone_in_schubert_subvrty}
The lightcone $\mathbf{C}(x)$ of $x \in \mathbb{A}$ is always contained in $\hyp_x \cap \mathbb{A}$, and $\I^{\pm}(x)$ are connected components of $\Affstd \smallsetminus \hyp_x$.
\end{fact}

 The past, the future and the lightcone of a point $x \in \Affstd$ are not invariant under the stabilizer  $\operatorname{Stab}_G(x)$ of $x$ in $G$. However, they are locally invariant:

\begin{lem}\label{lem_loc_action_U+}
    Let $x \in \Affstd$ and $g \in G$ be such that $g \cdot x \in \Affstd$. Then for any $\delta_1 \in \{-, +\}$, there exist $\delta_2 \in \{-, +\}$ and a neighborhood $\mathcal{U}$ of $x$ such that $g \cdot (\mathcal{U} \cap \mathbf{I}^{\delta_1}(x)) \subset \mathbf{I}^{\delta_2}(g \cdot x)$.
\end{lem}

\begin{proof}[Proof] Noticing that $\exp(X) \cdot \I^{\delta_1}(P) = \I^{\delta_1}(\exp(X)P)$ for all $X \in \uu^-$, we may assume that $x = P$. Let us prove the lemma for $\delta_1 = +$, the proof being the same for $\delta_1 = -$. 

Since $g P \in \Affstd$, by~\eqref{eq_A_egal_exp} one can write $g = g' \exp(Y)$, with $Y \in \uu^+$ and $g' \in P^-$. There exists a neighborhood $\mathcal{U}$ of $P$, convex in $\Affstd$, such that $g \cdot \mathcal{U} \subset \Affstd$. Hence we have~$\exp(Y) \cdot \mathcal{U} \subset (g')^{-1} \cdot \Affstd = \Affstd$. Recall the map $\ell_0: \mathbb{R}_{\geq 0} \rightarrow L$ defined in Section~\ref{section_structure_shilov}. Since $\mathcal{U}$ is convex, by~\eqref{eq_def_lzero}, one has $\ell_0(t) \cdot \mathcal{U} \subset \mathcal{U}$ for all $t \in [0,1]$. Then:
\begin{equation}\label{eq_Y_stab_U}
    \exp(t Y) \cdot \mathcal{U} = \ell_0(t)^{-1} \exp(Y) \ell_0(t) \cdot \mathcal{U} \subset \ell_0(t)^{-1} \exp(Y) \cdot \mathcal{U} \subset \ell_0(t)^{-1} \cdot \Affstd \subset \Affstd.
\end{equation}
Since $\mathcal{U}$ and $\mathbf{I}^+(P)$ are both convex, the set $\mathcal{U} \cap \mathbf{I}^+(P)$ is connected. For this reason, by~\eqref{eq_Y_stab_U} and since $U^+$ stabilizes $\hyp_P$, the set $\exp(t Y) \cdot (\mathcal{U} \cap \mathbf{I}^+(P))$ is contained in a connected component of $\Affstd \smallsetminus \hyp_P$ for all $t \in [0,1]$, let us denote this component by $\mathcal{V}$. By continuity, this component $\mathcal{V}$ does not depend on $t$. In particular, for $t = 0$, one has $\exp(t Y) = \operatorname{id}$, so $\mathcal{V} = \mathbf{I}^+(P)$. Hence $\exp(tY) \cdot  (\mathcal{U} \cap \mathbf{I}^+(P)) \subset \mathbf{I}^+(P)$. 

Now, since $g' \in P^-$, there exist $X \in \uu^-$ and $\ell \in L$, such that $g' = \exp(X) \ell$. But the element $\ell \in L$ either preserves $\mathbf{I}^+(P)$ or maps it to $\mathbf{I}^-(P)$ (see e.g.\ \cite[Cor.\ 5.3]{guichard2022generalizing}). On the other hand, since $U^-$ is abelian, one has $\exp(X)\cdot \I^{\pm}(P) = \I^+(\exp(X)P)$. Then:
\begin{align*}
g \cdot (\mathcal{U} \cap \I^+(P)) &\subset  g' \cdot  \I^+(P) \\ &= \exp(X) \ell \cdot  \I^+(P)\\ 
&=     \begin{cases}
\I^+(\exp(X)P) = \I^+(gP) \quad \text{if } \ell \cdot \mathbf{I}^+(P) = \mathbf{I}^+(P); \\
 \I^-(\exp(X)P) = \I^-(gP) \quad  \text{if } \ell \cdot \mathbf{I}^+(P) = \mathbf{I}^-(P).
    \end{cases} \qed
\end{align*}
\end{proof}

 The future and the past of a point satisfy the following causality properties:
\begin{fact}\label{lem_causality_order_relation}
    For all $x,y,z \in \Affstd$, one has:
    \begin{enumerate}
        \item [*]\emph{(reflexivity)} $x \in \mathbf{J}^+(y)\Leftrightarrow y \in \mathbf{J}^-(x)$;
        \item [*]\emph{(antisymmetry)} $[x \in \mathbf{J}^+(y) \text{ and } y \in \mathbf{J}^+(x)] \Rightarrow x=y$;
        \item [*]\emph{(transitivity)} $[x \in \mathbf{J}^{\pm}(y) \text{ and } y \in \mathbf{J}^\pm(z) \text{ } ] \Rightarrow x \in \mathbf{J}^{\pm} (z)$
    \end{enumerate}
These properties are also true replacing ``$\mathbf{J}$" with ``$\mathbf{C}$". Reflexivity and transitivity are also true replacing ``$\mathbf{J}$" with ``$\mathbf{I}$".
\end{fact}

\subsection{Examples}\label{ex_shilov_bndrs}

The complete list of Shilov boundaries associated with $\HTT$ Lie algebras is given in Table \ref{table_shilov_bnds}.

    \begin{table}[H]
        \centering
        \begin{tabular}{|c|c|c|}
        \hline
           $\g$   & $\SB(G)$ & $\mathfrak{l}_{s}$ \\
         \hline
             $\soo(2,n)$, $n \geq 3$ &  $\Ein^{n-1,1}$& $\soo(n-1,1)$  \\
           \hline
              $\spp(2r, \mathbb{R})$ & $\Lag_r(\mathbb{R}^{2r})$&$\mathfrak{sl}(r, \mathbb{R})$  \\
           \hline
              $\uu(r,r)$  &$\Lag_r(\mathbb{C}^{2r})$&  $\mathfrak{sl}(r, \mathbb{C})$  \\
           \hline
               $\soo^*(4r)$  & $\Lag_r(\mathbb{H}^{2r})$&   $\mathfrak{sl}(r, \mathbb{H})$  \\
           \hline
       $\mathfrak{e}_{7(-25)}$ &  $\Lag_3(\mathbb{O}^{6})$ &  $\mathfrak{e}_{6(-26)}$  \\
           \hline
        \end{tabular}
        \caption{Shilov boundaries associated with all the $\HTT$ Lie algebras.} \label{table_shilov_bnds}
    \end{table}

 Let us explain the notations in the table. For the notation $\mathfrak{e}_{7(-25)}$ and $\mathfrak{e}_{6(-26)}$, see \cite{faraut1994analysis}.
 
\subsubsection{The Lagrangians.} \label{sect_lag} Let $\Kf = \Rf, \Cf$ or $\Hf$, and let $r \geq 2$. We will define the set of Lagrangians as a submanifold of the space $\GRK$ of all the $r$-dimensional vector spaces of $\Kf^{2r}$. Let
\begin{equation*}
    J_{\Kf} = \begin{pmatrix}
        0 & -I_r \\ I_r & 0
    \end{pmatrix} \in \GL(2r, \Kf),
\end{equation*}
where $I_r$ is the identity matrix of size $r$. Given a matrix $g \in M_{2r}(\Kf)$, we denote by $\overline{g}$ the matrix whose $(i,j)$-th entry is the conjugate (in $\Kf$) of the $(i,j)$-th entry of $g$. Let~$G_{\Kf} := \left\{g \in \SL(2r, \Kf) \mid   \ ^t \overline{g} J_{\Kf} g = J_{\Kf}\right\}$. One has $G_{\Rf}= \Sp(2r, \Rf)$, $G_{\Cf} = \SU(r,r)$ and $G_{\Hf} = \SO^* (2r)$.

Let $\mathbf{b}$ be the bilinear form whose matrix in the canonical basis of $\Kf^{2r}$ is $J_{\Kf}$. The space $\Lag_r(\Kf^{2r})$ is the space of Lagrangians of $(\Kf, \mathbf{b})$, i.e.\ the space of totally isotropic~$r$-planes of $\Kf^{2r}$. It is a model for $\SB(G_{\Kf})$.

If $(e_1, \cdots e_{2r})$ is the canonical basis of $\Kf^{2r}$, then $\SB(G_{\Kf}) = G_{\Kf}/P$, where $P$ is the stabilizer in $G_{\Kf}$ of $\xi_0 :=\operatorname{Span}(e_1, \cdots , e_r)$, and $P^{\opp}$ is the stabilizer in $G_{\Kf}$ of $\xi_{\infty} := \operatorname{Span} (e_{r+1}, \cdots, e_{2r})$. This model gives the following descriptions of $\uu^{\pm}$ and $L$:
\begin{align*}
     \uu^- &= \left\{ \begin{pmatrix} 0_r & 0 \\ X & 0_r\end{pmatrix} \mid X \in  \Mat_{r}(\Kf), \ ^t \overline{X} = X \right\}; \\
        \uu^+ &= \left\{ \begin{pmatrix} 0_r & X \\ 0 & 0_r\end{pmatrix} \mid X \in  \Mat_{r}(\Kf), \ ^t \overline{X} = X \right\};\\
        L &= \big{\{} \operatorname{diag}(A, \ ^t\overline{A}^{-1}) \mid A \in  \GL(r, \Kf)\big{\}}. 
\end{align*}
Hence the affine chart $\Affstd = \Lag_r(\Kf^{2r}) \smallsetminus \hyp_{P^{\opp}} $ can be described as follows:
\begin{align*}
  \Affstd  = \left\{\operatorname{Im} \begin{pmatrix}
    I_r \\ X
    \end{pmatrix} \ \Big{\vert} \ X \in \Mat_{r} (\mathbb{K}), \ ^t \overline{X} = X\right\}.  
\end{align*}

Then a computation gives:
\begin{equation*}
    \begin{split}
         \Affstd \cap \hyp_{P} &= \left\{ \operatorname{Im}\begin{pmatrix}
        I_r \\ X
    \end{pmatrix} \ \Big{\vert} \ 
    ^t \overline{X} = X, \ \det(X) = 0\right\}; \\
     \mathbf{C}(P) &= \left\{ \operatorname{Im}\begin{pmatrix}
        I_r \\ X
    \end{pmatrix} \ \Big{\vert} \ ^t \overline{X}= X, \  \det(X) = 0 \text{ and } ^t \overline{x} X x \in \mathbb{R}_{\geq 0} \quad \forall x \in \Kf^{r} \right\}.
    \end{split}
\end{equation*}
Note that if $r \geq 3$, then the inclusion $\mathbf{C}(P) \subset \Affstd \cap \hyp_{P}$ is strict.

\subsubsection{The Einstein Universe.}\label{sect_Einstein_universe} Let $n \geq 2$ and let $(e_1, \cdots, e_{n+2})$ be the canonical basis of $\mathbb{R}^{n+2}$. For any vector $v \in \mathbb{R}^{n+2}$, we denote by $v_i$ the $i$-th coordinate of $v$, that is~$v = \sum_{i=1}^{n+2} v_i e_i$. Let $\b$ be the quadratic form of signature $(n,2)$ on $\mathbb{R}^{n+2}$ defined as:
\begin{equation*}
    \b (v,w) = \sum_{i=1}^{n-1} v_i w_i - v_n w_n - \frac{1}{2}(v_{n+1} w_{n+2} + v_{n+2} w_{n+1}) \quad \forall v,w \in \mathbb{R}^{n+2}.
\end{equation*}

Let $G = \SO(2,n)$ and $P$ be the stabilizer of $[e_{n+2}]$ in $G$. Then $\SB(G) = G / P$  identifies with the \emph{Einstein Universe} $\Ein^{n-1, 1}$, i.e.\ the set of isotropic lines of $(\Rf^{n,2}, \mathbf{b})$. Explicitly, the affine chart $\Affstd$ defined by $[e_{n+2}]$ is
\begin{equation*}
    \Affstd = \Ein^{n-1, 1} \smallsetminus \mathbb{P}\left(\left(\mathbb{R}e_{n+2}\right)\right)^{\perp_\b}) = \mathbb{P}\Big{\{} \sum_{i=1}^n v_i e_i + \psi\Big{(}\sum_{i=1}^n v_i e_i \Big{)} e_{n+1} + e_{n+2}  \Big{\}},
\end{equation*}
where $\psi$ is the quadratic form of signature $(n-1, 1)$ on $V:= \operatorname{Span}(e_1, \cdots, e_n)$ defined by $\psi \left(\sum_{i=1}^n v_i e_i \right) = \sum_{i=1}^{n-1} v_i^2 - v_n^2$. The identification $V \simeq \Affstd$ given by
\begin{equation*}
 \sum_{i=1}^n v_i e_i \longmapsto \mathbb{P}\Big{(}\sum_{i=1}^n v_i e_i + \psi\Big{(}\sum_{i=1}^n v_i e_i \Big{)} e_{n+1} + e_{n+2} \Big{)}
\end{equation*}
endows $\Affstd$ with the structure of a Minkowski space. Still denoting by $\psi $ the quadratic form induced on $\Affstd$ by this identification, one has:
\begin{equation*}
    \Affstd \cap \hyp_{y_0} = \{y \in \Affstd \mid \psi(y-y_0) = 0\} = \mathbf{C} (y_0),
\end{equation*}
see Figure~\ref{figure_past_and_future}. Hence in this case, the inclusion $\mathbf{C}(y_0) \subset \hyp_{y_0} \cap \Affstd$ is an equality for all~$y_0 \in \Affstd$. This phenomenon widely simplifies the proof of Theorem~\ref{thm_main} in the case where $G = \SO(n, 2)$; see \cite{chalumeau2024rigidity}.

\subsection{Embedding the projective line into $\SB(G)$}\label{sect_plong_proj_line} In this section we construct embeddings of the projective line into $\SB(G)$. The images of these embeddings are what we will define as \emph{photons} in Section~\ref{sect_photons_def}. In Lemma~\ref{lem_action_U_on_A}, we investigate the action of $U^+$ on these images.

Let $G$ be an $\HTT$ Lie group. We denote by $h_r$ the element $h_{\alpha_r}$ defined in Section~\ref{sect_real_lie_alg}. The root spaces $\g_{\pm\alpha_r}$ are one-dimensional, hence there are unique~$v^{\pm} \in \g_{\pm\alpha_r}$ such that 
\begin{equation}\label{eq_def_v}
    \tr_{\mathsf{std}} = (v^+, h_{r}, v^-)
\end{equation}
is an $\mathfrak{sl}_2$-triple, which we call \emph{standard}. The map $\plongsl_{\tr_{\mathsf{std}}}$ associated with $\tr_{\mathsf{std}}$ (see Section~\ref{sect_sl2-triples}) induces a group homomorphism $\plong: \SL_2(\mathbb{R}) \hookrightarrow G$ with kernel contained in~$\{\pm \operatorname{id}\}$ and with  differential $\plong_* = \plongsl_{\tr_{\mathsf{std}}}$ at $\operatorname{id}$.

\begin{lem}\label{lem_stab_parap}
    The stabilizer of $P$ in $\SL_2(\mathbb{R})$ is the standard Borel subgroup $P_1$ of~$\SL_2(\mathbb{R})$.
\end{lem}

\begin{proof}[Proof] Let us denote by $S$ this stabilizer. Note that $e^{\E}, e^{\He} \in S$, so the identity component $P_1^0$ of $P_1$ is contained in $S$. Since the orbit of $P \in \SB(G)$ is nontrivial, we have $P_1^0 \subset S \subset P_1$. It remains to show that $g := \plong(- \operatorname{id})$ is in $ P$. Noticing that $\Ad(g) = \Ad\left(\plong\left(e^{\pi (\E - \F)}\right)\right) = \exp\left(\pi \ad\left(v^+ - v^-\right)\right)$,
a direct computation provides $\Ad(g) \cdot h_{2 \varepsilon_i} = h_{2 \varepsilon_i}$ for all $ \forall 1 \leq i \leq r-1$. On the other hand, one has~$\Ad(g) \cdot h_{2 \varepsilon_r} = \Ad(g) \cdot h_r = \plong_*(\Ad(-\operatorname{id})\He) = h_r =h_{2 \varepsilon_r}$. Then $\Ad(g)$ acts trivially on  the basis $(h_{2 \varepsilon_i})$ of $\aaa$, and thus on $\aaa$. In particular, the element $g$ normalizes every root space $\g_{\alpha}$ for $\alpha \in \Sigma$. Then, by~\eqref{eq_lie_u}, the element $g$ normalizes $\uu^+$. Hence $g \in P$. This proves that $S = P_1$. $\qed$
\end{proof}

By Lemma~\ref{lem_stab_parap}, the map $\plong$ induces a $\plong$-equivariant embedding $\isl:\mathbb{P}(\mathbb{R}^2) \hookrightarrow \SB(G)$. It will be convenient to write this map explicitly:
\begin{equation}\label{eq_param_photons}
    \isl ([1:t]) =  \exp(tv^-) P \quad \forall t \in \mathbb{R}.
\end{equation}

\begin{lem}\label{lem_action_U_on_A}
For any $Y \in \uu^+$, there exists $\lambda \in \mathbb{R}$ such that  for all $t \in \mathbb{R} \smallsetminus \{-\lambda^{-1}\}$ (with $-\lambda^{-1} = \infty$ if $\lambda =0$), one has:
    \begin{equation*}
        \exp(Y) \cdot \isl([1:t]) = \isl\Big{(}\Big{[}1:\frac{t}{1+ \lambda t}\Big{]}\Big{)}.
    \end{equation*}
\end{lem}

\begin{proof}[Proof] We write $Y = \lambda v^+ + \sum_{\beta \in \Sigma_{\Theta}^+ \smallsetminus \{ \alpha_r \}} Y_{\beta}$ with $Y_{\beta} \in \g_{\beta}$ for all $\beta \in \Sigma_{\Theta}^+ \smallsetminus \{ \alpha_r \}$ and with $\lambda \in \mathbb{R}$. Since $\uu^+$ is abelian, one has:
\begin{equation}\label{actionYsurv}
     \exp(Y) \cdot \isl([1:t]) = \exp(\lambda v^+)\Bigg{(}\prod_{\beta \in \Sigma_{\alpha_r}^+ \smallsetminus \{\alpha_r\} }\exp(Y_{\beta}) \Bigg{)} \cdot \isl([1:t]).
\end{equation}
Let $\beta \in  \Sigma_{\alpha_r}^+ \smallsetminus \{\alpha_r\}$. Since $Y_{\beta} \in \operatorname{Lie} (P)$, one has 
\begin{align*}
    \exp(Y_{\beta}) \cdot \isl([1:t]) = \exp(Y_{\beta}) \exp(t v^-) P &= \exp(Y_{\beta}) \exp(t v^-) \exp(-Y_{\beta}) P \\
    &= \exp\left(t \Ad \left(\exp\left(Y_{\beta}\right) \right) v^-\right) P.
\end{align*} 
By Lemma~\ref{lem_auxiliary_lemma}, there exists $Z \in \operatorname{Lie}(P)$ such that $[v^-, Z] = 0$. Then 
\begin{align*}
    \exp\left(t \Ad \left(\exp\left(Y_{\beta}\right) \right) v^- \right) P &= \exp(t \Ad \left(\exp\left(Y_{\beta}\right)  \right) v^-) P \\
    &= \exp(t(v^- + Z)) P 
    = \exp(tv^-)\exp(t Z) P 
    = \exp(t v^-) P,
\end{align*}
the second last equality holding because $v^-$ and $Z$ commute, and the last one holding because $Z \in  \operatorname{Lie}(P)$. Hence by induction, Equation~\eqref{actionYsurv} becomes:
\begin{align*}
   \exp(Y)\cdot \isl([1:t]) &= 
   \exp(\lambda v^+) \cdot \isl([1:t]) = \plong(e^{\lambda \E}) \cdot \isl([1:t]) \\
   &=  \isl(e^{\lambda \E} \cdot [1:t]) 
   = \isl\Big{(}\Big{[}1:  \frac{t}{1+\lambda t }\Big{]}\Big{)},
\end{align*}
the last equality holding by an elementary computation and the second last by $\plong$-equivariance of $\isl$. $\qed$
\end{proof}

\subsection{Proximal representations of $\HTT$ Lie groups}\label{sect_irred_htt} We denote by $\omega_r$ the fundamental weight $\omega_{\alpha_r}$ associated with the root $\alpha_{r}$, with the notation of Section~\ref{sect_real_lie_alg}. There always exists $N \in \mathbb{N}^*$ such that $G$ admits a finite-dimensional real irreducible linear representation of highest weight $N \omega_r$ (see e.g.\ \cite[Lem.\ 3.2]{gueritaud2017anosov}). Such a representation is automatically $\{\alpha_r\}$-proximal. We will use the following classical lemma throughout this paper:

\begin{lem}\label{lem_degree_rep} Let $G$ be an $\HTT$ Lie group and let $(V, \rho)$ be a finite-dimensional real irreducible linear representation of $G$ with highest weight $\chi = N \omega_r$ for some $N \in \mathbb{N}_{>0}$. Let $e_1 \in V^{\chi} \smallsetminus \{ 0 \}$. Then $\rho_*(v^-)^k e_1 \ne 0$ for all $k \leq N$, and $\rho_*(v^-)^k e_1 = 0$ for all~$k \geq N+1$.
\end{lem}
\begin{proof}[Proof] By the definition of $\omega_r$ one has $0 \ne \rho_*(v^-)^k e_1 \in V^{N\omega_r - k\alpha_r}$ for all $0 \leq k \leq N$ (see e.g.\ \cite[Lem.\ 3.2.9]{goodman2009symmetry}).

Since $\FS$ is of type $C_r$, the Weyl group $W$ acts by signed permutations of the $(\varepsilon_i)_{1 \leq i \leq r}$ (see e.g.\ \cite{knapp1996lie}) and $\omega_r = \sum_{i=1}^r \varepsilon_i$. Hence the convex envelope of the orbit $W \cdot (N\omega_r)$ of $N\omega_r$ by $W$ in $\mathfrak{a}$ is $C := \{N\sum_{i=1}^r \delta_i \varepsilon_i \ \vert \ \delta_i \in \{\pm 1\} \}$.

The set of weights of $(V, \rho)$ is the intersection of $C$ with the translate $(N\omega_r) + \sum_{i=1}^r \mathbb{Z} \alpha_i$ of the root lattice (see e.g.\ \cite[Prop.\ 3.2.10]{goodman2009symmetry}). For $k \geq N+1$, one has 
$$N\omega_r- k \alpha_r  = N\sum_{i=1}^{r-1} \varepsilon_i - (2k - N) \varepsilon_r \notin C$$ 
Hence $\rho_*(v^-)^k e_1 \in V^{2\omega_r - k\alpha_r} = \{0\}$. $\qed$
\end{proof}

\section{Proper domains and their automorphism group}\label{sect_proper_autom}

In this section we recall some definitions and properties of domains in a flag manifold. In particular, we introduce the \emph{diamonds} (see Definition~\ref{def_diamant}) in $\SB(G)$, and recall their basic properties. These proper domains are almost-homogeneous, and Theorem~\ref{thm_main} states that they are the only proper domains in $\SB(G)$ with this property. 

\subsection{Generalities on proper domains} The notions recalled in this section are generalizations of classical notions of convex projective geometry, and most of them were introduced in \cite{zimmer2018proper}.

Let $G$ be a noncompact simple Lie group and $\Theta \subset \FS$ a subset of the simple roots.

\begin{definition}\label{def_domain_proper}
    Let $\O \subset \Fl_{\Theta}$ be an open subset. We say that $\O$ is:
    \begin{enumerate}
        \item a \emph{domain} if $\O$ is nonempty and connected;
        \item \emph{proper} if there exists $\xi \in G/P^-_{\Theta}$ such that $\overline{\O}\cap \hyp_{\xi}  
 = \emptyset$. In particular, if $\overline{\O}\cap \hyp_{P_{\Theta}^-}  
 = \emptyset$ then we will say that $\O$ is \emph{proper in $\Affstd$}. This is equivalent to saying that $\overline{\O}\subset \Affstd$.
    \end{enumerate}    
\end{definition}

\begin{rmk}\label{Rmk_contenu_dans_A_std}
    Given a proper domain $\O$ of $\Fl_{\Theta}$, we will always be able to assume that~$\O$ is proper in $\Affstd$. Indeed, since $G$ acts transitively on $\Fl_{\Theta}^-$, there exists $g \in G$ such that $\overline{g \cdot \O} \subset \Affstd$, and the properties we will investigate on $\O$ will be invariant under the action of $G$ on $\Fl_{\Theta}$. In this case, it will be possible to see $\O$ as a bounded domain of the affine space $\Affstd$.
\end{rmk}

\subsubsection{The automorphism group.} \label{sect_group_aut_omega} Given an open subset $\O \subset \Fl_{\Theta}$, we define  its \emph{automorphism group}
\begin{equation*}
    \Aut(\O) = \left\{ g \in G \mid g \cdot \O = \O \right\}.
\end{equation*}
By \cite{zimmer2018proper}, the group $\Aut(\O)$ is a Lie subgroup of $G$. Moreover, it acts properly discontinuously on $\O$ as soon as $\O$ is a proper domain.
\begin{rmk}
    In the case where $G$ is an $\HTT$ Lie group and $\Fl_{\Theta} = \SB(G)$, the group~$\Aut(\O)$ is commensurable to the \emph{conformal group} of $\O$, that is, the group of all the invertible maps from $\O$ to itself that preserve the causal structure of $\O$ \cite[Thm 2.3]{kaneyuki2011automorphism}.
\end{rmk}

The domain $\O$ is said to be \emph{almost-homogeneous} is there exists a compact subset $\mathcal{K} \subset \O$ such that $\O = \Aut(\O) \cdot \mathcal{K}$.

The \emph{full orbital limit set} of $\O$ is the set
\begin{equation*}
    \Lambda_{\O}^{\operatorname{orb}} := \bigcup_{x \in \O} (\overline{\Aut(\O) \cdot x}) \smallsetminus (\Aut(\O) \cdot x),
\end{equation*}
see~\cite{DGKproj}. Since $\Aut(\O)$ acts properly on $\O$, we have $\Lambda_{\O}^{\operatorname{orb}} \subset \partial \O$. A domain $\O$ is said to be \emph{almost-homogeneous} if $\Lambda_{\O}^{\operatorname{orb}} = \partial \O$. This is equivalent to saying that for all $p \in \partial \O$, there exist $x \in \O$ and $(g_n) \in \Aut(\O)^{\mathbb{N}}$ such that $g_n \cdot x \rightarrow p$, and for proper domains it is \emph{a priori} weaker than quasi-homogeneity. Note that if $G$ is an $\HTT$ Lie group, then the notions of quasi-homogeneity and almost-homogeneity only depends on the Lie algebra $\g$ of $G$, because $\Ad(G)$ has finite index in $\Aut(\g)$ (see e.g.\ \cite{satake2014algebraic}). The domain $\O$ is said to be \emph{symmetric} if for any $x \in \O$ there exists an order-two element $s_x \in \Aut_1(\g)$ such that $s_x \cdot \O = \O$ and $x $ is the only fixed point of $s_x$ in $\O$. This notion only depends on the Lie algebra $\g$ of $G$ (recall that, if $G$ is an $\HTT$ Lie group, then one has $\Aut_1(\g) = \Aut(\g)$ --- see Section~\ref{sect_root_syst_shilov}).

Note that $\Aut(g \cdot \O) = g \Aut(\O) g^{-1}$ for all $g \in G$; therefore the property of being almost-homogeneous, resp.\ symmetric, is invariant under the action of $G$ on $\Fl_{\Theta}$. In our case, i.e.\ when $G$ is an $\HTT$ Lie group, by Remark~\ref{Rmk_contenu_dans_A_std} and Section~\ref{section_structure_shilov}, it will always be possible to assume that $\O$ is proper in $\Affstd$, and given a point $x \in \overline{\O}$, we can always assume that $x = P$, up to translating $\O$ by an element of $G$.

We will say that a metric $d$ on $\O$ is \emph{$\Aut(\O)$-invariant} if $d(g \cdot x, g \cdot y) = d(x,y)$ for all~$x,y \in \O$ and $g \in \Aut(\O)$.

We will make use of the following lemma:

\begin{lem}\label{lem_inclusion_equality_almost_homogeneous}
    Let $\O, \O'$ be two proper domains of $\Fl_{\Theta}$ such that $\O \subset \O'$. Assume that $\Aut(\O) \subset \Aut(\O')$ and that $\O$ is almost-homogeneous. Then $\O = \O'$.
\end{lem}

\begin{proof}[Proof]
    
Let $p \in \partial \O$. There exist $x \in \O$ and $(g_n) \in \Aut(\O)^{\mathbb{N}}$ such that $g_n \cdot x \rightarrow p$. Thus $(g_n)$ is unbounded in $G$. Since $\Aut(\O) \subset \Aut(\O')$, the group $\Aut(\O)$ acts properly on the proper domain $\O'$. Thus $p \in \partial \O'$. 

We have proved that $\O$ is closed in $\O'$. Since it is also open, and $\O'$ is connected, we have $\O = \O'$. $\qed$
\end{proof}

We can now prove Lemma~\ref{thm_zim_semisimple}. The proof is similar to that of \cite{zimmer2018proper}. We give it for convenience:

\begin{proof}[Proof of Lemma~\ref{thm_zim_semisimple}] Let us denote by $\O_i \subset \SB(G)$ the image of $\O$ by the canonical projection $\Pi_i: \SB(G) \rightarrow \SB(G_i)$ for $1 \leq i \leq k$, and let $\O' := \O_1 \times \cdots \times \O_k$. Then $\O'
$ is a proper $\Aut(\O)$-invariant domain of $\SB(G)$ containing $\O$. Then by Lemma~\ref{lem_inclusion_equality_almost_homogeneous}, we have $\O = \O'$.

Moreover, note that for all $ 1 \leq i \leq k$, the domain $\O_i$ is almost-homogeneous, since $\partial \O_i \subset \Pi_i(\partial \O)$. $\qed$
\end{proof}

\subsubsection{The dual. }\label{sect_dual_proper_domain} Let $\O \subset \Fl_{\Theta}$ be a subset. The \emph{dual of $\O$} is the set
\begin{equation*}
    \O^*:= \{\xi \in G/P_{\Theta}^- \mid \hyp_{\xi} \cap \O = \emptyset\} \subset \Fl_{\Theta}^\opp.
\end{equation*}
Let us recall some properties of this set (see \cite{zimmer2018proper}):
\begin{enumerate}
    \item The set $\O^*$ is $\Aut(\O)$-invariant, i.e.\ one has $\Aut(\O) \subset \Aut(\O^*)$.
    \item If $\O$ is open, then $\O^*$ is compact. 
    \item The domain $\O$ is proper if, and only if, its dual $\O^*$ has nonempty interior.
\end{enumerate} 

The domain $\O$ is \emph{dually convex}, if for any $p \in \partial\O$ there exists $\xi \in \O^*$ such that~$p \in \hyp_{\xi}$.
 
In \cite[Cor.\ 9.3]{zimmer2018proper}, A. Zimmer proves that any proper almost-homogeneous domain of $\Fl_{\Theta}$ is dually convex. Based on his proof, we give in Section~\ref{sect_def_cara} a slightly stronger result (see Proposition~\ref{prop_zimmer_dual_convex}).

\begin{rmk}\label{rmk_bidual_dual_convex}
\begin{enumerate}
    \item Dual convexity is a generalization to arbitrary flag manifolds of the classical notion of convexity in the projective space, using the dual characterization of convexity. Proposition~\ref{prop_zimmer_dual_convex} is then analogous to a classical result of Kobayashi \cite{kobayashi1984projectively}.
    \item If $\O \subset \Fl_{\Theta}$ is a proper domain, then its bidual $\O^{**}$ is always a proper dually convex domain of $\Fl_{\Theta}$ by definition.
\end{enumerate}
\end{rmk}

\subsection{Diamonds }\label{ssect_def_diamonds} 

In this section, let $G$ be an $\HTT$ Lie group. Given two transverse points $p,q \in \SB(G)$, the set $\SB(G) \smallsetminus (\hyp_p \cup \hyp_q)$ admits several connected components, exactly two of which are proper.
\begin{definition}\label{def_diamant} A subset $\O$ of $\SB(G)$ is called a \emph{diamond} if there exist a (unique) pair of transverse points $p,q \in \SB(G)$ such that $\O$ is one of the two proper connected components of $\SB(G) \smallsetminus (\hyp_p \cup \hyp_q)$. The two points $p,q$ are then called the \emph{endpoints of $\O$}.
\end{definition}

Let $\Diams_{\mathsf{std}} := \I^+(P)$. Recall the order-two element $k_0\in G$ such that $k_0 P k_0^{-1} = P^-$ (see Section~\ref{sect_flag_mfds}). Then $\Diams_{\mathsf{std}}' :=  \mathbf{I}^-(P) = k_0 \cdot \Diams_{\mathsf{std}}$ is the interior of the dual of $\Diams_{\mathsf{std}}$ (see e.g.\ \cite[Lem.\ 13.11]{guichard2022generalizing}), and the domains $\Diams_{\mathsf{std}}$ and $\Diams_{\mathsf{std}}'$ are exactly the two diamonds with endpoints $P$ and $P^{\opp}$. They are proper in $\SB(G)$ --- although they are not proper in $\Affstd$. 

Given two transverse points $p,q \in \SB(G)$, one has $(p,q)=g \cdot (P, P^{\opp})$ for some~$g \in G$. The two diamonds $g \cdot \Diams_{\mathsf{std}}$ and $g \cdot \Diams_{\mathsf{std}}'$ are the diamonds with endpoints $p$ and $q$. 

By the two previous paragraphs, any diamond is a $G$-translate of $\Diams_{\mathsf{std}}$. In particular, up to the action of $G$ on $\SB(G)$, there is only one model of diamond in $\SB(G)$. It is convenient to consider models of diamonds that are proper in $\mathbb{A}$:

\begin{definition}
    If $p,q \in \Affstd$ and $q \in \I^+ (p)$, we define $\Diams(p,q)$ as the set~$\I^+(p) \cap \I^-(q)$. It is one of the two diamonds with endpoints $p$ and $q$.
\end{definition}
For $p,q \in \Affstd$ and $q \in \I^+ (p)$, the diamond $\Diams(p,q)$ is the only one of the two diamonds with endpoints $p$ and $q$ that is proper in $\Affstd$; see Figure \ref{fig:enter-label}.

The following fact is well-known (see e.g.\ \cite[Prop.\ 3.7, 5.2 and Remark 5.4]{guichard2018positivity} and \cite[Thm 2.3 and 3.5]{kaneyuki2011automorphism}): 
\begin{fact}\label{fact_autom_diam}
The automorphism group of $\Diams_{\mathsf{std}}$ in $\Aut(\g)$ contains a symmetry. The diamond $\Diams_{\mathsf{std}}$ is thus a symmetric domain of $\SB(G)$.

Moreover, the action of the identity component $L^0$ of $L$ on $\Diams_{\mathsf{std}}$ is transitive and the stabilizer of a point in $\Diams_{\mathsf{std}}$ is a maximal compact subgroup of $L^0$, so that any diamond is a model for the symmetric space of $L^0$. 
\end{fact}
Since $L$ has one-dimensional center, there exist finite-index subgroups $H_1 \leq \Aut(\Diams_{\mathsf{std}})$ and $H_2 \leq \Iso(\Rf \times X_{L_s})$ such that the diamond $\Diams_{\mathsf{std}}$ is $(H_1, H_2)$-equivariantly diffeomorphic to $\Rf \times X_{L_s}$, where $X_{L_s}$ is the symmetric space of $L_s$ (see e.g.\ \cite[Table I]{kaneyuki2006causal}). These identifications are listed in Table \ref{table_shilov_bnds_2}.

\begin{rmk}\label{rmk_diamonds_basis_of_neighborhoods}
\begin{enumerate}
    \item  The family of diamonds $\Diams(p,q)$ with $p,q \in \Affstd$ and~$q \in \mathbf{I}^+(p)$ forms a basis of neighborhoods of $\Affstd$.
    \item Since $L_s$ admits cocompact lattices \cite{borel1963compact}, any diamond is \emph{divisible}, in the sense that there exists a discrete subgroup $\Gamma \leq \Aut(\O)$ and a compact subset~$\mathcal{K} \subset \O$ such that $\O = \Gamma \cdot \mathcal{K}$.
    \item  Diamonds are defined more generally in any flag manifold $\Fl_{\Theta}$ admitting a $\Theta$-positive structure \cite{labourie2021positivity}. By Zimmer's theorem (Fact~\ref{th_zim_G/P}) and Fact~\ref{fact_autom_diam}, these diamonds are almost-homogeneous (resp.\ divisible) if, and only if, the flag manifold $\Fl_{\Theta}$ is the Shilov boundary associated with an $\HTT$ Lie group.
\end{enumerate}
\end{rmk}

\begin{ex}[Explicit construction of diamonds] \label{ex_eplicit_diamonds} Let us see what diamonds look like, for different values of $G$.

 (1) \textbf{$G = \SO(n,2)$ with $n \geq 3$.} Given two transverse points $p,q \in \Ein^{n-1, 1}$, the set $\SB(G) \smallsetminus (\hyp_p \cup \hyp_q)$ has exactly three connected components (see Figure \ref{figure_past_and_future} for $p = P$ and~$q = P^\opp$). For a general $\HTT$ Lie group $G$, the set $\SB(G) \smallsetminus (\hyp_p \cup \hyp_q)$, where $p,q $ are two transverse points, may have more connected components. If~$\dim(\SB(G)) \geq 3$, then there are exactly $(r+1)$ connected components, where $r$ is the real rank of $G$ (see e.g.\ \cite{kaneyuki1998sylvester}). This is related to the fact that, in general, the inclusion $\mathbf{C}(P) \subset \hyp_P \cap \Affstd$ is strict (see Section~\ref{sect_lag}), and is an obstruction to a direct generalization to any $\HTT$ Lie group of the proof of Theorem~\ref{thm_main} for $G = \SO(n,2)$ ($n \geq 2$) established in \cite{chalumeau2024rigidity}.

(2) \textbf{$G = \Sp(2r, \mathbb{R})$.} Let us take the notation of Section~\ref{ex_shilov_bndrs}. Then
\begin{equation*}
    \Diams_{\mathsf{std}} = \Big{\{} \operatorname{Im}\begin{pmatrix}
        I_r \\ X
    \end{pmatrix} \mid X \in \operatorname{Sym}_r^{++}(\Rf) \Big{\}}  ; \ \Diams_{\mathsf{std}}'  = \Big{\{}\operatorname{Im}\begin{pmatrix}
        I_r \\ -X
    \end{pmatrix} \mid X \in \operatorname{Sym}_r^{++}(\Rf) \Big{\}},
\end{equation*}
where $\operatorname{Sym}_r^{++}(\Rf)$ is the set of positive definite symmetric matrices of size $r~\times~r$. In particular, the diamond $\Diams_{\mathsf{std}}$ is a model for the symmetric space $\Rf \times \left(\SL(r, \Rf)/\SO(r)\right)$. The analogue diamonds for other Shilov boundaries associated with $\HTT$ Lie groups are given in Table \ref{table_shilov_bnds_2} (see also Table 1 of \cite{guichard2018positivity}).

\begin{figure}[H]
\begin{center}
\begin{minipage}[b]{.5\linewidth}
\begin{center}
\begin{tikzpicture}
            \fill[fill=black] (0,0) circle (1pt);
            \draw node at (0.4,0) {$P$};
             \draw (-1.7,-1.7) -- (0,0) -- (-1.7,1.7);
             \draw (1.7,-1.7) -- (0,0) -- (1.7,1.7);   \draw (-1,1) arc (180:360:1 and 0.15);
          \draw[dashed] (1,1) arc (0:180:1 and 0.15); 
           \draw (-1,-1) arc (180:360:1 and 0.15);
          \draw[dashed] (1,-1) arc (0:180:1 and 0.15);
          \draw node at (0,1.6) {$\mathbf{I}^+(P) = \Diams_{\mathsf{std}}$};
          \draw node at (0,-1.6) {$\mathbf{I}^-(P) = \Diams_{\mathsf{std}}'$};
          
        \end{tikzpicture}
\caption{Past and future of $P$ in $\Affstd$ in the case where $\g = \soo(3,2)$}
    \label{figure_past_and_future}
\end{center}
\end{minipage} \hfill
\begin{minipage}[b]{.47\linewidth}
\begin{center}
\begin{tikzpicture}
          \draw (-2,1) -- (0,-1) -- (2,1);
          
          \draw (-2, -1) -- (0,1) -- (2,-1);
          \draw (-1,0) arc (180:360:1 and 0.15);
           \draw [fill, pattern = north east lines, opacity=0.7] (-1,0) -- (0,1) -- (1,0) -- (0,-1);
          
          \draw[dashed] (1,0) arc (0:180:1 and 0.15);
          \fill[fill=black] (0,1) circle (1pt);
          \fill[fill=black] (0,-1) circle (1pt);
          \draw node at (0.15,1.15) {$q$};
          \draw node at (0.15,-1.15) {$p$};
        \end{tikzpicture}
\caption{The diamond $\Diams(p,q)$ for $p,q \in \Affstd$ and $q \in \mathbf{I}^+(p)$ (greyed-out area), seen in $\Affstd \simeq \mathbb{R}^{2,1}$ for $\g = \soo(3,2)$.}
    \label{fig:enter-label}
\end{center}
\end{minipage}
\end{center}
\end{figure}

    \begin{table}[H]
        \centering
        \begin{tabular}{|c|c|}
        \hline
         $G$ &  $\Diams_{\mathsf{std}} \simeq $\\ 
         \hline
           $\SO(n,2)$    & $\mathbb{R} \times \mathbb{H}^{n}$  \\
           \hline
           $\Sp(2r , \mathbb{R})$ &  $\mathbb{R} \times (\SL(r, \mathbb{R})/ \SO(r))$  \\
           \hline
           $\SU(r,r)$ & $\mathbb{R} \times (\SL(r, \mathbb{C})/ \operatorname{SU}(r))$ \\
           \hline
            $\SO^*(4r)$ & $\mathbb{R} \times (\SL(r, \mathbb{H})/ \operatorname{Sp}(r))$ \\
           \hline
            $E_{7(-25)}$ & $\mathbb{R} \times (E_{6(-26)}/ F_4)$ \\
           \hline
        \end{tabular}
        \caption{The diamonds in $\SB(G)$ for every $\HTT$ Lie group (up to local isomorphism).} \label{table_shilov_bnds_2}
    \end{table}
    
\end{ex}

\section{Photons}\label{sect_photons_def}

In this section, we introduce the notion of \emph{photon} in the Shilov boundary $\SB(G)$ associated with an $\HTT$ Lie group $G$. The name refers to the well-known photons of the Einstein Universe $\Ein^{n-1, 1}$. We will use these generalized photons to define the Kobayashi pseudo-metric of a domain of $\SB(G)$ in Section~\ref{sect_dist_kob}.

\subsection{Definition and basic properties} Let $G$ be an $\HTT$ Lie group. Recall the element $h_r \in \aaa$ defined in Section~\ref{sect_plong_proj_line}. An $\mathfrak{sl}_2$-triple $\tr = (e,h,f)$ in $\g$ is said to be \emph{photon-generating} if $h$ is conjugate to $h_r$ in $G$, i.e.\ if there exists $g \in G$ such that~$h = \Ad(g) h_r$. The embedding $\plongsl_\tr$ defined in Section~\ref{sect_sl2-triples} induces an action of~$\SL_2(\mathbb{R})$ on $\SB(G)$, referred to as a \emph{photon-generating action}. This action has a unique closed orbit, which is a topological circle (whose elements are all the attracting fixed points of proximal elements of $\SL_2(\mathbb{R})$ for this action). 

\begin{definition}
    A \emph{photon} is the unique closed orbit of a photon-generating action of~$\SL_2(\mathbb{R})$ on $\SB(G)$. We denote by $\mathcal{L}$ the set of all photons of $\SB(G)$.
\end{definition}

We denote by $\Phot_{\mathsf{std}}$ the \emph{standard} photon defined by the standard $\mathfrak{sl}_2$-triple $\tr_{\mathsf{std}}$ of~\eqref{eq_def_v}. The map $\isl$ of~\eqref{eq_param_photons} is then a parametrization of $\Phot_{\mathsf{std}}$. We will mostly use the following equality in the rest of this paper:
\begin{equation*}
    \mathcal{L} = \{g \cdot \Phot_{\mathsf{std}} \mid g \in G\}.
\end{equation*}

The following lemma is a direct consequence of Lemma~\ref{lem_action_U_on_A}: 
\begin{lem}\label{lem_stab_u_std}
    One has $U^+ \cdot \Phot_{\mathsf{std}} = \Phot_{\mathsf{std}}$.
\end{lem}

\subsubsection{Parametrization of a photon} Given a photon $\Phot \in \mathcal{L}$, there exists $g \in G$ such that $\Phot = g \cdot \Phot_{\mathsf{std}}$. The map $\isl_g$ defined as
\begin{equation}\label{eq_param_photons_2}
    \isl_g : \begin{cases}
        \mathbb{P}(\mathbb{R}^2) &\longrightarrow \SB(G) \\
        x &\longmapsto g \cdot \isl(x)
    \end{cases}
\end{equation}
is then a parametrization of $\Phot$. A priori, this parametrization depends on the choice of~$g \in G$ such that $\Phot = g \cdot \Phot_{\mathsf{std}}$ (although its image does not). The next lemma shows that two parametrizations given by different choices of $g \in G$ such that $\Phot = \Ad(g) \cdot \Phot_{\mathsf{std}}$ only differ by a projective reparametrization of $\mathbb{P}(\mathbb{R}^2)$. 

\begin{lem}\label{lem_parametrization} One has 
\begin{equation*}
    \operatorname{Stab}_G (\Phot_{\mathsf{std}}) = \plong(\SL_2(\mathbb{R})) \times \operatorname{Cent}_G \left(\plong\left(\SL_2(\mathbb{R})\right)\right),
\end{equation*}
where $\operatorname{Stab}_G (\Phot_{\mathsf{std}})$ is the stabilizer of $\Phot_{\mathsf{std}}$ in $G$ and $\operatorname{Cent}_G \left(\plong\left(\SL_2(\mathbb{R})\right)\right)$ is the centralizer in $G$ of the group $\plong(\SL_2(\mathbb{R}))$.
\end{lem}
By the definition of $\Phot_{\mathsf{std}}$, the centralizer of $\plong(\SL_2(\mathbb{R}))$ acts trivially on $\Phot_{\mathsf{std}}$. Therefore Lemma~\ref{lem_parametrization} implies that $\operatorname{Stab}_G (\Phot_{\mathsf{std}})$ acts on $\Phot_{\mathsf{std}}$ by projective transformations. For any photon $\Phot$, the set $\Phot$ is a $G$-translate of $\Phot_{\mathsf{std}}$, so $\operatorname{Stab}_G (\Phot)$ acts on $\Phot$ by projective transformations.

\begin{proof}[Proof of Lemma~\ref{lem_parametrization}]
The inclusion $\plong(\SL_2(\mathbb{R})) \times \operatorname{Cent}_G \left(\plong\left(\SL_2(\mathbb{R})\right)\right) \subset \operatorname{Stab}_G (\Phot_{\mathsf{std}}) $ follows from the definition of $\Phot_{\mathsf{std}}$. Let us prove the converse inclusion.

Let $g \in \operatorname{Stab}_G (\Phot_{\mathsf{std}})$. First assume that $g P \in \Affstd$. Then, by~\eqref{eq_A_egal_exp}, we can write $g = \exp(X) \ell \exp(Y)$ with $Y \in \uu^+$, $X \in \uu^-$ and $\ell \in L$. Since $g$ stabilizes $\Phot_{\mathsf{std}}$, there exist~$\mu, \delta \in \mathbb{R}$ such that $\ell = \exp(\mu h_r)$ and $X = \delta v^-$. By Lemma~\ref{lem_action_U_on_A}, there exists $\lambda \in \mathbb{R}$ such that for all $t \in  \mathbb{R} \smallsetminus \left\{ -\lambda^{-1} \right\}$, one has
\begin{equation*}
    g \cdot \isl([1: t]) = \exp(X) \exp(\mu h_r)\cdot \isl\Big{(}\Big{[} 1: \frac{t}{1+ \lambda t}\Big{]}\Big{)} =  \exp(X) \exp(\mu h_r)\cdot \isl\left(e^{\lambda \F}\left[ 1: z\right]\right). 
\end{equation*}
Moreover, one has $\exp(X) = \plong \left(e^{\delta \E}\right)$ and $\ell = \plong \left(e^{\mu \He} \right)$. Then, by $\plong$-equivariance of $\isl$, one has $g \cdot \isl([1: t]) =  \isl\left(A\cdot  \left[ 1: t\right]\right)$
with $A = e^{\delta \F} e^{\mu \He} e^{\lambda \E} \in 
 \SL_2(\mathbb{R})$. By continuity, this equality holds for all $x \in \mathbb{P}(\mathbb{R}^2)$. In particular, the element $\plong(A)^{-1}g$ fixes every point of $\Phot_{\mathsf{std}}$. By the definition of $\Phot_{\mathsf{std}}$, this implies that $\plong(A)^{-1}g \in \operatorname{Cent}_G \left(\plong\left(\SL_2(\mathbb{R})\right)\right)$.

Now if $g P \notin \Affstd$, since $g$ preserves $\Phot_{\mathsf{std}}$ one must have $g P = \isl ([0:1])$. Since $\SL_2(\mathbb{R})$ acts transitively on $\mathbb{P}(\mathbb{R}^2)$, there exists $B \in \SL_2(\mathbb{R})$ such that  
$$gP = \isl([0:1]) = \isl(B \cdot [1:0]) = \plong(B) \isl([1: 0]),$$
so that $\plong(B)^{-1}g \in \operatorname{Stab}_{G}(\Phot_{\mathsf{std}})$ satisfies $\plong(B)^{-1}g P = P \in \mathbb{A}$. Then by the previous case, one has $\plong(B)^{-1}g \in \plong(\SL_2(\mathbb{R})) \times \operatorname{Cent}_G \left(\plong\left(\SL_2(\mathbb{R})\right)\right)$. Hence 
$$g = \plong(B) \plong(B)^{-1} g \in \plong(\SL_2(\mathbb{R})) \times \operatorname{Cent}_G \left(\plong\left(\SL_2(\mathbb{R})\right)\right) \quad \qed$$ 
\end{proof}

\subsubsection{Photons in the standard affine chart.}
The following lemma states that photons intersecting $\Affstd$ are compactifications of certain affine lines of $\Affstd$.

\begin{lem}\label{lem_photons_affine}
   Let $\Phot \in \mathcal{L}$. If $\Phot \cap \Affstd$ is nonempty, then it is an affine line in $\Affstd$, and~$\Phot \cap \hyp_{P^\opp}$ is a singleton.
\end{lem}
\begin{proof}[Proof]
Assume that $\Phot \cap \Affstd \ne \emptyset$. Let $g \in G$ be such that $\Phot = g \cdot \Phot_{\mathsf{std}}$. Since $\Phot$ is not contained in $\hyp_{P^\opp}$, there exists $t \in \mathbb{R}$ such that $g\exp(tv^-)P \in \Affstd$. Since $\exp(tv^-)$ stabilizes $\Phot_{\mathsf{std}}$, we have $g\exp(tv^-) \cdot \Phot_{\mathsf{std}}  = \Phot$. Hence, to replacing $g$ with $g\exp(tv^-)$, we may assume that $gP \in \Affstd$. 

By~\eqref{eq_A_egal_exp}, one can thus write $g = \exp(X) \ell \exp(Y)$ with $X \in \uu^-$, $Y \in \uu^+$ and $\ell \in L$. By Lemma~\ref{lem_stab_u_std}, one then has $g \cdot \Phot_{\mathsf{std}} = \exp(X) \ell \cdot \Phot_{\mathsf{std}}$. Hence 
\begin{equation}\label{eq_affine_line_in_u}
    \varphi_{\mathsf{std}}^{-1}\left(\Phot \cap \Affstd\right) = \varphi_{\mathsf{std}}^{-1}\left((g \cdot \Phot_{\mathsf{std}}) \cap \Affstd\right) = X + \Ad(\ell) \g_{- \alpha_r} \subset \uu^-
\end{equation}
is an affine line of $\uu^-$. Hence $\Phot \cap \Affstd$ is an affine line of $\Affstd$ for the canonical affine structure. 

Moreover, the map $\isl_{g'}$, with $g' = \exp(X) \ell$, is a parametrization of $\Phot$. By~\eqref{eq_affine_line_in_u}, one has $\isl_{g'} ([1:t]) = \exp(X + t\Ad(\ell)  v^-)P \in \mathbb{A}$ for all $t \in \mathbb{R}$. Hence the only point of~$\Phot \cap \hyp_{P^\opp}$ is $\isl_{g'}([0:1])$, and $\Phot \cap \hyp_{P^\opp}$ is a singleton. $\qed$
\end{proof}

Lemma~\ref{lem_photons_affine} implies that for all $\xi \in \SB(G)$ such that $\Phot_{\mathsf{std}} \not\subset \hyp_{\xi}$, the set~$\Phot_{\mathsf{std}} \cap \hyp_{\xi}$ is a singleton.

\begin{rmk}\label{rmk_photons_included_in_cone}
(1) It follows directly from the definition of $c^0$ that $(\Phot \cap \Affstd) \subset \mathbf{C}(x)$ for any $x \in \Affstd$ and $\Phot \in \mathcal{L}$ such that $x \in \Phot$. Moreover, given a point $x \in \SB(G)$, we have the equivalence: $\Phot_{\mathsf{std}} \subset \hyp_{x}$ if, and only if, $x \in \Phot_{\mathsf{std}}$.

(2) More generally, given any $\mathfrak{sl}_2$-triple of $\g$, the unique closed orbit of the induced action of $\SL_2(\mathbb{R})$ on $\SB(G)$ is a topological circle, sometimes simply called a \emph{circle}. In \cite{labourie2021positivity}, \emph{positive circles} are defined from triples $\tr = (e,h,f)$ where $f$ is conjugate to an element of $c^0$. Here we state their definition only in the context of Shilov boundaries of $\HTT$s, but positive circles are defined in any flag manifold admitting a $\Theta$-positive structure and are particular \emph{magical $\mathfrak{sl}_2$-triples} \cite{bradlowgeneral}.

(3) The definition of a positive circle \cite{labourie2021positivity} is in some sense opposite to that of a photon, and its properties are opposite to those of photons (for instance, the centralizer of the associate embedding of $\SL_2(\mathbb{R})$ into $G$ is compact).

(4) Photons have already been defined in \cite{van2019rigidity} for the Grassmannians $\operatorname{Gr}_p(\mathbb{R}^n)$, where they are called \emph{rank-one lines}. For the natural identification of $\Lag_r(\mathbb{R}^{2r})$ with a submanifold of $\operatorname{Gr}_r(\mathbb{R}^{2r})$, the photons of $\Lag_r(\mathbb{R}^{2r})$ are exactly the rank-one lines of~$\operatorname{Gr}_r(\mathbb{R}^{2r})$ that are contained in $\Lag_r(\mathbb{R}^{2r})$.
\end{rmk}

\subsection{Intersection polynomials}\label{sect_inter_poly} In this section we define the \emph{intersection polynomials}, which algebraically describe the intersection between the standard photon $\Phot_{\mathsf{std}}$ and the nontransverse set $\hyp_{\xi}$ of a point $\xi \in \SB(G)$ such that $\Phot_{\mathsf{std}} \not\subset \hyp_{\xi}$. 

Recall the notation of Sections~\ref{sect_irred_semi} and~\ref{sect_irred_htt}. Fix a finite-dimensional real irreducible linear representation $(V, \rho)$ of $G$ with highest weight $\chi = N \omega_{r}$ for some $N \in \mathbb{N}_{> 0}$. We let $\iota: \SB(G) \hookrightarrow \mathbb{P}(V)$ and $\iota^*: \SB(G) \hookrightarrow \mathbb{P}(V^*)$ be the two embeddings induced by $\rho$ by Proposition~\ref{prop_ggkw}, and we fix a vector $e_1 \in V^{\chi} \smallsetminus \{0\}$. By Lemma~\ref{lem_degree_rep} and Equation~\eqref{eq_param_photons}, we have, for $t \in \mathbb{R}$,
\begin{equation}\label{forme_de_la_ddr1_2}
    \iota \circ \isl ([1:t])
=  [e^{t \rho_*(v^-)}\cdot e_1]
    = \big{[}e_1 + t \rho_*(v^-)e_1 + \cdots + \frac{t^N}{N!} \rho_*(v^-)^N e_1 \big{]} 
\end{equation}

Let us define the dense open subset
\begin{equation*}
    \Contlam := \SB(G) \smallsetminus \Phot_{\mathsf{std}}
\end{equation*}
of $\SB(G)$. By the ``moreover" part of Remark~\ref{rmk_photons_included_in_cone}.(1), we have $\Contlam = \left\{ \xi \in \SB(G) \mid \Phot_{\mathsf{std}} \not\subset \hyp_{\xi}\right\}$. Given $\xi \in \Contlam$, we choose any lift $f \in V^* \smallsetminus \{0\}$ of $\iota^*(\xi) \in \mathbb{P}(V^*)$. Since $\Phot_{\mathsf{std}} \not\subset \hyp_{\xi}$, by Proposition~\ref{prop_ggkw}, the polynomial defined by
\begin{equation*}
    f\big{(}e^{t \rho_*(v^-)} \cdot e_1 \big{)} = f(e_1) + t f(\rho_*(v^-)e_1) + \cdots + t^N f(\rho_*(v^-)^N e_1) \quad \forall t \in \mathbb{R}
\end{equation*}
is nonzero. Then there exists a maximal $0 \leq n(\xi) \leq N$ such that 
$f(\rho_*(v^-)^{n(\xi)} e_1) \ne 0$,  
and $n(\xi)$ does not depend on the choice of the lift $f$ of $\iota^*(\xi)$. Hence we may choose~$f$ such that $f(\rho_*(v^-)^{n(\xi)} e_1) = n(\xi)!$. This defines a polynomial $\mathsf{Q}_\xi^\rho$ with coefficients in~$\mathbb{R}$, depending only on $\xi$ and $(V, \rho)$:
\begin{equation}\label{eq_inter_pol}
  \mathsf{Q}_\xi^{\rho}(t) = f\big{(}e^{t \rho_*(v^-)} \cdot e_1 \big{)} = f(e_1) + t f(\rho_*(v^-)e_1) + \cdots + t^{n(\xi)} \quad \forall t \in \mathbb{R}.
\end{equation}

\begin{definition}
    Given a point $\xi \in \Contlam$, the polynomial $\mathsf{Q}_\xi^{\rho}$ defined in~\eqref{eq_inter_pol} is called the \emph{intersection polynomial of $\xi$ associated with the representation $(V, \rho)$}.
\end{definition}
Now let 
\begin{equation}\label{eq_A_rho}
    \mathcal{A}_{\rho} := \left\{\xi \in \SB(G) \mid  [\rho_*(v^-)^N e_1] \notin \iota^*(\xi) \right\} \subset \Contlam
\end{equation}
be the set of all elements of $\Contlam$ such that $n(\xi)$ is maximal.

\begin{lem}\label{lem_A_dense}
    The set $\mathcal{A}_\rho$ is open and dense in $\SB(G)$.
\end{lem}

\begin{proof}[Proof]
By irreducibility of $\rho$, for any open set $\mathcal{O} \subset \SB(G)$, there exist $x_1, ..., x_D \in \mathcal{O}$ such that $\iota^*(\xi_1) \oplus ... \oplus \iota^*(\xi_D) = V^*$ (see e.g.\ \cite[Lem.\ 4.7]{zimmer2018proper}). If the set 
$$F :=\{ \xi \in \SB(G) \mid [\rho_*(v^-)^N e_1]  \in \iota^*(\xi)\}$$
had nonempty interior, then we would have $f(\rho_*(v^-)^N e_1) = 0$ for all $f \in V^*$. This is absurd because $\rho_*(v^-)^N e_1 \ne 0$ by Lemma~\ref{lem_degree_rep}, so $F$ has empty interior and $\mathcal{A}_{\rho}$ is a dense open subset of $\SB(G)$. $\qed$
\end{proof}

Let $\xi \in \Contlam$ and $t \in \mathbb{R}$. One has 
\begin{align*}
    \mathsf{Q}_{\xi}^\rho(t) = 0 \ &\Longleftrightarrow \ f\big{(}e^{t \rho_*(v^-)} \cdot e_1 \big{)}\Longleftrightarrow \  \iota(\exp(tv^-)P) \in \iota^*(\xi) \\
    &\Longleftrightarrow \ \exp(tv^-)P \in \hyp_{\xi} \Longleftrightarrow \ \isl([1:t]) = \exp(tv^-)P \in \Phot_{\mathsf{std}} \cap \hyp_{\xi},
\end{align*}
the last equivalence holding by~\eqref{eq_param_photons} and Proposition~\ref{prop_ggkw}. Hence the real roots of $\pol_\xi^\rho$ describe the intersection points of $\Phot_{\mathsf{std}} \cap \Affstd$ with $\hyp_{\xi}$. By Lemma~\ref{lem_photons_affine}, the set~$\Phot_{\mathsf{std}} \cap \hyp_{\xi}$ is a singleton, so the polynomial $\mathsf{Q}_\xi^\rho$ has at most one real root $t$, satisfying $\isl([1:t]) \in  \hyp_{\xi}$. We will see in Section~\ref{sect_comp_cara_kob} that the complex roots of $\pol_\xi^\rho$ also describe the intersection of two sets, corresponding to complexifications of $\Phot_{\mathsf{std}} \cap \Affstd$ and $\hyp_\xi$. This is why we call the polynomial $\mathsf{Q}_\xi^\rho$ the intersection polynomial of $\xi$.

\section{An invariant metric}\label{sect_dist_kob}
In this section we define the \emph{Kobayashi metric} of a proper domain $\O$ of $\SB(G)$. Constructions of Kobayashi metrics are classical and were initiated by S. Kobayashi \cite{kobayashi1967invariant, kobayashi1984projectively}. The properties of the Kobayashi metric (in particular, its properness, see Corollary~\ref{cor_kobayashi_geodesic}) in the almost-homogeneous case will allow us to relate the geometry of the boundary of a proper almost-homogeneous domain to the dynamics of its automorphism group in Section~\ref{sect_dyn_bndr}.

\subsection{A Kobayashi metric}
In this section, the group $G$ is an $\HTT$ Lie group.
\subsubsection{Chains.} Let $\O \subset \SB(G)$ be a domain, not necessarily proper. We say that two points $x_1,x_2 \in \O$ are \emph{conjugate}, denoted by $x_1 \sim x_2$, if they belong to the same photon~$\Phot$ and are in the same connected component of $\Phot \cap \O$. 

Now let $x,y \in \O$ be any two points. An \emph{$N$-chain from $x$ to $y$} ($N\in\mathbb{N}$) is a sequence of $(N+1)$ elements $(x_0 =x, \cdots , x_N = y)$ of $\O$ such that $x_i \sim x_{i+1}$ for all $0 \leq i \leq N-1$. We denote by $\mathcal{C}_{x,y}(\O)$ (resp.\ $\mathcal{C}_{x,y}^N(\O)$) the set of  all the chains (resp.\ $N$-chains) from $x$ to $y$ in $\O$.

In this section, we prove that there is a bound $\nG$ depending only on $G$ such that, locally, one can always join two points $x,y \in \SB(G)$ by an $\nG$-chain.

\begin{lem}\label{lem_existence_N_uniform}
There is an integer $n \geq 2$ only depending on $G$, such that for any two transverse points $p,q \in \SB(B)$ and any diamond $D \subset \SB(G)$ with endpoints $p$ and~$q$, there is a sequence $(x_0 = p, x_1, \cdots, x_{n} = q)$ of $(n+1)$ elements of $\overline{D}$ such that $x_i$ and~$x_{i+1}$ are on the same photon for all $0 \leq i \leq n-1$. 
\end{lem}
\begin{proof}[Proof]
\indent Let $n$ be the smallest natural integer such that there exist $\ell_1,..., \ell_{n} \in L^0$ such that $\sum_{k=1}^{n} \Ad(\ell_k) \cdot v^- \in c^0$. The integer $n$ only depends on $G$. Define 
\begin{equation*}
    x_i = \begin{cases}
        \ \ \ P &\text{ if } i=0; \\
        \exp \big{(}\sum_{k=1}^i \Ad(\ell_k) \cdot v^- \big{)} P &\text{ if } 1 \leq i \leq n.
    \end{cases}
\end{equation*}
For $i \geq 1$, one has $x_i \in \mathbf{J}^+(x_{i-1})$, so by transitivity (Fact~\ref{lem_causality_order_relation}), one has $x_i \in \mathbf{J}^+(x_0)$. In particular, one has $x_n \in \mathbf{I}^+(x_0)$. Hence the set $D = \mathbf{I}^+(x_0) \cap \mathbf{I}^-(x_n)$ is a diamond. By construction, for any $0 \leq i \leq n-1$, the points $x_i$ and $x_{i+1}$ are on the same photon, and all the $x_i$ are contained in $\overline{D}$. 

Since $G$ acts transitively on diamonds and elements of $G$ send photons on photons, the integer $n$ is the one sought. $\qed$
\end{proof}

\begin{lem}\label{lem_paths_diamonds} There is an integer $\nG \geq 2$ depending only on $G$, such that for any diamond $D \subset \SB(G)$, the set $\mathcal{C}^{\nG}_{x,y}(D)$ is nonempty for any $x,y \in D$.
\end{lem}

\begin{proof}[Proof] We may assume that $D$ is of the form $D = \Diams(P,q)$, with~$q \in \mathbf{I}^+(P)$; in particular, the domain $D$ is a convex subset of $\Affstd$.

Let $z \in D$ be such that $x,y \in \mathbf{I}^+(z)$. Then $\overline{\Diams(z,x)}$ and $\overline{\Diams(z,y)}$ are both contained in $D$. Let $n$ be the integer associated with $G$ by Lemma~\ref{lem_existence_N_uniform}. Then there exists a sequence $u = (x_0 = z, \cdots , x_{n} = x)$ of $(n+1)$ elements of $\overline{\Diams(z,x)} \subset D$ such that~$x_i$ and~$x_{i+1}$ are conjugate. By convexity of $D$ in $\Affstd$, the sequence $u$ is an $n$-chain from~$z$ to $x$ in $D$. Similary, we get an $n$-chain from $z$ to $y$. Concatenating them both in $z$ gives a $2n$-chain from $x$ to $y$. Let $\nG := 2n$. This number only depends on $G$. $\qed$
\end{proof}

\begin{rmk}\label{rmk_N_non_minimal}
    The number $\nG$ found in the proof of Lemma~\ref{lem_paths_diamonds} is not optimal. Indeed, for example when $G = \SO(n,2)$, the number $\nG = 2$ works; see \cite{chalumeau2024rigidity}.
\end{rmk}

\begin{figure}[H]
    \centering
        \begin{tikzpicture}[scale=1.3]
        \draw[dashed] (1,0) arc (0:180:1 and 0.15);
        \draw[pink,thick] (0,-0.5) -- (-0.5,-0.05) -- (-0.1,0.3);
        \draw[blue,thick] (0,-0.5) -- (0.46,-0.04) -- (0.2,0.2);
          \fill[fill=white] (-0.06,-0.15) circle (1.5pt);
          \draw (-1,0) -- (0,-1) -- (1,0) -- (0,1) -- (-1,0);
          \draw (-1,0) arc (180:360:1 and 0.15);
          \fill[fill=black] (0,1) circle (1pt);
          \fill[fill=black] (0,-1) circle (1pt);
          \fill[fill=black] (0,-0.5) circle (1pt);
          \fill[fill=black] (-0.1,0.3) circle (1pt);
          \fill[fill=black] (0.2,0.2) circle (1pt);
          \fill[fill=black] (0,-0.5) circle (1pt);
          \draw node at (0.15,1.15) {$q$};
          \draw node at (0.15,-1.15) {$p$};
          \draw node at (0.15,-0.65) {$z$};
          \draw node at (-0.2,0.4) {$x$};
          \draw node at (0.35,0.3) {$y$};
        \end{tikzpicture}
    \caption{Example of path as constructed in the proof of Lemma~\ref{lem_paths_diamonds}, when $G = \SO(3,2)$. In this case, the number $\nG$ is not minimal; see Remark~\ref{rmk_N_non_minimal}.}
    \label{Geodésiques du Diams}
\end{figure}

\subsubsection{The pseudo-metric.} Let $x,y \in \O$ be two conjugate points. Let $\Phot_{x,y} \in \mathcal{L}$ be a photon containing $x$ and $y$, and let $g \in G$ be such that $\Phot_{x,y} = g \cdot \Phot_{\mathsf{std}}$. We denote by $I_{x,y}$ the connected component of $\O \cap \Phot_{x,y} $ containing $x$ and $y$. Recall the parametrization $\isl_g : \mathbb{P}(\Rf^2) \rightarrow \Phot_{x,y}$ of~\eqref{eq_param_photons_2}. We define: 
\begin{equation*}
    \ro_{\O}(x,y) := H_{\isl_g^{-1}\left(I_{x,y} \right)} \left(\isl_g^{-1}(x), \isl_g^{-1}(y) \right).
\end{equation*}
Recall that we denote by $H_I$ the Hilbert pseudo-metric of an interval $I$ of $\mathbb{P}(\mathbb{R}^2)$ (see Section~\ref{sect_cross_ratio}). Due to the $\SL_2(\mathbb{R})$-invariance of the cross ratio on $\mathbb{P}(\mathbb{R}^2)$ and Lemma~\ref{lem_parametrization}, the quantity $\ro_{\O}(x,y)$ does not depend on the choice of $g \in G$ such that $\Phot_{x,y} = g \cdot \Phot_{\mathsf{std}}$. 

\begin{definition}\label{def_kob_map}
Given a domain $\O \subset \SB(G)$, we define $\kob_{\O} : \O \times \O \rightarrow \mathbb{R}_+ \cup \{ + \infty\}$ by:
\begin{equation*}
    \forall x,y \in \O, \quad \kob_{\O}(x,y) = \inf \left\{ \sum_{i=0}^N \ro_{\O}(x_i, x_{i+1}) \mid N \in \mathbb{N}^*, \ \ (x_0, \cdots x_N) \in \mathcal{C}_{x,y}(\O) \right\}.
\end{equation*}
\end{definition}

For $x,y$ sufficiently close to each other, we can find a diamond included in $\O$ containing $x$ and $y$; see Remark~\ref{rmk_diamonds_basis_of_neighborhoods}.(1). In that case $\mathcal{C}_{x,y}(\O)$ (and even $\mathcal{C}^{\nG}_{x,y}(\O)$) is nonempty by Lemma~\ref{lem_paths_diamonds}. Since the relation ``$x$ and $y$ can be joined by a chain'' is an equivalence relation, the set $\mathcal{C}_{x,y}(\O)$ is never empty. Since $\kob_{\O}(a,b)$ is always finite whenever $a,b$ are conjugate, we know by Lemma \ref{lem_paths_diamonds} that the quantity $\kob_{\O}(x,y)$ is thus always finite as well. Thus $\kob_{\O}$ is actually a map $\O \times \O \rightarrow \mathbb{R}_{+}$.

We have the following standard naturality property:

\begin{prop}\label{prop_properties_kobayashi_metric}
Let $\O_1 $ and $ \O_2$ be two domains of $\SB(G)$. Then:
\begin{enumerate}
    \item If $\O_1 \subset \O_2$, then for any $x,y \in \O_1$ one has $\kob_{\O_2}(x,y) \leq \kob_{\O_1}(x,y)$.
    \item For any $g \in G$, for any $x,y \in \O_1$, one has $\kob_{g \cdot \O_1}(g \cdot x, g \cdot y) = \kob_{\O_1}(x,y)$. In particular, the metric $\kob_{\O_1}$ is $\Aut(\O_1)$-invariant.
\end{enumerate}
\end{prop}

\begin{proof}[Proof]
This is a consequence of the definition of $K_{\O}$ and of the fact that an element~$g$ of~$G$ sends a chain from $x$ to $y$ to a chain from $g \cdot x$ to $g \cdot y$. $\qed$
\end{proof}

Note that $\kob_{\O_2}$ and $\kob_{\O_1}$ do not need to be metrics in Proposition~\ref{prop_properties_kobayashi_metric}. 

As one can concatenate and reverse the orientation of a chain, the map $\kob_{\O}$ is symmetric and satisfies the triangle inequality. It is thus a pseudo-metric, and we call it the \emph{Kobayashi pseudo-metric}. In the next section, we investigate when $\kob_{\O}$ is a metric.

\subsubsection{Kobayashi hyperbolicity} Let $\O \subset \SB(G)$ be a domain. The domain $\Omega$ is said to be \emph{Kobayashi hyperbolic} if $\kob_{\O}$ is a metric, that is, if $\kob_{\O}$ separates points.

\begin{definition}
    Let $\O \subset \SB(G)$ be a Kobayashi-hyperbolic domain. Then the map~$\kob_{\O}$ defined in Definition~\ref{def_kob_map} is called the \emph{Kobayashi metric} of $\O$.
\end{definition}

\begin{prop}\label{prop_generate_topo_std}
    Proper domains of $\SB(G)$ are Kobayashi hyperbolic. Moreover, for such a domain $\O$, the metric $\kob_{\O}$ generates the standard topology.
\end{prop}

\begin{proof}[Proof]
Let us first show that a properly convex domain $C$ of $\Affstd$ is Kobayashi hyperbolic, where $\Affstd$ is endowed with its affine structure. Since $C$ is a properly convex domain of~$\Affstd$, it has a classical Hilbert metric $\mathsf{H}_C$ (see e.g.\ \cite{goldman2022geometric}).
By the definition of $\mathsf{H}_C$, if~$a,b \in C$ are conjugate then one has 
\begin{equation}\label{eq_hilbert_metric_convex_of_A}
    \mathsf{H}_C(a,b) =\ro_{C}(a,b).
\end{equation}
Now let $x,y \in C$ distinct and $\gamma = (x_0,\dots, x_N) \in \mathcal{C}_{x,y}(\O)$ be a chain from $x$ to $y$. Since~$\mathsf{H}_C$ satisfies the triangle inequality, one has:

$$ \sum_{i=0}^{N-1} \ro_{C}(x_i , x_{i+1}) = \sum_{i=0}^{N-1} \mathsf{H}_C(x_i , x_{i+1}) \geq \mathsf{H}_C(x,y).$$
This is true for all $\gamma \in \mathcal{C}_{x,y}(\O)$, so by taking the infinimum we get 
\begin{equation}\label{eq_comp_kob_hilb}
    \kob_C (x,y) \geq \mathsf{H}_C (x,y) > 0.
\end{equation}
Therefore $\kob_C$ separates the points and $C$ is Kobayashi hyperbolic. 

Now let $\O$ be any domain. We may assume that $\O$ is proper in $\Affstd$ (Remark~\ref{Rmk_contenu_dans_A_std}). Let~$C$ be any properly convex domain of $\Affstd$ containing $\O$ (for instance the convex hull of $\O$ in $\Affstd$). Then, by Proposition~\ref{prop_properties_kobayashi_metric}:
\begin{equation}\label{eq_comp_Hilbert_kob}
    \kob_C(x,y) \leq \kob_{\O}(x,y) \quad \forall x,y \in \O.
\end{equation}
Since $\kob_C$ separates the points, so does $\kob_{\O}$. \\
\indent Let us show that $\kob_{\O}$ generates the standard topology. Let $\mathcal{T}_{\mathsf{std}}$ be the standard topology on $\O$ and $\mathcal{T}$ the topology induced by $\kob_{\O}$. We will prove that~$\mathcal{T}_{\mathsf{std}} = \mathcal{T}$. By Equations \eqref{eq_comp_kob_hilb} and \eqref{eq_comp_Hilbert_kob} and the fact that the Hilbert metric generates the standard topology on properly convex domains, one has $\mathcal{T}_{\mathsf{std}}\subset \mathcal{T}$. To prove the reverse inclusion, one need to show that $\kob_{\O}$ is continuous with respect to the standard topology. By the inequality
\begin{equation*}
    |\kob_{\O}(x_0, y_0 ) - \kob_{\O}(x, y)| \leq \kob_{\O}(x_0, x ) + \kob_{\O}(y_0, y) \quad \forall x_0, y_0, x, y \in \O,
\end{equation*}
one only needs to show that for any $x_0 \in \O$ the map $x \mapsto \kob_{\O}(x_0, x)$ is continuous at $x_0$. For this we see $\Affstd$ as an Euclidean space and  
we denote by $||.|| $ the associated Euclidean norm. Up to dilating at $x_0$, we may assume that the Euclidean ball $B$ of center $x_0$ and of radius $1$ is contained in $\O$. Let $D$ be a diamond containing $x_0$ and contained in $B$, with diameter $\delta \in ]0, 1[$ for the Euclidean norm (Remark~\ref{rmk_diamonds_basis_of_neighborhoods}.(1)). Let $N := \nG$ given by Lemma~\ref{lem_paths_diamonds}. There exists an $N$-chain $(x_0, x_1, \cdots , \ x_{N} = x)$ contained in $D \subset B$. Then, by Proposition~\ref{prop_properties_kobayashi_metric} and Equation~\eqref{eq_hilbert_metric_convex_of_A}, one has
\begin{align*}
  \kob_{\O}(x,x_0) &\leq \kob_B (x_0, x) 
  \leq \sum_{k=0}^{N-1} \kob_B(x_i, x_{i+1}) = \sum_{k=0}^{N-1} \mathsf{H}_B(x_i, x_{i+1}) \\
  &\leq \sum_{k=0}^{N-1} \mathsf{H}_B(x_0, x_{i}) + \mathsf{H}_B(x_0, x_{i+1}) \\
  &= \sum_{k=1}^{N-1}  \log \frac{1 +||x_1-x_0|| }{1 - ||x_i-x_0||} + \log \frac{1+||x_{i+1}-x_0|| }{1 -||x-x_0||} \leq  \sum_{k=1}^{N-1}  2\log \frac{1 + \delta }{1 - \delta} \ \underset{\delta \rightarrow 0}{\longrightarrow} 0
\end{align*}
This proves that $\kob_{\O}(x_0, x ) \rightarrow 0$ as $x \rightarrow x_0$.
$\qed$
\end{proof}

\begin{rmk}\label{rmk_kob_tends_zero}
    The proof of Proposition~\ref{prop_generate_topo_std} gives that for any two sequences $(x_k), (y_k) \in \O^{\mathbb{N}}$ such that $x_k \rightarrow p \in \partial \overline{\O}$ and $y_k \rightarrow q \in \overline{\O}$, if $\kob_{\O}(x_k, y_k) \rightarrow 0$ then one has $p=q$.
\end{rmk}

\subsubsection{Kobayashi length. }\label{section_chains_length} In this section, we recall some definitions and fix some notation. Let $\O$ be a Kobayashi hyperbolic domain, contained in $\Affstd$. For a continuous path~$\gamma:[0,1] \rightarrow \O$, we define the \emph{Kobayashi length} or \emph{$K_{\O}$-length} of $\gamma$ in the the usual way, as
\begin{equation*}
    \leng(\gamma) = \sup \sum_{i=0}^N \kob_{\O}(\gamma(t_i),\gamma(t_{i+1})),
\end{equation*}
where the supremum is taken over all finite subdivisions of $\gamma$. 

Let $x,y \in \O$ be two conjugate points, and let $\Phot_{x,y}$ be the unique photon containing $x$ and $y$. We denote by $[x,y]$ the closure of the only connected component of $\Phot_{x,y} \smallsetminus \{x,y\}$ that is contained in $\O$. By Lemma~\ref{lem_photons_affine}, it is an affine segment in $\Affstd$. This segment can be parametrized by
\begin{equation*}
    [t_1,t_2] \longrightarrow [x,y]; 
    \ t  \longmapsto \isl_g([1:t]),
\end{equation*}
where $g \in G$ is such that $\Phot_{x,y} = g \cdot \Phot_{\mathsf{std}}$ and $x = \isl_g([1:t_1])$ and $y = \isl_g([1:t_2])$. This parametrization depends on the choice of $g \in G$. 

Now, let $x, y \in \O$ be any two points. Any element of $u = (x_0, \cdots, x_{N}) \in \mathcal{C}_{x,y}(\O)$ gives rise to a continuous path $\gamma$ from $x$ to $y$, defined as the concatenation of all the segments $[x_0, x_1], \cdots , [x_{N-1}, x_N]$ in this order, endowed with a parametrization as described above. This path is uniquely defined by $u$ up to reparametrization.
The $K_{\O}$-length $\leng(\gamma)$ of $\gamma$ does not depend on the choice of parametrization of the $[x_i, x_{i+1}]$ for $0 \leq i \leq N$. This defines a unique $K_{\O}$-length for the chain $(x_0, \cdots, x_{N})$. 

In the rest of the paper, we will identify a chain with the unique (up to parametrization) path it defines by the process described above. In particular, this will allow us to consider the $K_{\O}$-length of a chain.

\subsection{Comparison with the Caratheodory metrics}\label{section_carat_metric}

The aim of this section is to prove Proposition~\ref{lem_key_lemma1}. The results we prove are actually stronger; see Proposition~\ref{lem_key_lemma} and Corollary~\ref{cor_kobayashi_geodesic}.

\subsubsection{The Caratheodory metric} \label{sect_def_cara} Let $G$ be a real semisimple Lie group and $\Theta \subset \FS$ be a subset of the simple roots. Let $(V, \rho)$ be a $\Theta$-proximal finite-dimensional real irreducible linear representation of $G$ (see Section~\ref{sect_irred_semi}). Let $\iota: \Fl_{\Theta} \hookrightarrow \mathbb{P}(V)$ and $\iota^*: \Fl_{\Theta}^\opp \hookrightarrow \mathbb{P}(V^*)$ be the two embeddings induced by $(V, \rho)$, see Proposition~\ref{prop_ggkw}. Given $x,y \in \Fl_{\Theta}$ and~$\xi, \eta \in \Fl_{\Theta}^{\opp}$, let
\begin{equation*}
    \nu_x \in \iota(x) \smallsetminus \{ 0 \}; \quad \nu_y \in \iota(y) \smallsetminus \{ 0 \}; \quad f_{\xi} \in \iota^*(\xi) \smallsetminus \{ 0 \}; \quad f_{\eta} \in \iota^*(\eta) \smallsetminus \{ 0 \}.
\end{equation*}
We define the \emph{cross ratio} of $\xi, x, y, \eta$ \emph{relative to} $(V, \rho)$ as follows:
\begin{equation}\label{eq_cross_ratio}
    \left[\xi : x : y: \eta \right]_{\rho} :=   \frac{f_{\xi}(\nu_x)f_{\eta}(\nu_y)}{f_{\xi}(\nu_y)f_{\eta}(\nu_x)}.
\end{equation}
This quantity does not depend on the choice of representatives $\nu_x, \nu_y, f_{\xi}, f_{\eta}$. In \cite{zimmer2018proper}, A. Zimmer introduces the following map $C_{\O}$ associated with a domain $\O \subset \Fl_{\Theta}$: 
\begin{equation*}\label{Cara}
    C_{\O}: 
    \begin{cases}
    \O \times \O &\longrightarrow \quad \quad \mathbb{R}_+ \\
    (x,y ) &\longmapsto \sup_{\xi, \eta \in \O^*} \log \big{|} \left[\xi : x : y :\eta \right]_{\rho} \big{|}.
    \end{cases}
\end{equation*}

By \cite[Thm 5.2 and 9.1]{zimmer2018proper}, the map $C_{\O}$ is an $\Aut(\O)$-invariant metric generating the standard topology, as soon as $\O$ is a proper domain of $\Fl_{\Theta}$. Whenever this is the case, we will say that $C_{\O}$ is the \emph{Caratheodory metric} on $\O$ \emph{induced by $(V, \rho)$}. 

In \cite{zimmer2018proper}, this metric is used to prove that quasi-homogeneous domains are dually convex. In this paper, we need the slightly stronger following result, whose proof relies on the one of \cite[Cor.\ 9.3]{zimmer2018proper}:

\begin{prop}\label{prop_zimmer_dual_convex}
Any proper almost-homogeneous domain of $\Fl_{\Theta}$ is dually convex.
\end{prop}

\begin{proof}
  By definition, we have $\O \subset \O^{**}$ and $\Aut(\O) \leq \Aut(\O^{**})$. Hence the connected component of $\O^{**}$ that contains $\O$ is $\Aut(\O)$-invariant. Since it is also proper, by Lemma~\ref{lem_inclusion_equality_almost_homogeneous}, we get that $\O$ is equal to this connected component. It is thus dually convex (see Remark~\ref{rmk_bidual_dual_convex}.(2)). $\qed$
\end{proof}

\subsubsection{Comparison for Shilov boundaries. }\label{sect_comp_cara_kob} The goal of this section is to prove Proposition~\ref{lem_key_lemma} below. Together with Corollary~\ref{cor_kobayashi_geodesic} in the next section, it will imply Proposition~\ref{lem_key_lemma1}.

\begin{prop}\label{lem_key_lemma} Let $G$ be a simple Lie group of Hermitian type of tube type and let~$(V, \rho)$ be a finite-dimensional real irreducible linear representation of $G$ with highest weight $\chi = N \omega_{\alpha_r}$ ($N \in \mathbb{N}_{>0}$). Let $\O \subset \SB(G)$ be a proper dually convex domain, and let $C_{\O}$ be the Caratheodory metric on $\O$ induced by $(V, \rho)$. Then for any two conjugate points $x,y \in \O$, one has
    \begin{equation*}
        \ro_\O (x,y) = \frac{1}{N}C_{\O}(x,y).
    \end{equation*}
    In particular,  
    \begin{enumerate}
        \item one has $\kob_{\O} \geq N^{-1}C_{\O}$;
        \item the metric $\kob_{\O}$ is a length metric, and given two conjugate points $x,y \in \O$, the $K_{\O}$-length of the $1$-chain $(x,y)$ is equal to $\ro_{\O}(x,y) = K_{\O}(x,y)$. 
    \end{enumerate}
\end{prop}

Let us fix once and for all some notation for the rest of this section. Let $G$ be an $\HTT$ Lie group. Let $(V, \rho)$ be a finite-dimensional real irreducible linear representation of $G$ with highest weight $\chi = N \omega_{r}$ for some $N \in \mathbb{N}_{> 0}$. We let $\iota: \SB(G) \hookrightarrow \mathbb{P}(V)$ and~$\iota^*: \SB(G) \hookrightarrow \mathbb{P}(V^*)$ be the two embeddings induced by $\rho$ from Proposition~\ref{prop_ggkw}, and we fix a vector $e_1 \in V^{\chi} \smallsetminus \{0\}$.

The key point of the proof of Proposition~\ref{lem_key_lemma} is the comparison of two cross ratios: the one defined in Equation~\eqref{eq_cross_ratio} and the one given in the definition of $\ro_{\O}$. This is done in Lemma~\ref{lem_comp_crossratios} below. The argument uses the intersection polynomials introduced in Section~\ref{sect_inter_poly}. We have seen that the real roots of an intersection polynomial $\mathsf{Q}_{\xi}^\rho$ for~$\xi \in \Contlam$ are geometrically described by the intersection points of $\iota(\Phot_{\mathsf{std}})$ with $\iota^*(\xi)$, so that they have at most one real root by Lemma~\ref{lem_photons_affine}. But our argument will rely on the fact that they have only one \emph{complex} root, so that they are split (see Corollary~\ref{cor_inter_split}). The complex roots of $\mathsf{Q}_\xi^\rho$ will be geometrically described as the intersection points of the complexification of $\iota^*(\xi)$ and a set, denoted $\mathcal{P}_N$, that plays the role of the complexification of $\iota(\Phot_{\mathsf{std}})$. We describe this intersection in Lemma~\ref{lem_inter_singleton} below and actually prove that it is still a singleton. To this end, we work in a complexification of the representation $(V, \rho)$.

Let $V^{\Cf} := V \otimes \Cf$ be the complexification of $V$. Recall the notation of Section \ref{sect_cross_ratio}.

For any $\xi \in \SB(G)$ and any lift $f \in V^* \smallsetminus \{ 0 \}$ of $\iota^* (\xi)$, the map $f$ extends uniquely to a linear form $f^{\Cf} : V^\Cf \rightarrow \Cf$. We denote by $\iota^*(\xi)^\Cf$ the class $[f^\Cf]_c$ of $f^{\Cf}$ in $\mathbb{P}_{c}\left((V^\Cf)^*\right)$. This definition does not depend on the choice of the lift $f$ of $\iota^*(\xi)$ in $V^* \smallsetminus \{0\}$. As for the real case, we identify $\mathbb{P}_c\left((V^\Cf)^*\right)$ with the set of projective hyperplanes of~$\mathbb{P}_c(V^{\Cf})$. 

Similarly, for any $g \in G$, the operator $\rho (g)$ uniquely extends to an automorphism of~$V^{\Cf}$. We will still denote by $\rho(g)$ this extension.  

The map $\plong : \SL_2(\mathbb{R}) \rightarrow G $ induces a group homomorphism $\SL_2(\mathbb{R}) \rightarrow G \rightarrow \GL(V)$ with kernel included in $\{\pm \operatorname{id}\}$. It extends to a homomorphism $\plong_{\mathbb{C}}: \SL_2(\Cf) \rightarrow \GL(V^{\mathbb{C}})$ (with kernel $\{\pm \operatorname{id}\}$), such that

$$\plong_{\mathbb{C}}\left(e^{\E} \right) = e^{\rho_*(v^+)}, \quad \plong_{\mathbb{C}}\left(e^{\F} \right) = e^{\rho_*(v^-)}, \quad \plong_{\mathbb{C}}\left(e^{\He} \right) = e^{\rho_*(h_r)}.$$ 

The stabilizer of $e_1$ in $\SL_2(\mathbb{C})$ is the standard Borel subgroup $P_1^{\Cf}$ of $\SL_2(\mathbb{C})$. Hence the map $\plong_\Cf$ induces a $\plong_\Cf$-equivariant embedding $\FC : \mathbb{P}_c(\Cf^2) \hookrightarrow \mathbb{P}_c(V^{\Cf})$. The image of~$\FC$ is denoted by $\Phoc$. Explicitly, the set $\Phoc$ is the closure in $\mathbb{P}_c(V^{\Cf})$ of
\begin{equation*}
   \big{\{}\FC\left([1:z]_c\right) = e^{z \rho_*(v^-)} \cdot [e_1]_c \ \ \big{\vert}  \ z \in \mathbb{C}\big{\}}.
\end{equation*}

\begin{lem}\label{lem_inter_singleton} \begin{enumerate}
    \item One has $\rho(U^+) \cdot \Phoc = \Phoc$.
    \item Let $\xi \in \Contlam$. Then the set $\Phoc \cap \iota^*(\xi)^{\Cf}$ is a singleton.
\end{enumerate}
\end{lem}

\begin{proof}[Proof]
Let us first prove (1). Note that this is not just a consequence of Lemma~\ref{lem_stab_u_std}, although it has a similar proof. As in the proof of Lemma~\ref{lem_action_U_on_A}, for all $Y \in \uu^+$, using Lemma~\ref{lem_auxiliary_lemma}, we can find $\lambda \in \mathbb{R}$ such that for all $z \in \mathbb{C} \smallsetminus \{-\lambda^{-1}\}$ (with $-\lambda^{-1} = \infty$ if~$\lambda = 0$), one has 
\begin{equation*}
        \rho\left(\exp(Y)\right) \cdot \FC\left([1:z]_c\right) 
        = \FC\left(\left[1:\frac{z}{1 + z \lambda}\right]_c \right) \in \Phoc.
\end{equation*}
Hence $ \rho\left(\exp(Y)\right) \cdot \FC\left([1:z]_c\right) \in \Phoc$ for all $z \in \mathbb{C} \smallsetminus \{ -\lambda^{-1}\}$. Taking the closure, we get $\rho\left(\exp(Y)\right) \cdot \Phoc \subset \Phoc$. The converse inclusion also holds by the same argument applied to $-Y$ instead of $Y$. Therefore $\rho\left(\exp(Y)\right) \cdot \Phoc = \Phoc$.

Now let us prove (2).

\noindent \textit{Step 1.} Let us first prove (2) for $\xi = P^- \in \Contlam$. Let $f_0 \in V^* \smallsetminus \{ 0 \}$ be any lift of $\iota^*(P^\opp)$ and let $x \in \iota^*(P^-)^{\Cf} \cap \Phoc$.

There exists $(z_n) \in \Cf^{\mathbb{N}}$ be such that $\FC \left( e^{z_n \F} \cdot [1:0]_c\right) \rightarrow x$ as $n \rightarrow + \infty$. Since
$$f_0^{\Cf}\big{(}e^{z \rho_*(v^-)}e_1\big{)} =  f_0(e_1) + z f_0\left(\rho_*(v^-)e_1\right) + \cdots + \frac{z^N}{N!} f_0\left(\rho_*(v^-)^N e_1\right) = f_0(e_1) \ne 0$$
for all $z \in \Cf$, we must have $|z_n| \rightarrow + \infty$. Then $e^{z_n \F} \cdot[1:0]_c  \rightarrow [0:1]_c$. Hence $x$ has to be equal to $\FC \left([0:1]_c\right)$.

\noindent \textit{Step 2.} Now let $\xi \in \Contlam$ be any point, and let $g \in G$ be such that $\xi = g^{-1} P^\opp$. Since~$\Phot_{\mathsf{std}} \not\subset \hyp_{\xi}$, there exists $t \in \mathbb{R}$ such that $g^{-1}\exp(tv^-)P \in \Affstd$. Since $\exp(tv^-)$ preserves $P^\opp$, one has $g^{-1}\exp(tv^-) P^\opp = \xi$. Hence, up to replacing $g$ with $g^{-1}\exp(tv^-)$, we may assume that $g P \in \Affstd$. Then, by~\eqref{eq_A_egal_exp}, we can write $g = h \exp(Y)$ with~$Y \in \uu^+$ and~$h \in P^-$. By point (1) of the lemma, one has $\rho\left(\exp(Y)\right) \cdot \Phoc = \Phoc$. On the other hand, since $\rho(h^{-1})$ preserves $\iota^*(P^-)$, its $\Cf$-extension preserves $\iota^*(P^-)^\Cf$. Thus~$\rho(h^{-1}) \cdot \iota^*(P^-)^{\Cf} = \iota^*(P^-)^{\Cf}$. This gives
\begin{align*}
     \Phoc\cap  \iota^*(\xi)^\Cf 
     &= \Phoc \cap  \left(\rho\left(\exp(-Y)\right) \rho\left(h^{-1}\right)\cdot\iota^*(P^\opp)^{\Cf} \right)\\
     &=  \rho\left(\exp(-Y)\right)\left(\rho(\exp(Y))\Phoc \cap   \iota^*(P^\opp)^{\Cf}\right) \\
     & = \rho\left(\exp(-Y)\right)\left(\Phoc \cap   \iota^*(P^\opp)^{\Cf}\right),
\end{align*}
the last equality holding by point (1). Then, by Step 2, the set~$\Phoc\cap  \iota^*(\xi)^{\Cf}$ is a singleton.
$\qed$
\end{proof}
Lemma~\ref{lem_inter_singleton}.(2) above admits the following corollary:

\begin{cor}\label{cor_inter_split}
Let $\xi \in \Contlam$. Then the intersection polynomial $\pol_{\xi}^\rho$ of $\xi$ is split. If moreover $\Phot_{\mathsf{std}} \cap \hyp_{\xi} \subset \Affstd$, then the unique complex root of $\pol_{\xi}^\rho$ is equal to the unique~$t \in \mathbb{R}$ satisfying $\isl([1:t]) \in \hyp_{\xi}$.
\end{cor}

\begin{proof}[Proof] With the notation of Section~\ref{sect_inter_poly}, let $f \in V^*$ be the unique lift of $\xi$ such that~$f(\rho_*(v^-)^{n(\xi)} e_1) = n(\xi)!$. For all $z \in \mathbb{C}$, one has:
\begin{equation*}
    \mathsf{Q}_{\xi}^\rho(z) = f(e_1) + z f(\rho_*(v^-)e_1) + \cdots + z^{n(\xi)} = f^\Cf\big{(}e^{z \rho_*(v^-)} \cdot e_1 \big{)} 
\end{equation*}
Hence one has:
\begin{equation}\label{eq_equiv_roots}
    \mathsf{Q}_{\xi}^\rho(z) = 0  \ \Longleftrightarrow \ f^\Cf\big{(}e^{z \rho_*(v^-)} \cdot e_1 \big{)} = 0  \ \Longleftrightarrow \  \FC\left([1:z]_c\right) \in \Phoc \cap \iota^*(\xi)^\Cf.
\end{equation}
By Lemma~\ref{lem_inter_singleton}.(2), the intersection $\Phoc \cap \iota^*(\xi)^\Cf$ is a singleton, so the injectivity of $\FC$ and the equivalence of~\eqref{eq_equiv_roots} above give that $\mathsf{Q}_{\xi}^\rho$ has only one complex root. Therefore, the polynomial $\mathsf{Q}_{\xi}^\rho$ is split.

If moreover $\Phot_{\mathsf{std}} \cap \hyp_{\xi} \subset \Affstd$, then there exists $t \in \mathbb{R}$ such that $\isl([1:t]) \in \hyp_{\xi}$. Then~$\pol_{\xi}^\rho (t) = 0$, so $t \in \mathbb{R}$ is the unique complex root of $\pol_{\xi}^\rho$. $\qed$
\end{proof}

Recall that, by Lemma~\ref{lem_photons_affine}, given a point $\xi \in  \SB(G) \smallsetminus \Phot_{\mathsf{std}} = \Contlam$, the set~$\hyp_{\xi} \cap \Phot_{\mathsf{std}}$ is a singleton. We can then define a $G$-equivariant projection $\pstd: \SB(G) \rightarrow \Phot_{\mathsf{std}}$ by setting $\pstd(\xi) := \xi$ if $\xi \in \Phot_{\mathsf{std}}$, and $\pstd(\xi) := p$ where $\{p\} = \hyp_{\xi} \cap \Phot_{\mathsf{std}}$ if $\xi \in \Contlam$. Using Lemma~\ref{lem_inter_singleton} and Corollary~\ref{cor_inter_split}, we now establish a comparison between two cross ratios, involving the projection $\pstd$: 
\begin{lem}\label{lem_comp_crossratios}
Let $\xi_1, \xi_2 \in \SB(G)$, and for $i \in \{1,2\}$, let $b_i \in \mathbb{P}(\mathbb{R}^2)$ be such \mbox{  } that~$\pstd(\xi_i) = \isl(b_i)$. Let $a_1, a_2 \in \mathbb{P}(\mathbb{R}^2)$ be such that $b_1, a_1, a_2, b_2 \in \mathbb{P}(\mathbb{R}^2)$ are aligned in this order. Then one has
    \begin{equation*}
       \log \left|[\xi_1 : \isl(a_1) : \isl(a_2) : \xi_2]_{\rho} \right| = N \log (b_1: a_1: a_2: b_2).
    \end{equation*}
\end{lem}

Given two points $x,y \in \Phot_{\mathsf{std}}$, 
Lemma~\ref{lem_comp_crossratios} expresses that the cross ratio $[\xi_1,x,y,\xi_2]_{\rho}$, where $\xi_1, \xi_2 \in \SB(G)$, depends only on the projections of $\xi_1$ and $\xi_2$ to $\Phot_{\mathsf{std}}$. 

\begin{rmk}After this paper was completed, we learnt that Beyrer--Guichard--Labourie--Pozzetti--Wienhard \cite{beyrer2024positivity} prove a similar property for flag manifolds with a $\Theta$-positive structure. Other versions are established in \cite[Lem.\ 10.4]{zimmer2015rigidity2} for Grassmannians $\operatorname{Gr}_p (\mathbb{R}^n)$ and in \cite[Lem.\ 2.9]{lee2017collar} for the full flag manifold of $\SL(n, \mathbb{R})$.
\end{rmk}

\begin{proof}[Proof of Lemma~\ref{lem_comp_crossratios}] Let us set $x := P = \isl(a_1)$, $y := \isl(a_2)$, $p_1 = \isl(b_1)$ and~$p_2 = \isl(b_2)$. We may assume that $a_1 = [1:0]$ and $a_2= [1:1]$, and that there exist $t_1, t_2 \in \mathbb{R}$ such that $b_i = [1: t_i]$ for $i \in \{ 1, 2\}$, and that $t_1 < 0 <1< t_2$, the lemma then following by a continuity argument. Then the four distinct points~$p_1, x, y, p_2$ are aligned on $\Phot_{\mathsf{std}} \cap \Affstd$ in this order. Note that $\iota(x) = [e_1]$. Let $e_2 :=e_1 + \rho_*(v^-)e_1 + \cdots +  (1/N!) \rho_*(v^-)^N e_1$. Then, by~\eqref{forme_de_la_ddr1_2}, one has $\iota(y) = [e_2]$.

Recall the open set $\mathcal{A}_{\rho}$ of~\eqref{eq_A_rho}. By Lemma~\ref{lem_A_dense}, the set $\mathcal{A}_{\rho}$ is a dense open subset of~$\SB(G)$. Hence for $i \in \{1, 2 \}$, we can find a sequence $(\xi_{i,n}) \in \mathcal{A}_{\rho}^{\mathbb{N}}$ such that $\xi_{i,n} \rightarrow \xi_i$. For all $n \in \mathbb{N}$, let $p_{i,n} := \pstd(\xi_{i,n})$. Then, by continuity of $\pstd$, one has $p_{i,n} \rightarrow p_i \in \Affstd$, so up to extracting we may assume that $p_{i,n} \in \Affstd$ for all $n \in \mathbb{N}$.

Let $f_i \in V^* \smallsetminus \{0\}$ (resp.\ $f_{i,n} \in V^* \smallsetminus \{0\}$) be a lift of $\iota^*(\xi_i)$ (resp.\  $\iota^*(\xi_{i,n})$). For every~$n \in \mathbb{N}$ we choose $f_{i,n}$ such that $f_{i,n}(\rho_*(v^-)^N e_1) = N!$. For any $n \in \mathbb{N}$ and $i \in \{ 1, 2 \}$, the intersection polynomial
\begin{equation}\label{eq_def_Q}
    \pol_{i,n}(z) := \pol_{\xi_{i,n}}^\rho(z) = 
    (f_{i,n})^{\Cf}(e^{z \rho_*(v^-)}e_1) = f_{i,n}(e_1) + z f_{i,n}(\rho_*(v^-)e_1) + \cdots + z^N 
\end{equation}
is nonzero, so there exists $t \in \mathbb{R}$ such that $\pol_{i,n}(t) = f_{i,n}(e^{t \rho_*v^-} e_1) \ne 0$. This implies in particular that $\Phot_{\mathsf{std}} \not\subset \hyp_{\xi_{i,n}}$, i.e.\ $\xi \in \Contlam$. By Corollary~\ref{cor_inter_split}, the polynomial $\pol_{i,n}$ is thus split. But we also know that $\Phot_{\mathsf{std}} \cap \hyp_{\xi_{i,n}} = \{p_{i,n}\}$ is contained in $\Affstd$. Hence by the ``moreover" part of Corollary~\ref{cor_inter_split}, and since $n(\xi_{i,n}) = N$, the polynomial $\pol_{i,n}$ can be written $\pol_{i,n} (z)= (z-t_{i,n})^N$, with $t_{i,n} \in \mathbb{R}$ satisfying $\isl ([1:t_{i,n}]) = p_{i,n}$. Since $(p_{i,n})$ converges to $p_i$, the sequence $(t_{i,n})$ converges to $t_i$. One then has:
 \begin{align*}
    \log (s_1: [1:0] : [1:1] : s_2) &= \log \left| \frac{t_1 \cdot (t_2-1) }{t_2 \cdot (t_1-1)} \right| \\ 
    &= \lim_{n \rightarrow + \infty}  \log \left| \frac{t_{1,n} \cdot (t_{2,n}-1) }{ t_{2,n} \cdot (t_{1,n}-1)} \right| \\ 
     &= \frac{1}{N}\lim_{n \rightarrow + \infty}  \log \left| \frac{\pol_{1,n}(0) \pol_{2,n}(1) }{\pol_{2,n}(0) \pol_{1,n}(1)} \right| \\
     &= \frac{1}{N}\lim_{n \rightarrow + \infty}  \log \left| \frac{f_{1,n}(e_1) f_{2,n}(e_2) }{f_{2,n}(e_1) f_{2,n}(e_2)} \right| \\
     & = \frac{1}{N}  \log \left| \frac{f_{1}(e_1) f_{2}(e_2) }{f_{2}(e_1) f_{2}(e_2)} \right| 
     = \frac{1}{N}\log \big{|} \left[\xi_1 : x : y : \xi_2 \right]_{\rho}\big{|}. \quad \qed
 \end{align*}

\end{proof}

\begin{proof}[Proof of Proposition~\ref{lem_key_lemma}] Up to translating $\O$ by an element of $G$, one may assume that $x,y \in \Phot_{\mathsf{std}}$. Then there exist $a_1, a_2 \in \mathbb{P}(\mathbb{R}^2)$ such that $x= \isl(a_1)$ and $ y = \isl(a_2)$. 

Recall that we denote by $I_{x,y}$ the connected component of $\Phot_{\mathsf{std}} \cap \O$ containing $x$ and~$y$. Let $p_1,p_2 \in \partial \O$ be the endpoints of $I_{x,y}$, such that $p_1, x,y, p_2$ are aligned on $\Phot_{\mathsf{std}}$ in this order. Then there exist $b_1, b_2 \in \mathbb{P}(\mathbb{R}^2)$ such that $b_1, a_1, a_2, b_2$ are aligned in this order and $p_1 = \isl(b_1)$, and $p_2 = \isl(b_2)$.

By dual convexity, for $i \in \{ 1,2 \}$ there exists $\xi_i\in \O^*$ such that $p_i \in \hyp_{\xi_i}$. Then, by Lemma~\ref{lem_comp_crossratios}, one has
 \begin{equation*}
     \ro_{\O}(x,y) =  \log \left| (b_1 : a_1 : a_2: b_2)\right|  = \frac{1}{N}\log \big{|} \left[\xi_1 : x : y : \xi_2 \right]_{\rho}\big{|}.
 \end{equation*}
By the definition of $C_{\O}$, this implies that $\ro_{\O}(x,y) \leq N^{-1} C_{\O}(x,y)$. 

For the converse inequality,  let $\eta_1, \eta_2 \in \O^*$ be such that $C_{\O}(x,y) = \log\left|[\eta_1 : x:y: \eta_2]_{\rho}\right|$. For $i \in \{ 1,2\}$, let $b_i' \in \mathbb{P}(\mathbb{R}^2)$ be such that $\isl(b_i') = \pstd(\eta_i)$. Then, again by Lemma~\ref{lem_comp_crossratios}:
\begin{equation*}
\big{\vert} \log \left| (b_1' : a_1 : a_2 : b_2') \right| \big{\vert}=  \frac{1}{N} \log \left| [\eta_1 : x: y: \eta_2]_{\rho} \right|.
\end{equation*}
Since $\eta_1, \eta_2 \in \O^*$, the two points $\isl(b_1'), \isl(b_2')$ are not contained in $I_{x,y}$. Thus one has
\begin{equation*}
    \big{\vert} \log \left| (b_1' : a_1 : a_2 : b_2') \right| \big{\vert} \leq \log \left| \left(b_1 : a_1 : a_2 : b_2 \right) \right| = \ro_{\O}(x,y).
\end{equation*}
 Hence one has $
    N^{-1}C_{\O}(x,y) \leq \ro_{\O}(x,y)$. We have proved that $N^{-1}C_{\O}(x,y) = \ro_{\O}(x,y)$.

Now let us prove that $\kob_{\O} \geq N^{-1}C_{\O}$. Let $x, y\in \O$ be any two points, and let~$(x_0, \cdots , x_M) \in \mathcal{C}_{x,y}(\O)$. Then one has
\begin{equation}\label{eq_comp_kob_carat}
\begin{split}
    \sum_i \ro_\O(x_i, x_{i+1}) = 
    \frac{1}{N}\sum_i C_{\O}(x_i, x_{i+1}) &\geq \sum_i \frac{1}{N} \sup_{\xi_1^i, \xi_2^i \in \O^*}\log[\xi_1^i : x_i: x_{i+1}: \xi_2^i]_{\rho} \\
    &\geq  \frac{1}{N} \sup_{\xi_1, \xi_2 \in \O^*} \sum_i \log[\xi_1 : x_i: x_{i+1}: \xi_2]_{\rho} \\  &\geq  \frac{1}{N} \sup_{\xi_1, \xi_2 \in \O^*}  \log[\xi_1 : x: y: \xi_2]_{\rho}   \\
    & = \frac{1}{N}C_{\O}(x,y).
\end{split}
\end{equation}
Since this is true for all $(x_0, \cdots, x_{M}) \in \mathcal{C}_{x,y}(\O)$, by taking the infimum we get the inequality~$\kob_{\O}(x,y) \geq N^{-1} C_{\O}(x,y)$. \\
\indent Now let us show that $\kob_{\O}$ is a length metric. In~\eqref{eq_comp_kob_carat}, take $x$ and $y$ to be two conjugate points. The fact that $C_{\O}(x,y) = N \ro_\O(x,y)$ and Equation~\eqref{eq_comp_kob_carat} imply that the segment~$[x,y]$ has $K_{\O}$-length $\ro_{\O}(x,y)$. Hence the $K_{\O}$-length of $\gamma = (x_0, \cdots, x_M) \in \mathcal{C}_{x,y}(\O)$ is 
\begin{equation}\label{eq_longueur_chaine}
    \leng(\gamma) = \sum_i \kob_{\O}(x_i, x_{i+1}).
\end{equation}
Then one has $\kob_{\O}(x,y) = \inf \left\{ \leng(\gamma) \mid \gamma \in \mathcal{C}_{x,y}(\O) \right\}$.

Now let $\mathcal{C}_{x,y}'(\O)$ the set of all rectifiable curves joining $x$ and $y$ in $\O$. By the definition of the length of a curve, one has $\kob_{\O}(x,y) \leq \inf \left\{ \leng(\gamma) \mid \gamma \in \mathcal{C}_{x,y}'(\O) \right\}$. Since chains are rectifiable (for the identification with continuous paths, see Section~\ref{section_chains_length}), this last inequality is an equality. Hence $\kob_{\O}$ is a length metric.
    $\qed$
\end{proof}

\subsection{Properness}\label{sect_prop_K} In this section we state a corollary of Proposition~\ref{lem_key_lemma}, whixh is the properness of the Kobayashi metric on a proper dually convex domain of $\SB(G)$. This fact is of independant interest and will not be used in the rest of the paper.

Let us fix a real finite-dimensional real irreducible linear representation $(V, \rho)$ of $G$ with highest weight $N \omega_r$ for some $N \in \mathbb{N}_{>0}$. Let $\O$ be a proper domain of $\SB(G)$, and let $C_{\O}$ be the Caratheodory metric on $\O$ induced by~$(V, \rho)$. In \cite[Thm 9.1]{zimmer2018proper}, Zimmer proves that the following three assertions are equivalent:
\begin{enumerate}
        \item $\O$ is dually convex;
        \item $C_{\O}$ is a proper metric;
        \item $C_{\O}$ is a complete metric.
\end{enumerate}
The equivalence $(1) \Leftrightarrow (3)$ is stated in \cite[Thm 9.11]{zimmer2018proper}, and the equivalence $(1) \Leftrightarrow (2)$ is a consequence of its proof.

By Proposition~\ref{lem_key_lemma}, one has $\kob_{\O} \geq N^{-1} C_{\O}$. Thus the Kobayashi metric $\kob_{\O}$ is also proper:

\begin{cor}\label{cor_kobayashi_geodesic}
    If $\O \subset \SB(G)$ is a proper dually convex domain, then $\kob_{\O}$ is a proper metric. In particular, if $\O \subset \SB(G)$ is a proper almost-homogeneous domain, then $\kob_{\O}$ is a proper metric.
\end{cor}

Proposition~\ref{lem_key_lemma} and Corollary~\ref{cor_kobayashi_geodesic} imply in particular that $\kob_{\O}$ is geodesic as soon as $\O$ is a proper almost-homogeneous (and even just dually convex) domain, although we will not use this fact.

\begin{rmk}
    The only consequence of Proposition~\ref{lem_key_lemma} and Corollary~\ref{cor_kobayashi_geodesic} we will use in the proof of Theorem~\ref{thm_main} is the fact that $\kob_{\O}$ is proper as soon as $\O$ is almost-homogeneous, and that the $K_{\O}$-length of a $1$-chain $(x,y) $ is in this case equal to $\ro_{\O}(x,y)$. Other approaches would have led to this conclusion: for instance, the theory of Euclidean Jordan algebras, or looking at explicit proximal representations of $G$ for each value of $\g$ in Table \ref{table_shilov_bnds}. However, the approach we choose makes it possible to compare the Kobayashi metric with the already existing Caratheodory metrics, and is generalizable to other flag manifolds in which Kobayashi constructions have already been made (see Remark~\ref{rmk_comp_kob_cara}).
\end{rmk}

\section{The dynamics at the boundary}\label{sect_dyn_bndr}

This section is divided into two parts. Section~\ref{sect_visuel_extr} is devoted to the definition and the analysis of a family of points of the boundary of a proper domain in $\SB(G)$, those of \emph{$\MRr$-extremal} points. In Section~\ref{sect_conical_extr}, we prove the existence of $\MRr$-extremal points satisfying a particular geometric property, namely \emph{strongly $\MRr$-extremal points}.

\subsection{$\MRr$-extremal points}\label{sect_visuel_extr}

In this section, following \cite{van2019rigidity}, we investigate the relation between the structure of the boundary of a proper almost-homogeneous domain and the dynamics of its automorphism group.

 In \cite{van2019rigidity}, a notion of $\MRr$-extremal point is defined using rank-one lines in the Grassmannians. A similar notion is introduced in \cite{chalumeau2024rigidity}, using the classical photons of the Einstein Universe. Here we define the analogous notion in the context of Shilov boundaries of $\HTT$ Lie groups:
\begin{definition}\label{def_pts_extremes}
    We say that a point $p \in \partial \O$ is \emph{$\MRr$-extremal} if for any photon $\Phot$ through~$p$, the relative interior of $\Phot \cap \partial \O$ in $\Phot$ does not contain $p$.
\end{definition}
 Following the notation of \cite{van2019rigidity}, we denote by $\Rr(\O)$ the set of $\MRr$-extremal points of $\O$. We will see in Section~\ref{sect_conical_extr} that this set is never empty.

Whenever $\O$ is almost-homogeneous, $\mathcal{R}$-extremal points satisfy a strong geometric property:  
\begin{thm}\label{lem_geom_prop_visuel_2}
    Assume that $\O$ is a proper almost-homogeneous domain of $\SB(G)$. Let $p \in \partial \O$ be an $\mathcal{R}$-extremal point. Then $\hyp_p \cap \O = \emptyset$.
\end{thm}

\begin{rmk}Theorem~\ref{lem_geom_prop_visuel_2} is specific to our context: given a point $p$ of an arbitrary flag manifold $\Fl_{\Theta}$, the set $\hyp_p$ is a subset of $\Fl_{\Theta}^\opp$, which cannot be $G$-equivariantly identified with $\Fl_{\Theta}$ in general. There are versions of Theorem~\ref{lem_geom_prop_visuel_2} for flag manifolds that are not self-opposite, but they express a weaker geometric property for~$\MRr$-extremal points (see e.g.\ \cite{galiay2025convex}). For instance, for the flag manifold $\mathbb{P}(\mathbb{R}^n)$ with $n \geq 2$, which is as far as possible from being self-opposite, this weaker version of Theorem~\ref{lem_geom_prop_visuel_2} expresses that any~$\MRr$-extremal point of $\partial \O$ is contained in a supporting projective hyperplane to $\O$, which is already a consequence of the convexity of $\O$.
\end{rmk}

 For the proof of Theorem~\ref{lem_geom_prop_visuel_2}, we follow the strategy of \cite[Thm 7.4]{van2019rigidity}. We will need the following definition:

\begin{definition}
Let $\O$ be a proper domain of $\SB(G)$ and $x,y \in \O$, and $N \in \mathbb{N}^*$. Let us define
$$\kob_{\O}^N (x,y) := \inf\left\{\leng(\gamma) \big\vert \, \gamma \in \mathcal{C}_{x,y}^{N}(\O)\right\}.$$
Recall that $\leng(\gamma)$ is the $K_{\O}$-length for the metric $\kob_{\O}$ of the path $\gamma$ (see Section~\ref{section_chains_length}). The quantity $\kob_{\O}^N (x,y)$ is finite if, and only if, the set $\mathcal{C}_{x,y}^{N}(\O)$ is nonempty.
\end{definition}
 The map $\kob_{\O}^N: \O \times \O \rightarrow \Rf \cup \{ \infty \}$ is $\Aut(\O)$-invariant. The sequence $(\kob_{\O}^N(x,y))_{N \in \mathbb{N}}$ is nonincreasing, eventually finite,  and one has $\kob_{\O}(x,y) = \lim_{N \rightarrow + \infty} \kob_{\O}^N (x,y)$.

\begin{lem}\label{lem_asymptotic_behav_extremal}
   Let $\O \subset \SB(G)$ be a proper dually convex domain. Let $p \in \partial \O$ be an $\MRr$-extremal point, and let $q \in \overline{\O}$. Let $(x_k), (y_k) \in \O^{\mathbb{N}}$ be two sequences such that $x_k \rightarrow p$ and $y_k \rightarrow q$, and such that there exist $N \in \mathbb{N}$ and $M>0$ such that $\kob_{\O}^N (x_k, y_k) \leq  M$ for all $k \in \mathbb{N}$. Then $p=q$.
\end{lem}
\begin{proof}[Proof]
    For any $k \in \mathbb{N}$, let $\gamma_k = (x_k^0 := x_k ,\dots, x_k^{N} := y_k) \in \mathcal{C}^{N}_{x_k,y_k}(\O)$ be such that 
    $$\sum_{i=0}^{N-1} \kob_{\O}(x_i^k, x_{i+1}^k) = \leng(\gamma_k) \leq \kob_{\O}^N (x_k, y_k) + 1 \leq M + 1,$$
    the first equality holding because of~\eqref{eq_longueur_chaine}.
    Then, one has $\kob_{\O}(x_k^i, x_k^{i+1}) \leq M + 1$ for all $0 \leq i \leq  N-1$. Hence one can assume that $N = 1$, and the lemma follows by induction.
    
    Let us then assume that $N=1$. For all $k$, the two points $x_k $ and $y_k$ lie in the same connected component of the intersection $I_k := \Phot_k \cap \O$ of a photon $\Phot_k$ with $\O$. Let~$a_k,b_k$ be the endpoints of $I_k$ such that $a_k, x_k, y_k, b_k$ are aligned in this order. If $g_k \in G$ is such that $\Phot_k = g_k \cdot \Phot_{\mathsf{std}}$ and if we define $\isl^k := \isl_{g_k^{-1}}$ (recall~\eqref{eq_param_photons_2}), then there exist~$r_k, s_k, t_k, u_k \in \mathbb{P}(\mathbb{R}^2)$, aligned in this order, such that 
    \begin{equation*}
            \isl^k(r_k) = a_k; \quad \isl^k(s_k) = x_k; \quad
            \isl^k(t_k) = y_k; \quad
            \isl^k(u_k) = b_k.
    \end{equation*}
Then Proposition~\ref{lem_key_lemma} implies that
    \begin{equation*}
       \log (r_k : s_k : t_k:u_k) = \ro_{\O}(x_k, y_k) =  \kob_{\O}(x_k, y_k) \leq M.
    \end{equation*}
Up to extracting, we may assume that there exist $a_{\infty}, b_{\infty} \in \partial \O$ such that $a_k \rightarrow a_{\infty}$ and $b_k \rightarrow b_{\infty}$ as $k \rightarrow + \infty$, and also that there exist $r,s,t,u \in \mathbb{P}(\mathbb{R}^2)$ such that~$(r_k,s_k, t_k, u_k) \rightarrow (r,s,t,u)$. For all $k \in \mathbb{N}$, the points $a_k, x_k, y_k, b_k$ lie on the same photon in this order, so $a_{\infty}, p,q, b_{\infty}$ lie on the same photon, in this order. Since $p$ is $\MRr$-extremal, it must be equal to either $a_{\infty}$ or $b_{\infty}$. If it is equal to $b_{\infty}$, then~$q \in [p, b_{\infty}]$ is automatically equal to $p$. If $p = a_{\infty}$, then we must have $s = r$. Since the sequence~$\left(\log (r_k : s_k : t_k:u_k) \right)$ is bounded, this implies that $s = t$. Hence $p = q$. $\qed$
\end{proof}

We can now prove an analogue of Fact~\ref{lem_vey}:

\begin{lem}\label{lem_analogue_vey}
    Let $\O$ be a proper almost-homogeneous domain. Then for all $p \in \Rr(\O)$ there exists $(g_n) \in \Aut(\O)^{\mathbb{N}}$ such that for every compact subset $\mathcal{K} \subset \O$, one has $g_n \cdot \mathcal{K} \rightarrow \{p\}$ for the Hausdorff topology.
\end{lem}

\begin{proof}
        Since $\Omega$ is almost-homogeneous we can find $x\in \Omega$ and some sequence $(g_k)\in\Aut(\O)^{\mathbb{N}}$ such that 
    $g_k\cdot x\underset{k\rightarrow+\infty}{\longrightarrow}p.$
    Now let $y \in \O$ and $N \in \mathbb{N}$ such that $\kob_{\O}^N(x,y) < + \infty$. Then, by $\Aut(\O)$-invariance of $\kob_{\O}^N$, one has 
\begin{equation*}
    \kob_{\O}^N(g_k \cdot x, g_k \cdot y) = \kob_{\O}^N(x, y) \quad \forall k \in \mathbb{N}.
\end{equation*}
Thus by Lemma~\ref{lem_asymptotic_behav_extremal}, we have $g_k \cdot y \rightarrow p$. This holds for all $y \in \O$.

Let $\mathcal{K} \subset \O$ be a compact subset. If the sequence~$g_k \cdot \mathcal{K}$ does not converge to $\{p\}$ for the Hausdorff topology, then there is a neighborhood~$\mathcal{V}$ of $p$ in $\SB(G)$ and a sequence $(y_k) \in \mathcal{K}^{\mathbb{N}}$ such that $g_k \cdot y_k \notin \mathcal{V}$ for all~$k \in \mathbb{N}$. Since $\mathcal{K}$ is a compact subset of $\O$ by Corollary~\ref{cor_kobayashi_geodesic}, up to extracting we may assume that there exists $y \in \mathcal{K}$ such that $y_k \rightarrow y$. Then $(g_k \cdot y)$ converges to~$p$. But~$\kob_{\O}(g_k \cdot y_k, g_k \cdot y) \rightarrow 0$, so $g_k \cdot y_k \rightarrow p$. But this is impossible, since we assumed that $y_k \notin \mathcal{V}$ for all $k$.

Hence $g_k \cdot \mathcal{K} \rightarrow \{ p\}$ for the Hausdorff topology. $\qed$
\end{proof}

We can now prove Theorem~\ref{lem_geom_prop_visuel_2}:

\begin{proof}[Proof of Theorem~\ref{lem_geom_prop_visuel_2}]
Applying Lemma~\ref{lem_analogue_vey} to a compact subset $\mathcal{K} \subset \O$ with nonempty interior, and using the KAK decomposition, up to extracting, we may assume that there exists $q \in \SB(G)$ such that $(g_k)$ is $(p,q)$-contracting, in the sense that 
$$g_{k|\SB(G) \smallsetminus \hyp_q} \rightarrow p,$$
uniformly on compact subsets of $\SB(G) \smallsetminus \hyp_q$ (see e.g.\ \cite[Appendix A]{weisman2022extended} and \cite[Prop.\ 4.16]{kapovich2017anosov}). Since $\O$ is proper, the dual $\O^*$ of $\O$ has nonempty interior (see Section~\ref{sect_dual_proper_domain}). Hence there exists $z \in \O^* \smallsetminus \hyp_q$. Then $g_k \cdot z \rightarrow p$. But $\O^*$ is $\Aut(\O)$-invariant, so $g_k \cdot z \in \O^*$ for all $k \in \mathbb{N}$. Hence $p \in \O^*$ and by the definition of the dual we must have $\hyp_p \cap \O = \emptyset$. $\qed$
\end{proof}

\subsection{Strongly  $\MRr$-extremal points}\label{sect_conical_extr}
Let $\Omega\subset \SB(G)$ be a domain which is proper in $\Affstd$. We say that a point $p\in\partial\Omega$ is \emph{strongly $\MRr$-extremal} if either $ \mathbf{C}^-(p) \cap \overline{\O} = \{p\}$ or $\mathbf{C}^+(p) \cap \overline{\O} = \{p\}$. 

In general there are less strongly  $\MRr$-extremal points than $\MRr$-extremal points. However, the next lemma shows that strongly  $\MRr$-extremal points always exist.
\begin{lem}\label{lem_existence_extremal}
    Let $\Omega$ be a domain which is proper in $\Affstd$. Then for any $x\in \Omega$ there exist two strongly $\MRr$-extremal points $p \in \J^-(x)$ and $q \in \J^+(x)$.
\end{lem}
\begin{proof}[Proof]
Up to translating $\O$ in $\Affstd$, we may assume that $x = P$. Since $c^0$ is a properly convex cone of $\uu^-$, there is a nonzero linear map $\psi$ of $\uu^-$ such that $\overline{c^0} \smallsetminus \{0\} $ is contained in $\{\psi >0\}$. Let $X \in \uu^-$ be the element of $\overline{\varphi_{\mathsf{std}}^{-1}(\O)} \cap (-\overline{c^0})$ such that $\psi(X)$ is minimal. Then $p := \exp(X)P $ lies in $\partial \O$. 

Let us show that $p$ is strongly $\MRr$-extremal. Let $y \in \mathbf{C}^- (p)$. Write $y = \exp(Y) P$ with $Y \in X- \partial c^0$. Then one has
\begin{equation}\label{egalssi}
 \psi(X-Y) \geq 0,\text{ with equality if, and only if, }y=p.
\end{equation}
If moreover $y \in \partial \O$, then $y \in \mathbf{J}^-(x) \cap \overline{\O}$, so $\psi(Y) \geq \psi(X)$. Then, by~\eqref{egalssi}, one has~$y = p$. Hence $p$ is strongly $\MRr$-extremal. $\qed$
\end{proof}

\begin{rmk} \label{rmk_strong_extr_pts}
\begin{enumerate}
\item Lemma~\ref{lem_existence_extremal} is a generalization of~\cite[Lem.\ 6.4]{chalumeau2024rigidity}, where the $\mathcal{R}$-extremal points are called \emph{photon-extremal}. 
    \item By Lemma~\ref{lem_photons_affine}, Remark~\ref{rmk_photons_included_in_cone} and the fact that $c^0$ is a properly convex cone in $\uu^-$, strongly $\MRr$-extremal points are always $\MRr$-extremal, but the converse is false in general. For instance, for $G = \SO(n,2)$, take $p,q \in \Affstd$ with $q \in \mathbf{I}^+(p)$. Then $\Diams(p,q)$ has exactly two strongly $\MRr$-extremal points, namely $p$ and $q$. The points of $\mathbf{C}^+(p) \cap \mathbf{C}^-(q)$ are $\MRr$-extremal but not strongly $\MRr$-extremal.
    \item Contrary to the notion of $\MRr$-extremality, that of strong $\MRr$-extremality is only defined for a domain $\O$ which is proper in $\Affstd$. It is not clear at first that this second notion is invariant under $\Aut(\O)$. Using Lemma~\ref{lem_loc_action_U+} it is actually possible to show that this is the case whenever $\O$ is dually convex, although we will not use this fact. We will only prove this invariance in the \emph{almost-homogeneous case}; see Lemma  \ref{lem_action_strongly_extremal_pts}.
\end{enumerate}
\end{rmk}

\section{End of the proof of Theorem~\ref{thm_main}} \label{sect_end_mainth}

In this section we finish the proof of Theorem~\ref{thm_main}. We take $\O \subset \SB(G)$ an almost-homogeneous domain, proper in $\Affstd$.

Let $x \in \O$ and let $p_0 \in \partial \O \cap \mathbf{J}^-(x)$ and $q_0 \in \partial \O \cap \mathbf{J}^+(x)$ be two strongly $\MRr$-extremal points of $\partial \O$ given by Lemma~\ref{lem_existence_extremal}. Then in particular $p_0, q_0 \in \Rr(\O)$ (Remark~\ref{rmk_strong_extr_pts}.(2)), so by Theorem~\ref{lem_geom_prop_visuel_2} 
\begin{equation}\label{eq_omega_inter_vide}
    \O \cap \hyp_{p_0} = \O \cap \hyp_{q_0} = \emptyset.
\end{equation}
By reflexivity, one has $x \in \mathbf{J}^+(p_0)$. By~\eqref{eq_omega_inter_vide}, we also know that $x \notin \hyp_{p_0}$. Then, by Fact~\ref{fact_cone_in_schubert_subvrty}, one has $x \notin \mathbf{C}(p_0)$, and hence $x \in \mathbf{I}^+(p_0)$. Similarly, one has $x \in \mathbf{I}^-(q_0)$. Hence $x \in \Diams(p_0, q_0)$. By connectedness of $\O$, we then have the inclusion
\begin{equation}\label{eq_omega_inclus_diam}
    \O \subset \Diams(p_0,q_0).
\end{equation}

The goal of the rest of this section is to prove the converse inclusion. First observe that $p_0$ and $q_0$ are characterized among $\MRr$-extremal points of $\partial \O$ by a geometric property:
\begin{lem}\label{lem_caract_strongly_extremal_points}
    Let $p \in \Rr(\O)$ be such that $\mathbf{I}^+(p) \cap \O \ne \emptyset$ (resp.\ $\mathbf{I}^-(p) \cap \O \ne \emptyset$). Then $p = p_0$ (resp.\ $p = q_0$).
\end{lem}
\begin{proof}[Proof]Let us prove the Lemma for $\mathbf{I}^+(p) \cap \O \ne \emptyset$, the proof being similar for proving~$\mathbf{I}^-(p) \cap \O \ne \emptyset$. Since $p $ is $\MRr$-extremal, by Theorem~\ref{lem_geom_prop_visuel_2}, one has $\hyp_p \cap \O = \emptyset$, so by connectedness, the set $\O$ is included in one of the connected components of $\Affstd \smallsetminus \hyp_p$. Since $\mathbf{I}^+(p) \cap \O \ne \emptyset$, one has $\O \subset \mathbf{I}^+(p)$. But then one has $p_0 \in \overline{\O} \subset \mathbf{J}^+(p)$. By~\eqref{eq_omega_inclus_diam}, we also have $p \in \overline{\O} \subset  \mathbf{J}^+(p_0)$. By antisymmetry this implies $p=p_0$. $\qed$
\end{proof}

Using Lemma~\ref{lem_loc_action_U+}, we can now prove:

\begin{prop}\label{lem_action_strongly_extremal_pts}
    Let $g \in \Aut(\O)$. Then $g \cdot p_0, g \cdot q_0 \in \{p_0, q_0\}$.
\end{prop}
\begin{proof}[Proof]
    Let us prove the proposition only for $p_0$, the case of $q_0$ being similar.
    Up to translating $\O$ in $\Affstd$, one can assume that $p_0 = P$. Since $gP = g \cdot p_0 \in \overline{\O} \subset \Affstd$, by Lemma~\ref{lem_loc_action_U+} there is a neighborhood $\mathcal{U}$ of $P$ such that $g \cdot (\mathcal{U} \cap \I^+(P)) \subset \I^\delta(gP)$ for some $\delta \in \{ -, +\}$. Since $P \in \partial \O$, there exists $z \in \mathcal{U} \cap \O$. Since $g \in \Aut(\O)$, one has~$g \cdot z \in \O$. Hence $g \cdot z \in \O \cap \mathbf{I}^{\delta}(g P) \ne \emptyset$. Then, by Lemma~\ref{lem_caract_strongly_extremal_points}, we must have~$g \cdot p_0 = p_0$ if $\delta = +$, and $g \cdot p_0 = q_0$ if $\delta = -$.  $\qed$
\end{proof}

We have shown in Proposition~\ref{lem_action_strongly_extremal_pts} that for any $g \in \Aut(\O)$, the element $g$ stabilizes the pair $\{ p_0, q_0 \}$. Then $\Aut(\O)$ preserves the set $\SB(G) \smallsetminus (\hyp_{p_0} \cup \hyp_{q_0})$, and hence permutes its connected component. Since $\O \subset \Diams(p_0, q_0)$ is $\Aut(\O)$-invariant, the group~$\Aut(\O)$ preserves the connected component $\Diams(p_0, q_0)$ of $\SB(G) \smallsetminus (\hyp_{p_0} \cup \hyp_{q_0})$. Then:
\begin{equation*}\label{eq_auto_inclus}
\Aut(\O) \leq \Aut(\Diams(p_0, q_0)).    
\end{equation*}
Since $\Diams(p_0, q_0)$ is proper, Lemma~\ref{lem_inclusion_equality_almost_homogeneous} implies that $\O = \Diams(p_0, q_0)$. This concludes the proof of Theorem~\ref{thm_main}.

\section{Application: closed manifolds with proper development}\label{section_cpt_manifolds_proper_dvt} In this section we prove Corollary~\ref{main_cor}. A manifold $M$ is a \emph{$(G, \SB(G))$-manifold}, if there exists a (maximal) atlas of charts $(U, \psi_{U})_{U \in \mathcal{A}}$ on $M$ with values in $\SB(G)$, such that for any $U,V \in \mathcal{A}$ with $U \cap V \ne \emptyset$, the map $\psi_V \circ \psi_U^{-1}$ is the restriction of an element of $G$ to $\varphi_U(U \cap V)$. In this case there exists a map $\dev:\widetilde{M}\rightarrow\SB(G)$, called the \emph{developing map} (unique up to postcomposition by elements of $G$), where $\widetilde{M}$ is the universal cover of $M$. We say that $M$ is \emph{proper} if $\dev(\widetilde{M})$ is a proper domain of $\SB(G)$. This property is independent of the the choice of the developing map for $M$.

\begin{proof}[Proof of Corollary~\ref{main_cor}] The corollary is straightforward if $G$ is of real rank 1, hence we may assume that the real rank of $G$ is $r \geq 2$.

Let $\pi_1(M)$ be the fundamental group of $M$ and let $\hol: \pi_1(M) \rightarrow G$ be the holonomy of $M$, so that the map $\dev$ is $\hol$-equivariant. Let $\O=\dev(\widetilde{M})$. Since $M$ is closed, there is a compact fundamental domain $\mathcal{K}\subset\widetilde{M}$ intersecting every $\pi_1(M)$-orbit. By~$\hol$-equivariance of $\dev$, the compact set $\dev(\mathcal{K})\subset\O$ intersects any $\hol(\pi_1(M))$-orbit. Hence $\O$ is almost-homogeneous and thus almost-homogeneous. Theorem~\ref{thm_main} implies that $\O$ is a diamond and is~$\left(H_1, H_2\right)$-equivariantly diffeomorphic to the symmetric space $\mathbb{R} \times X_{L_{s}}$, where $H_1, H_2$ are finite-index subgroupds of $\Aut(\O), \Iso \left(\Rf \times X_{L_s}\right)$ respectively. Up to taking a finite cover of $M$, we may assume that $H_1 = \Aut(\O)$ and $H_2 = \Iso(\Rf \times X_{L_s})$.

Let $g_\O$ be the invariant Riemannian metric of $\O$ equal to $g_{\mathbb{R}}\oplus g_{X_{L_{s}}}$ under the previous identification and let $g=\dev^*g_\O$. The metric $g$ is invariant under $\pi_1(M)$, so it defines a Riemannian metric on $M$. This metric must be complete since $M$ is closed, so $g$ is also complete. The map $\dev$ is a local isometry between complete Riemannian manifolds, so it is a covering map. Since $\O$ is simply connected, the covering map is a diffeomorphism onto its image and $M$ is a quotient of $\O$ by a cocompact lattice of~$\Aut(\O)$. 

Since $\O$ is equivariantly diffeomorphic to  
 $\Rf \times X_{L_s}$, the manifold $M$ is a quotient of~$\Rf \times X_{L_s}$ by a cocompact lattice of $\Iso(\Rf) \times L_s = (\Rf \rtimes (\mathbb{Z}/2\mathbb{Z})) \times L_s$. It is a classical fact that the cocompact lattices of a product $\mathbb{R} \times G$, where $G$ is a simple Lie group, are virtually products of a cocompact lattice of $\mathbb{R}$ by a cocompact lattice of $G$. This fact applied to $\mathbb{R} \times L_s$ completes the proof of the corollary. $\qed$
\end{proof}

\bibliographystyle{alpha}
\bibliography{bibliography}
\vspace{1cm}
\small\noindent \textsc{Institut des Hautes Etudes Scientifiques,
35 rte de Chartres, 91440 Bures-sur-Yvette, France.} \emph{email address:} \texttt{galiay@ihes.fr}

\end{document}